\documentclass[10pt,a4paper,reqno]{amsart}
\usepackage{amsfonts}
\usepackage{amsthm}
\usepackage{amsmath}
\usepackage{mathtools}
\usepackage{amscd}
\usepackage{upgreek}
\usepackage[OT1]{fontenc}
\usepackage[utf8]{inputenc}
\usepackage{t1enc}
\usepackage[mathscr]{eucal}
\usepackage{indentfirst}
\usepackage{graphicx}
\usepackage{graphics}
\usepackage{esint}
\usepackage{pict2e}
\usepackage{epic}
\usepackage[dvipsnames]{xcolor}
\usepackage{float}
\usepackage{MnSymbol}
\usepackage[nocompress]{cite}
\usepackage{multirow}
\usepackage{booktabs}
\usepackage{pdflscape}
\usepackage{shuffle} %the symbol shuffle product
\usepackage[table,xcdraw]{xcolor} %table background color
\usepackage{enumitem} %leftmargin = *
\usepackage{hyperref}
\usepackage{stmaryrd}
\usepackage[listings]{tcolorbox}
\tcbuselibrary{breakable}
\numberwithin{equation}{section}
\usepackage[margin=2.9cm]{geometry}
\usepackage{epstopdf}
 
\allowdisplaybreaks

\newmuskip\pFqmuskip
\newcommand*\pFq[6][8]{%
	\begingroup % only local assignments
	\pFqmuskip=#1mu\relax
	\mathchardef\normalcomma=\mathcode`,
	\mathcode`\,=\string"8000
	\begingroup\lccode`\~=`\,
	\lowercase{\endgroup\let~}\pFqcomma
	{}_{#2}F_{#3}{\left[\genfrac..{0pt}{}{#4}{#5};#6\right]}%
	\endgroup
}
\newcommand{\pFqcomma}{{\normalcomma}\mskip\pFqmuskip}

\theoremstyle{plain}
\newtheorem{theorem}{Theorem}[section]
\newtheorem{lemma}[theorem]{Lemma}
\newtheorem{corollary}[theorem]{Corollary}
\newtheorem{proposition}[theorem]{Proposition}

\theoremstyle{definition}
\newtheorem{definition}[theorem]{Definition}
\newtheorem{conjecture}[theorem]{Conjecture}
\newtheorem{observation}[theorem]{Observation}
\newtheorem{remark}[theorem]{Remark}

\newtheorem{example}[theorem]{Example}

\newcommand{\lcmD}[1]{ \left[1,\cdots,#1\right] }

\renewcommand\div{\operatorname{div}}
\newcommand\M{\mathcal{M}}
\newcommand\E{\mathcal{E}}
\renewcommand\S{\mathcal{S}}
\renewcommand\Im{\operatorname{Im}}
\renewcommand\Re{\operatorname{Re}}
\newcommand\SL{\operatorname{SL}}
\newcommand\den{\operatorname{den}}

\newcommand\Ap{\operatorname{Ap}}
\newcommand\Bp{\operatorname{Bp}}

\newcommand\boldL{\boldsymbol{\mathrm{L}}}

\newcommand\SmallMatrix[1]{{%
		\tiny\arraycolsep=0.3\arraycolsep\ensuremath{\begin{pmatrix}#1\end{pmatrix}}}}

\begin{document}
	
	\title{Holonomic sequences arising from meromorphic modular forms}
	
	\author[Kam Cheong Au]{Kam Cheong Au}
	
	\address{University of Cologne \\ Department of Mathematics and Computer Science \\ Weyertal 86-90, 50931 Cologne, Germany} 
	
	\email{kau@uni-koeln.de}
	

	\begin{abstract} 
	We introduce a systematic framework for constructing holonomic recurrences from meromorphic modular forms. Its central ingredient is a modular analogue of creative telescoping, which turns finite-dimensional quotients of spaces of meromorphic modular forms into explicit recurrence operators. This provides a common explanation for several previously isolated examples: it recovers Ap\'ery's recurrence, Beukers' modular constructions and Zagier's seven sporadic recurrences. The framework also produces a large collection of new examples across different weights and groups. The modular origin of these sequences also confers them with non-trivial arithmetic properties, such as denominator bounds, irrationality statements, and $p$-adic congruences. %These examples suggest that these recurrences form a broad class of arithmetic holonomic recurrences encompassing many phenomena previously studied separately.
	\end{abstract}
	
	\maketitle

	\section{Introduction}\label{sec:introduction}
	A sequence $\{A_n\}$ is called \textit{holonomic} \cite{paule2021holonomic, zeilberger1990holonomic} if it satisfies a recurrence with polynomial coefficients,
	$$p_l(n) A_{n+l} + \cdots + p_1(n) A_{n+1} + p_0(n) A_n = 0,\qquad p_j\in \mathbb{C}[n].$$
	One of the best-known examples is the sequence of Ap\'ery's numbers $\Ap_n$ defined by
	$$ (n+2)^3 \Ap_{n+2} -(3 + 2 n) (39 + 51 n + 17 n^2) \Ap_{n+1} + (n+1)^3 \Ap_n = 0, \qquad (\Ap_0,\Ap_1) = (1,5),$$
	used in the proof that $\zeta(3)$ is irrational \cite{zagier2018arithmetic, apery1979irrationalite}. \par
	Beukers \cite{beukers1987irrationality} gave an interpretation of $\Ap_n$ in terms of modular forms on $\Gamma_1(6)$. A goal of this article is to extend Beukers' interpretation to a much broader setting. Our framework systematically produces recurrences, which we call \textit{modulonomic}, with arithmetic properties resembling those of Ap\'ery's recurrence. We will see that this concept interacts with several other subjects, including classical holonomic sequences, irrationality questions, hypergeometric series, and $p$-adic congruences. \par
	We begin by stating our main theoretical results before turning to examples. Let $\Gamma$ be a subgroup of $\text{SL}_2(\mathbb{R})$ commensurable with $\text{SL}_2(\mathbb{Z})$. Let $S$ be a finite subset of the compactified Riemann surface $X(\Gamma) = \Gamma \backslash \mathbb{H}^\ast$ that contains all the cusps. Let $\M_k^S(\Gamma)$ be the space of meromorphic modular forms of weight $k\in \mathbb{Z}$ on $\Gamma$ whose poles are confined to $S$.\footnote{We tacitly assume that whenever we discuss odd weights, the underlying group $\Gamma$ does not contain $-I$.} Let $D$ be the differential operator $\frac{1}{2\pi i}\frac{d}{d\tau}$. Bol's identity says $D^{k-1}(\M_{2-k}^S(\Gamma)) \subset \M_k^S(\Gamma)$. Let $\E_k(\Gamma)$ and $\S_k(\Gamma)$ denote the spaces of Eisenstein series and cusp forms, respectively.
	\begin{theorem}[Dimension formula]\label{dim_theorem}
		For a positive integer $e$, define
		$$\mu_k(e):=\left\lfloor\frac{(2-k)(e-1)}{2e}\right\rfloor-\left\lfloor\frac{(2-k)(e-1)}{2e}-\frac{k-1}{e}\right\rfloor \in \mathbb{Z}_{\geq 0}.$$
		Then for $k\geq 2$, the quotient $$\frac{\M_k^S(\Gamma)}{D^{k-1}(\M_{2-k}^S(\Gamma))}$$ is finite-dimensional, with dimension
		\begin{equation}\label{aux_20_intro}\dim\E_k(\Gamma)+2\dim\S_k(\Gamma)+\sum_{x\in S\cap(\Gamma\backslash\mathbb{H})}\mu_k(e_x),\end{equation}
		where $e_x$ is the elliptic order of the point $x$, with $e_x=1$ if $x$ is not an elliptic point.
	\end{theorem}
	
	In the simplest case, where $S$ consists of all cusps and no interior points, this is a familiar conclusion \cite{bringmann2014special, bringmann2015cycle, guerzhoy2008hecke, bonisch2023modularity,golyshev2021interpolated, bringmann2013eichler, bringmann2017harmonic}. Here $\M_k^S(\Gamma)$ becomes the space of weakly holomorphic modular forms, which we denote by $\M_k^!(\Gamma)$. The decomposition \begin{equation}\label{aux_27}\frac{\M_k^!(\Gamma)}{D^{k-1}(\M_{2-k}^!(\Gamma))} = \E_k(\Gamma) \oplus \frac{\S_k^!(\Gamma)}{D^{k-1}(\M_{2-k}^!(\Gamma))},\end{equation}
	where $\S_k^!(\Gamma)$ consists of elements of $\M_k^!(\Gamma)$ whose constant term at each cusp is zero,
	%\footnote{we mention the anecdote that the $\S$ in $\S_k(\Gamma)$ and $\S_k^!(\Gamma)$ have different etymologies: the former comes form German "Spritze" but latter comes from English "second", since these forms behave like meromorphic differential form without residue (i.e. those of \textit{second kind}).} 
	then implies that the second summand has dimension $2\dim \S_k(\Gamma)$.\par
	
	Choose $\mathfrak{B}\subset \M_k^S(\Gamma)$ whose images form a basis of the quotient space. If $t\in \M_0^S(\Gamma)$ and $g\in \M_k^S(\Gamma)$, then $t^ng$ also belongs to $\M_k^S(\Gamma)$ for every integer $n\geq 0$. Consequently, we can expand $t^ng$ in terms of these basis elements:
		\begin{equation}\label{eq:modulonomic_decomposition_intro}t^n g = \sum_{f\in \mathfrak{B}} c_{f,n} f + D^{k-1}(h_n),\qquad c_{f,n}\in \mathbb{C},\quad h_n\in \M_{2-k}^S(\Gamma).\end{equation}
	The following is the principal new theoretical result of the paper. By a recurrence operator $\boldL$, we mean an element of $\mathbb{C}(n)[\textbf{S}_n]$, where $\textbf{S}_n$ is the shift operator that replaces $n$ by $n+1$.
	\begin{theorem} [Creative telescoping] \label{creative_telescoping_S_intro}
	Let $k\geq2$, $t\in\M_0^S(\Gamma)$, and $g\in \M_k^S(\Gamma)$. Assume that the divisor of $t$ is supported on $S$. Then there exist a linear recurrence operator $\boldL$ and an element $R_n \in \M_{2-k}^S(\Gamma)\otimes_\mathbb{C} \mathbb{C}(n)$ such that
	$$\boldL(t^n g) = D^{k-1}(t^n R_n).$$
	Moreover, $\boldL$ may be chosen to have order at most $$\dim \frac{\M_k^S(\Gamma)}{D^{k-1}(\M_{2-k}^S(\Gamma))} - \delta_{k,2}.$$
\end{theorem}

	There is also a version of this theorem with character; see Theorem \ref{creative_telescoping_char}. Theorem \ref{creative_telescoping_S_intro} closely parallels creative telescoping in the classical setting \cite{Koutschan09, AequalsB, wilf1990rational, wilf1992algorithmic}; see Remark \ref{CT_classical}. We also give an algorithm for computing $\boldL$ and $R_n$ explicitly in Section~\ref{sec:creative-computation}. \par
	
	Theorem \ref{creative_telescoping_S_intro} immediately implies that, for each $f\in\mathfrak{B}$, the sequence $\{c_{f,n}\}_{n\in \mathbb{N}}$ is holonomic: applying $\boldL$ to both sides of \eqref{eq:modulonomic_decomposition_intro} gives
	\begin{equation}\label{eq:modulonomic_recurrence_intro} D^{k-1}(t^n R_n) = \sum_{f\in \mathfrak{B}} \boldL(c_{f,n}) f + D^{k-1}( \boldL(h_n)).\end{equation}
	Because the images of $\mathfrak{B}$ form a basis of the quotient, $\boldL(c_{f,n})=0$. We call recurrence operators obtained from Theorem \ref{creative_telescoping_S_intro} \textit{modulonomic}, although in this article we use the term descriptively rather than give a formal definition. \par
	Our proofs of Theorems \ref{dim_theorem} and \ref{creative_telescoping_S_intro} are based on divisors of automorphic forms and the Riemann--Roch theorem. We draw heavily on the ideas in Shimura's standard text \cite[Chapter~2]{shimura1971introduction}. The key ingredient in creative telescoping is Proposition \ref{Dt^n_continuation}, which allows us to interpret an expression of the form
	$$\frac{D^{k-1}(t^n h)}{t^n}, \qquad t\in \M_0^S(\Gamma), h\in \M_{2-k}^S(\Gamma)$$
	as an element of $\M_k^S(\Gamma) \otimes_\mathbb{C} \mathbb{C}(n)$. The creative-telescoping theorem then follows by extending the classical Riemann--Roch theorem to the function field $\mathbb{C}(n)$; see Section~\ref{creative_telescope_section}. \par

	Throughout, let $\lcmD{x}$ denote the least common multiple of all positive integers at most $x$. We say that a sequence $c_n\in\overline{\mathbb{Q}}$ has denominator type $d_n$ if there exists a nonzero constant $A$, independent of $n$, such that $Ad_nc_n$ is an algebraic integer for every $n$. For a Laurent series $f\in \mathbb{C}((q))$, let $a_n(f)$ denote its coefficient of $q^n$.

	\subsection{An example on $\Gamma_1(6)$ and Ap\'ery's sequences}\label{sec:intro-gamma1-6}
	The best-known example of a modulonomic recurrence is Ap\'ery's recurrence. We give a more detailed account in Section~\ref{sec:gamma1-6-minus} and briefly recall the framework here. The group $\Gamma_1(6)$ has four cusps, $\{\infty,0,\frac12,\frac13\}$, and we consider the following element of $\M_0^!(\Gamma_1(6))$:
	\begin{equation}\label{t_choice_level6_intro}t(\tau) := \left(\frac{\eta(\tau)\eta(6\tau)}{\eta(2\tau) \eta(3\tau)}\right)^{12} =q-12q^2+66q^3-220q^4+495q^5+O(q^6),\end{equation}
	Its divisor is $(0)+(\infty)-(\frac12)-(\frac13)$. Let $g$ be
	%$$g(\tau) = -\frac{E_4(\tau)}{480}+\frac{7}{120}E_4(2\tau)-\frac{21}{160}E_4(3\tau)+\frac{3}{40}E_4(6\tau)$$
	the unique Eisenstein series, up to scaling, with Fricke eigenvalue $-1$ that vanishes at the cusp $\infty$. Since this eigenspace contains no cusp forms, Theorem \ref{dim_theorem} implies that, for every $n\geq0$,
	\begin{equation}\label{aux_11_intro}t^n g = \Ap_n g + D^3(h_n),\qquad \text{for some } \Ap_n\in \mathbb{Q},\quad h_n\in \M_{-2}^!(\Gamma_1(6)).\end{equation}
	We integrate both sides from $0$ to $i\infty$. The integral converges because $t$ is holomorphic at the endpoints, although the integrand has poles of increasing order at the other two cusps, $\frac12$ and $\frac13$. This gives
	$$(2\pi i)^3 \int_0^{i\infty} t(\tau)^n g(\tau) d\tau =  \Ap_n \zeta(3) + 12a_0(h_n) := \Ap_n \zeta(3) - \Bp_n.$$
	The integral of $g$ is evaluated using the $L$-values of Eisenstein series, while the integral of $D^3(h_n)$ follows from the fundamental theorem of calculus (Lemma \ref{IBP_lemma}). The left-hand side decays geometrically because the maximum value of $|t|$ on $i\mathbb{R}_+$ is $(\sqrt{2}-1)^4$. The sequences $\Ap_n$ and $\Bp_n$ are precisely Ap\'ery's sequences. One way to see this is to compute the creative telescoper from Theorem \ref{creative_telescoping_S_intro} and compare it with Ap\'ery's recurrence. Namely, $\boldL(t^n g)=D^3(t^nR)$, where
	$$\begin{aligned}\boldL &= (n+1)^3 -(3 + 2 n) (39 + 51 n + 17 n^2) \textbf{S}_n +(n+2)^3 \textbf{S}_n^2,\\ R&= -\frac{\eta(\tau)^{17}\eta(6\tau)^{17}}{2\eta(2\tau)^{19}\eta(3\tau)^{19}}.\end{aligned}$$
	
	The integrality of $\Ap_n$ and $\lcmD{n}^3\Bp_n$ follows naturally from \eqref{aux_11_intro}. A useful feature of our theory is that the denominator types of modulonomic sequences can be determined without first finding explicit formulas for them. By contrast, the classical approach usually requires an explicit expression. The same method applies in every weight and gives the following corollary; see Corollary \ref{Gamma1(6)_denomtype_cor}.
	
	\begin{corollary}\label{cor_denomtype_intro}
		Let $k\geq 3$, and let $\mathfrak{B}$ be a basis of $\M_k^!(\Gamma_1(6))$ modulo $D^{k-1}(\M_{2-k}^!(\Gamma_1(6)))$, chosen compatibly with the splitting \eqref{aux_27}. Let $t$ be as in \eqref{t_choice_level6_intro} and let $g\in \M_k(\Gamma_1(6)) \cap \overline{\mathbb{Q}}((q))$. Write
		$$t^n g = \sum_{f\in \mathfrak{B}} c_{f,n} f + D^{k-1}(h_n),\qquad h_n\in \M_{2-k}^!(\Gamma_1(6)).$$ Then
		\begin{itemize}[leftmargin=*]
			\item For any $f\in \mathfrak{B}\cap\E_k(\Gamma_1(6))$, $c_{f,n}$ has trivial denominator type.
			\item For any $f\in \mathfrak{B}\cap\S_k^!(\Gamma_1(6))$, $c_{f,n}$ has denominator type $\lcmD{n}^{k-1}$.
			\item $a_0(h_n)$ has denominator type $\lcmD{n}^{k-1}$.
		\end{itemize}
	\end{corollary}
	The proof of Corollary \ref{cor_denomtype_intro} analyzes the $q$-expansions of various weight-zero forms at the cusps. Section~\ref{sec:denominator-types} develops a broad generalization of this method for determining the denominator types of modulonomic sequences.
	
	The group $\Gamma_1(6)$ has a weight-four cusp form with Fricke eigenvalue $+1$. Let $\Gamma^+:=\langle\Gamma_1(6),W_6\rangle$, let $f = \eta(\tau)^2\eta(2\tau)^2\eta(3\tau)^2\eta(6\tau)^2$ be this cusp form, and fix a meromorphic partner $f^\sharp$ of $f$ that is analytic at $0$ and $\infty$. Let $E$ be the unique Eisenstein series on $\Gamma^+$ that vanishes at $0$ and $\infty$. The dimension formula \ref{dim_theorem} gives
	\begin{equation}\label{aux_14_intro}t^n f = A_n E + B_n f + C_n f^\sharp + D^3(h_n), \qquad A_n,B_n,C_n\in \mathbb{Q},\quad h_n \in \M_{-2}^!(\Gamma^+).\end{equation}
	With suitable normalizations (see Section~\ref{sec:gamma1-6-plus}), one obtains
	$$\begin{aligned}\frac{(2\pi i)^3}{2} \int_0^\infty t(\tau)^nf(\tau) d\tau &= 47\zeta(3)A_n+B_n\xi_1-C_n\eta_1 + \frac{78}{5}C_n-6a_0(h_n),\\
		\frac{(2\pi i)^3}{2}\int_0^\infty t(\tau)^nf(\tau) \tau d\tau &= \frac{2\pi^3}{3}A_n+B_n\xi_2-C_n\eta_2.
	\end{aligned}$$
	Here $\xi_1,\xi_2$ are periods of $f$, and $\eta_1,\eta_2$ are quasi-periods. Both left-hand sides decay geometrically with ratio $(\sqrt{2}-1)^4$. All four sequences have denominator type $\lcmD{n}^3$. Thus, the same inequality $(\sqrt{2}-1)^4 e^3<1$ as in Ap\'ery's argument yields an irrationality result; see Corollary \ref{cor_plus_level6_irrational}. Theorem \ref{creative_telescoping_S_intro} also implies that the four sequences $A_n$, $B_n$, $C_n$, and $a_0(h_n)$ satisfy the same fourth-order recurrence $\boldL$. \par
	
	We also give an example in Section \ref{CM_level7_sect} illustrating that, when the cusp form has CM, the denominators of the corresponding modulonomic sequences are supported on particular sets of primes rather than on the full $\lcmD{n}^{k-1}$; see Proposition \ref{-7_CM_cusp_form}.

	\subsection{Examples on other genus-zero groups}\label{sec:intro-other-groups}
	The technique used above to prove Corollary \ref{cor_denomtype_intro} also works in positive genus. When the group has genus zero and $t$ is a Hauptmodul, the denominator types can also be read from an explicit pairing between $t$ and an Eichler integral; see Proposition \ref{lemma:eichler_solution}. For example, Ap\'ery's sequences $\Ap_n$ and $\Bp_n$ admit the alternative expressions
	$$\Ap_n = a_0(t^{-n}G),\qquad \Bp_n = -12a_0\bigl(t^{-n}G\E_3(g)\bigr),$$
	where $g$ is normalized as in \eqref{aux_11_intro}, and
	$$G:=\Theta\frac{Dt}{t},\qquad \Theta:=\frac{\eta(2\tau)^7\eta(3\tau)^7}{\eta(\tau)^5\eta(6\tau)^5}.$$
	Here and below, for $f \in \mathbb{C}((q))$, we define
	$$\E_{k-1}(f) := \sum_{\substack{n\in \mathbb{Z} \\ n\neq 0}} \frac{a_n(f)}{n^{k-1}}q^n.$$
	These expressions are already implicit in Beukers' work \cite[Theorem~1]{beukers1987irrationality}. The denominator type $\lcmD{n}^3$ is then immediately explained by $\E_3(g)$.

	Together with the dimension bound from creative telescoping, this produces many second-order modulonomic recurrences and recovers several isolated examples from the literature. In particular, we recover Zagier's seven recurrences $\textbf{A}, \textbf{B}, \ldots, \textbf{G}$ \cite{zagier2009integral}, which played a pivotal role in a recent breakthrough by Calegari, Dimitrov and Tang \cite{CDT}. More precisely, we will see in Section~\ref{sec:zagier} that
	\begin{align}\label{aux_32}\text{Zagier's } &\textbf{A}, \textbf{B}, \cdots, \textbf{G} \text{ recurrences}\iff \text{ weight 3 modulonomic recurrences } \\ &\text{from genus-zero, torsion-free }\Gamma \text{ with exactly 4 cusps, all regular}. \nonumber \end{align}
	Other examples arise by relaxing some of the conditions on the right-hand side of \eqref{aux_32}. For instance, let $\mathcal{N}$ be the unique index-two subgroup of $\Gamma_0(6)$ for which
	$$\sqrt{\frac{\eta(2\tau)^7 \eta(3\tau)^7}{\eta(\tau)^5 \eta(6\tau)^5}} = 1+\frac{5q}{2}+\frac{27q^2}{8} + \frac{49q^3}{16}+\frac{147q^4}{128}+\cdots$$
	is a weight-one modular form. Its square belongs to $\M_2(\Gamma_0(6))$. The group $\mathcal{N}$ is a noncongruence subgroup satisfying all the conditions in \eqref{aux_32}, except that some of its cusps are irregular. A direct Ap\'ery-limit argument proves the irrationality of the period
	\begin{equation}\label{aux_33}(2\pi i)^2\int_0^{i\infty} \sqrt{\frac{\eta^9(\tau)\eta^9(6\tau)}{\eta^3(2\tau)\eta^3(3\tau)}}\tau \,d\tau \approx 0.400749797\end{equation}
	and the quasi-period
	$$\begin{aligned}
	(2\pi i)^2\int_0^{i\infty} \left( \sqrt{\frac{\eta(\tau)^{43}\eta(6\tau)^{23}}{\eta(2\tau)^{29}\eta(3\tau)^{25}}}
	-72\sqrt{\frac{\eta(\tau)^{23}\eta(6\tau)^{43}}{\eta(2\tau)^{25}\eta(3\tau)^{29}}} \right)\tau\,d\tau&\approx 0.05133879023,
	\end{aligned}$$ 
	of the corresponding meromorphic partner. Beukers had already proved the assertion about the period \eqref{aux_33} \cite{beukers1987irrationality}. In our framework, the arguments for the period and quasi-period proceed identically; see Examples \ref{example:beukers_weight_three} and \ref{example:beukers_weight_three_partner}. Both numbers are Ap\'ery limits of second-order recurrences. For example, \eqref{aux_33} corresponds to
	$$(2n+3)^2-2(68n^2+238n+209)\textbf{S}_n+4(n+2)^2\textbf{S}_n^2.$$
	One solution $A_n$ of this recurrence can be described as the square root of Ap\'ery's generating function:
	\begin{equation}\label{gen_square_root}\left(\sum_{n\ge0}A_nx^n\right)^2=
	\sum_{n\ge0}\Ap_n x^n,\qquad A_0 = 1.\end{equation}
	
	\subsection{Examples on $\Gamma_0(4)$ and moments of elliptic integrals}\label{sec:intro-gamma0-4}
	We next consider an example on $\Gamma_0(4)$, discussed in detail in Section~\ref{sec:gamma0-4}. This group has three cusps, and we choose the weight-zero modular function
	$$t(\tau) = \frac{256\eta^8(\tau)\eta^{16}(4\tau)}{\eta^{24}(2\tau)}=  256(q-8q^2+44q^3+\cdots) \sim (\infty)-(\frac12).$$
	Choose $g$ to be an Eisenstein series that vanishes at $0$ and $\infty$. Then for $n\geq 0$,
	\begin{equation}\label{eq:level4_moment_intro}(2\pi i)^3 \int_{0}^{i\infty} t(\tau)^n g(\tau) d\tau \in \mathbb{Q}\zeta(3) + \mathbb{Q}.\end{equation}
	This conclusion is known in the literature, although it is not usually stated in this form. More precisely, modular forms on $\Gamma_0(4)$ can be related to a hypergeometric function through the complete elliptic integral of the first kind,
	$$K(x) := \int_0^1 \frac{1}{\sqrt{(1-y^2)(1-xy^2)}}dy = \frac{\pi}{2} \pFq{2}{1}{\frac12,\frac12}{1}{x}.$$
	Under this substitution, \eqref{eq:level4_moment_intro} becomes
	$$\int_0^1 x^n K(x)^2 dx = 16^{-n}(2\pi i)^3 \int_{0}^{i\infty} t(\tau)^n g(\tau) d\tau.$$
	Wan first evaluated the left-hand side and related variants using hypergeometric techniques; these integrals were subsequently studied by several authors \cite{zhou2014legendre, Wanthesis, campbell2019interplay, wan2012moments, chanlegendre, wan2016integrals, rogers2015moments, wan2012moments}.
	
	Because these sequences are modulonomic, we can readily express \eqref{eq:level4_moment_intro} as $A_n\zeta(3)+B_n$, where $A_n$ has trivial denominator type and $B_n$ has denominator type $\lcmD{n}^3$. These facts are not immediate from the hypergeometric perspective, where one must again resort to explicit formulas. \par
	
	An advantage of $\Gamma_0(4)$ over $\Gamma_1(6)$ is that it has fewer cusp forms as the weight increases. This allows us to construct moments involving a single odd zeta value. For example, in weight six, a suitable Eisenstein series $g$ gives the following identity; see Section~\ref{sec:gamma0-4-weight6}:
	\begin{equation}\label{aux_34}(2\pi i)^5 \int_0^{i\infty} t(\tau)^n g(\tau)\,d\tau
	=16\int_0^1 [256x(1-x)]^n(2x-1)K(x)^4\,dx \in \mathbb{Q}\zeta(5)+\mathbb{Q}.\end{equation}
	This is a rare example of a family of moments involving only $\zeta(5)$, without contributions from lower zeta values. The same argument as in weight four shows that the two coefficients have denominator types $1$ and $\lcmD{n}^5$, respectively.
	\subsection{Congruence properties}\label{sec:intro-congruences}
	Modulonomic sequences also appear to have particularly strong $p$-adic properties. For example, Gessel proved \cite{gessel1982some} that Ap\'ery's sequence $\Ap_n$ in \eqref{aux_11_intro} satisfies, for every $p\geq5$ and $a,m\geq1$,
	$$\Ap_{p-1} \equiv 1 \pmod{p^3}, \qquad \Ap_{p^a m -1} \equiv \Ap_{p^{a-1} m -1} \pmod{p^{3a}}.$$
	Congruences of the same type seem to hold for modulonomic sequences whenever there is no cuspidal contribution. For example, Corollary \ref{congruence_level4_cor_2} shows that the coefficient of $\zeta(3)$ in \eqref{eq:level4_moment_intro} satisfies
	$$A_{p-1} \equiv 256^{p-1}\pmod{p^3}.$$
The coefficient of $\zeta(5)$ in \eqref{aux_34} satisfies
$$A_{p^am-1}\equiv 4096^{(p-1)p^{a-1}m}A_{p^{a-1}m-1}\pmod{p^{5a}},\qquad p\geq3,\quad a,m\geq1.$$
We do not pursue these congruences in depth, since the literature already contains many techniques that may be used to prove them \cite{beukers2024supercongruences, long2021supercongruences, allen2025atkin}. We do, however, present some experimental observations concerning Atkin--Swinnerton-Dyer-type congruences for modulonomic sequences arising from cusp forms.
\begin{observation}
Let $f = \eta(\tau)^2\eta(2\tau)^2\eta(3\tau)^2\eta(6\tau)^2$, let $t$ be as in \eqref{t_choice_level6_intro}, and set $U_n = a_0(t^{-n} f)$.\footnote{Up to a rational multiple, this is the same as the sequence $A_n$ in \eqref{aux_14_intro}.} We empirically observe that
$$U_{p^am}-a_p(f) U_{p^{a-1}m}+p^3U_{p^{a-2}m}\stackrel{?}{\equiv} 0\pmod{p^{2a}},\qquad p\geq 5,\quad a\geq2,\quad m\geq1.$$
\end{observation}
Another observation concerns a modulonomic sequence on the noncongruence subgroup $\mathcal{N}$ defined above. Let $v_p$ denote the $p$-adic valuation normalized by $v_p(p)=1$.
\begin{observation}
Let $A_n$ be defined by \eqref{gen_square_root}. For every prime $p$ that splits in $\mathbb{Q}(\sqrt{-6})$, we have experimentally
$$v_p(A_n) \stackrel{?}{\geq} 2v_p( \binom{2n}{n}).$$
\end{observation}
	
	\subsection{Meromorphic examples}\label{sec:intro-meromorphic}
	We briefly describe two examples of modulonomic sequences associated with meromorphic modular forms; see Section~\ref{sec:meromorphic}. Their arithmetic properties appear to differ substantially from those arising in the weakly holomorphic case and are connected with deeper modular objects. \par
	Both examples take $\Gamma=\SL_2(\mathbb{Z})$, $k=4$, and $S=\{\infty,\tau_D\}$, where $\tau_D$ is a CM point of discriminant $D$.
	
	For the first example, take $D=-3$, so that $\tau_D$ is an elliptic point of order three. Let $J=j/1728$. The dimension formula \ref{dim_theorem} implies that, for every $n\geq1$, we can write
	\begin{equation}\label{aux_17_intro}\frac{E_4}{J^n}=A_n\frac{E_4}{J}+D^3(h_n), \qquad A_n\in \mathbb{Q},\quad h_n\in \M_{-2}^S(\SL_2(\mathbb{Z})).\end{equation}
	The sequence $A_n$ satisfies a second-order recurrence whose other solution is $B_n:=a_0(h_n)$. Their denominator types appear much harder to prove. Numerical experiments suggest the following conjecture.
	\begin{conjecture}\label{LCM_bound_Ex1_intro}
		The denominator types of $A_n$ and $B_n$ are $$A_n:\quad 64^n\lcmD{3n-2},\qquad
		B_n:\quad 64^n\lcmD{3n-2}  \lcmD{n-1} \lcmD{\frac{3n-2}{4}}.$$
	\end{conjecture}
	In this case, an explicit expression for $A_n$ yields even sharper denominator estimates; see Proposition \ref{beta_exp_mero}. Integrating both sides also gives
	\begin{equation}\label{aux_18_intrdo} (2\pi i)^3 \int_0^{i\infty} \frac{E_4}{J^n} d\tau = A_n \left((2\pi i)^3 \int_0^{i\infty} \frac{E_4}{J} d\tau \right)-2B_n.\end{equation}
	A nontrivial evaluation, recalled in Remark \ref{E4/J_eval}, is
	\begin{equation*}(2\pi i)^3 \int_0^{i\infty} \frac{E_4}{J} d\tau = 270 L_{-3}(2).\end{equation*}
	Thus, \eqref{aux_18_intrdo} expresses the $L$-value $L_{-3}(2)$ in terms of moments of meromorphic modular forms with poles of arbitrarily high order. \par
	When $\tau_D$ is not an elliptic point, the dimension formula \ref{dim_theorem} shows that $\frac{\M_4^S}{D^3(\M_{-2}^S)}$ has dimension four. Let $j_D=j(\tau_D)$. One basis is
	$$\left\{ E_4, \frac{E_4}{j-j_D},\frac{E_4}{(j-j_D)^2},\frac{E_4}{(j-j_D)^3}\right\}.$$
	For reasons that we do not yet fully understand, this does not appear to be the natural basis for the arithmetic questions considered here. A more suitable basis is the one introduced by Zhang in the context of magnetic modular forms \cite[Example~4.1]{zhang2025elliptic}. Explicitly, let $\psi_D = \frac{E_2^\ast E_4}{E_6}(\tau_D) \in \mathbb{Q}(j_D)$ and
	{\small $$\begin{aligned}
	G_D^{(1)}&=\frac{E_4}{j-j_D},\\
	G_D^{(2)}&=\left(\frac{j_D}{2(j_D-1728)}+\frac13+\frac{\psi_D}{6}\right)\frac{E_4}{j-j_D}
	+j_D\frac{E_4}{(j-j_D)^2},\\
	 G_D^{(3)}&=\left[\left(\frac{61}{72}+\frac{\psi_D}{12}\right)\frac{j_D}{j_D-1728}
	+\frac{\psi_D}{18}+\frac{\psi_D^2}{72}\right]\frac{E_4}{j-j_D}+j_D\left(\frac{3j_D}{2(j_D-1728)}+\frac43+\frac{\psi_D}{6}\right)\frac{E_4}{(j-j_D)^2}
	+2j_D^2\frac{E_4}{(j-j_D)^3}.
	\end{aligned}$$}
	Consider the modulonomic sequences $A_{{1},n}, A_{{2},n}, A_{{3},n} \in \mathbb{Q}(j_D)$ determined by\footnote{When the class number of $D$ is greater than one, the values $j_D$ and $\psi_D$ are determined only up to Galois conjugacy. This does not affect the following conjecture.}
	$$\frac{E_4}{(j-j_D)^n} = \sum_{j=1}^3 A_{(j),n} G_D^{(j)} + D^3(h_n),\qquad h_n \in \M_{-2}^S(\SL_2(\mathbb{Z})).$$
	These three sequences satisfy a fourth-order linear recurrence whose remaining solution is $B_n:=a_0(h_n)$. Our numerical experiments suggest the following conjecture.
	\begin{conjecture}\label{conj_mero_D_intro}Let $D$ be a negative discriminant $\neq -3$ or $-4$. Write $s_D = j_D(j_D-1728)$. Let $X_{i,n}=s_D^n A_{(i),n}$ and $Y_n = s_D^n  B_n$. Then \\
		(a) The denominator types of these four sequences are \begin{equation}\label{conj_mero_denom_intro}\begin{aligned}&X_{2,n}:\quad \lcmD{n-1}\\ &X_{1,n} \text{ and }X_{3,n}:\qquad \lcmD{n-1}\lcmD{\frac{n-1}{2}}\\
				&Y_n:\qquad \lcmD{n-1}\lcmD{\frac{n-1}{2}} \lcmD{\frac{n-1}{3}} .\end{aligned}\end{equation}
		(b) Let $p$ be coprime to $s_D$. If $p\mid \mathrm{den}(X_{1,n})$, then $p$ is inert in $\mathbb{Q}(\sqrt{D})$; if $p\mid \mathrm{den}(X_{3,n})$, then either $p$ splits in $\mathbb{Q}(\sqrt{D})$ or $p^2<n$.
	\end{conjecture}
	
	This conjecture explains why the basis $\{G_D^{(1)}, G_D^{(2)}, G_D^{(3)}\}$ is natural: it separates the prime support of the denominators of the first and third sequences into inert and split primes while keeping the denominator of the second sequence small. The normalizing factor $s_D$ in Conjecture \ref{conj_mero_D_intro} does not appear to be optimal; see Conjecture \ref{conj_mero_-7} for the case $D=-7$.
	~\\[0.01in]
	
	The article is organized as follows. Section~\ref{sec:gamma0-4} gives a gentle introduction through the group $\Gamma_0(4)$ and explains its connection with elliptic integrals. Section~\ref{sec:gamma1-6} considers $\Gamma_1(6)$, recovers Ap\'ery's sequences from the weight-four minus space, and explores the corresponding phenomena in the plus space; it also proves Corollary \ref{cor_denomtype_intro}. Sections~\ref{sec:divisors} and~\ref{creative_telescope_section} form the theoretical core of the article: they recall the language of divisors of automorphic forms, prove Theorems \ref{dim_theorem} and \ref{creative_telescoping_S_intro}, and give an algorithm for computing creative telescopers. Section~\ref{sec:meromorphic} studies the two meromorphic examples introduced in Section~\ref{sec:intro-meromorphic}. Section~\ref{sec:zagier} studies a special geometric configuration that leads to the equivalence in \eqref{aux_32}. Section~\ref{sec:characters} develops character versions of the two main theorems and presents further examples. Section~\ref{sec:denominator-types}, which may be read directly after Section~\ref{sec:gamma1-6}, develops a general method for determining the denominator types of modulonomic sequences.
		
	\section*{Acknowledgment}
	The author expresses his sincere gratitude to Prof. Kathrin Bringmann, Don Zagier and Wadim Zudilin for valuable discussions and moral support. The author also thanks Prof. Frank Calegari, Vesselin Dimitrov, Yunqing Tang for discussions during Oberwolfach in March 2026. The author has received funding from the European Research Council (ERC) under the European Union’s Horizon 2020 research and innovation programme (grant agreement No. 101001179). 
	
	\section*{Declaration on AI usage}
	All theoretical inputs and ideas in this article come from the human author. ChatGPT 5.6 Sol and 6 Astra were used for the following auxiliary tasks: checking the language; verifying the numerical quantities presented in the examples; computing creative telescopers using the pseudocode in Section~\ref{creative_telescope_section}; and writing code to test Conjecture \ref{conj_mero_D_intro}.

	\section{Examples on $\Gamma_0(4)$}\label{sec:gamma0-4}
	Recall the familiar weight $k$ slash operator acting on functions on the upper half-plane $\mathbb{H} := \{\tau \in \mathbb{C} \mid \Im(\tau)>0\}$, given by
	$$(f|_k \gamma)(\tau) = (c\tau+d)^{-k} f(\frac{a\tau+b}{c\tau+d}), \qquad \gamma = \begin{psmallmatrix}a&b\\c&d\end{psmallmatrix} \in \SL_2(\mathbb{R}).$$
	Define the differential operator $$D := \frac{1}{2\pi i} \frac{d}{d\tau} = q \frac{d}{dq}, \qquad q=e^{2\pi i \tau}.$$
	Bol's identity \cite[Lemma~5.3]{bringmann2017harmonic} says that, for any smooth $f:\mathbb{H}\to \mathbb{C}$, we have
	$$D^{k-1}(f|_{2-k} \gamma) = (D^{k-1} f)|_k \gamma.$$
	Throughout, we omit the subscript in the slash notation whenever the weight is obvious. Thus Bol's identity can also be written as $D^{k-1}(f|\gamma) = (D^{k-1} f)|\gamma.$ \par
	First we recall some basic facts about the group $\Gamma:=\Gamma_0(4)$ and its modular forms. It has three cusps: $\{0,\infty,\frac12\}$. Its normalizer in $\SL_2(\mathbb{R})$ permutes these cusps. Indeed, $$V = \begin{pmatrix}1&0\\2&1\end{pmatrix} \text{ and } W_4 = \frac{1}{\sqrt{4}} \begin{pmatrix}0&-1\\4&0 \end{pmatrix}$$
	normalize the group, and $V(\infty) = \frac12$, while $W_4(\infty) = 0$. Let \begin{align}\label{Level4_uv}u &= \frac{\eta^{20}(2\tau)}{\eta^8(\tau) \eta^8(4\tau)} = 1+8q+24q^2+32q^4+\cdots \in \M_2(\Gamma), \\ v &= \frac{\eta^8(4\tau)}{\eta^4(2\tau)} = q+4q^3+6q^5+\cdots \in \M_2(\Gamma).\end{align}
	Their divisors are $$\div u = (\frac12),\quad \div v = (\infty),\quad \div(u-16v) = (0).$$
	It is easy to see \begin{equation}\label{aux_7}u| V = 16v,\qquad v|V = \frac{u}{16},\qquad u| W_4 = -u,\qquad v| W_4 = -\frac{u}{16}+v.\end{equation}
	The ring of modular forms and its dimensions are given by $$\bigoplus_{k\geq 0}\M_k(\Gamma) = \mathbb{C}[u,v],\qquad \dim \M_{2k}(\Gamma) = k +1.$$
	\subsection{Weight $k=4$ and moments of elliptic integrals}\label{sec:gamma0-4-weight4}
	Let $t$ and $g$ be the following forms of weights $0$ and $4$, respectively: $$\begin{aligned}&t(\tau) := \frac{256v}{u} = \frac{256\eta^8(\tau)\eta^{16}(4\tau)}{\eta^{24}(2\tau)} = 256(q-8q^2+44q^3+\cdots) \sim (\infty)-(\frac12), \\ &g(\tau) := (u-16v)v = q-8q^2+28q^3+\cdots \sim (0)+(\infty).\end{aligned}$$
	The integral $$(2\pi i)^3 \int_{0}^{i\infty} t(\tau)^n g(\tau) \tau^j d\tau$$
	is absolutely convergent for any $n\in \mathbb{N}, j\in \mathbb{Z}$, because the integrand vanishes at $0$ and $\infty$. We will be mainly interested in those $j$ corresponding to the critical values, which in this case are $j\in  \{0,1\}$.
	As a toy example to illustrate our theory, we will prove
	\begin{proposition}\label{basic_example}For $n\geq 0$, we have
	$$\begin{aligned}
	(2\pi i)^3 \int_{0}^{i\infty} t(\tau)^n g(\tau) d\tau &= A_n \frac{7\zeta(3)}{2} + B_n , \\
	(2\pi i)^3 \int_{0}^{i\infty} t(\tau)^n g(\tau) \tau d\tau &= A_n \frac{\pi^3 i}{16},
	\end{aligned}$$
	for some $A_n, B_n\in \mathbb{Q}$. Moreover, $A_n\in\mathbb{Z}$ and $\lcmD{n}^3 B_n\in\mathbb{Z}$.
	\end{proposition}

\begin{table}[h]
	\begin{tabular}{|c|c|c|c|c|c|c|c|}
		\hline
		$n$ & $0$ & $1$ & $2$ & $3$ & $4$ & $5$ & $6$ \\ \hline
		$A_n$ & $1$ & $8$ & $88$ & $1088$ & $14296$ & $195008$ & $2728384$ \\ \hline
		$B_n$ & $0$ & $8$ & $136$ & $\frac{56384}{27}$ & $\frac{850648}{27}$ & $\frac{1601473216}{3375}$ & $\frac{24177906368}{3375}$ \\ \hline
	\end{tabular}
\caption{First few values of $A_n, B_n$.}
\label{AB_table_1}
\end{table}
	
	Since $\dim \M_4(\Gamma_0(4)) = 3$ and $\dim \S_4(\Gamma_0(4))=0$, Theorem \ref{dim_theorem} implies that the space $$\frac{\mathcal{M}_4^!(\Gamma)}{D^3(\mathcal{M}_{-2}^!(\Gamma))}$$ is three-dimensional. Recall that for even weight $k\geq 4$ and any given cusp, there exists an Eisenstein series that does not vanish at that cusp and vanishes at all other cusps. When $n\geq 0$, $t^ng$ vanishes at $0$ and $\infty$. Since $g \in \E_4(\Gamma_0(4))$ also vanishes at $0$ and $\infty$, we have
	\begin{equation}\label{aux_6}t^n g = A_n g + D^3(h_n),\qquad A_n\in \mathbb{Q}, h_n\in \M_{-2}^!(\Gamma_0(4)).\end{equation} 
	
	Next we derive an explicit expression for $A_n$. For the matrix $V = \SmallMatrix{1 & 0 \\ 2 & 1}$ defined above, we have
	$$u| V = 16v, \qquad v|V = \frac{1}{16}{u},\qquad t|V = \frac{256}{t}=\frac{u}{v}.$$ Taking constant terms in $(t^n g)| V = A_n g| V + D^3(h_n|V)$ gives $$A_n = \frac{a_0((t^ng)|V)}{a_0(g|V)} = \frac{a_0((256/t)^n \frac{u}{16}(16v-u))}{a_0(\frac{u}{16}(16v-u))} = a_0\left((u/v)^n u(u-16v)\right).$$
	Since $u/v\in q^{-1}(1+q\mathbb{Z}\llbracket q\rrbracket)$ and $u(u-16v)$ has integral Fourier coefficients, we see that $A_n\in\mathbb{Z}$. This formula provides a convenient way to calculate $A_n$ for small $n$.
	\par
	
	We now integrate both sides of \eqref{aux_6} from $0$ to $i\infty$. As an Eisenstein series, $g$ satisfies $g(\tau) = \frac{1}{240}E_4(\tau )-\frac{17}{240} E_4(2 \tau )+\frac{1}{15} E_4(4\tau)$. Thus
	$$L(g(\tau),s) = \frac{(-2\pi i)^s}{\Gamma(s)} \int_0^{i\infty} g(\tau) \tau^{s-1} d\tau = 4^{-s} \left(16-17\times 2^s+4^s\right) \zeta (s-3) \zeta (s).$$
	Letting $s\to 1$ and $s\to 2$ gives respectively $$(2\pi i)^3 \int_{0}^{i\infty} g(\tau) d\tau = \frac{7\zeta(3)}{2}, \qquad (2\pi i)^3 \int_{0}^{i\infty} g(\tau) \tau d\tau  =  \frac{\pi^3 i}{16}.$$
	Therefore we have
	$$\begin{aligned}(2\pi i)^3 \int_{0}^{i\infty} t(\tau)^n g(\tau) d\tau &= A_n \frac{7\zeta(3)}{2} + \int_0^{i\infty} \left(\frac{d^3}{d\tau^3} h_n(\tau)\right) d\tau, \\
	(2\pi i)^3 \int_{0}^{i\infty} t(\tau)^n g(\tau) \tau d\tau &= A_n \frac{\pi^3 i}{16} + \int_0^{i\infty} \left(\frac{d^3}{d\tau^3} h_n(\tau)\right) \tau d\tau
	\end{aligned}$$
	It is easy to evaluate these two integrals using the fundamental theorem of calculus and integration by parts. For later reference, we state a more general result:
	\begin{lemma}\label{IBP_lemma}
	Let $\Gamma$ be a subgroup of $\text{SL}_2(\mathbb{R})$ commensurable with $\text{SL}_2(\mathbb{Z})$, let $k\geq 2$, and let $h\in \M_{2-k}^S(\Gamma)$. Let $\alpha,\beta \in \mathbb{Q}\cup \{\infty\}$, and suppose that $h$ is analytic at these two cusps. Then, for $0\leq i\leq k-2$,
	$$\int_\alpha^\beta \left( \frac{d^{k-1}}{d\tau^{k-1}} h(\tau) \right) \tau^i d\tau = B_h(\beta,i) - B_h(\alpha,i),$$ where
	$$B_h(\alpha,i) = \begin{cases}0 & \alpha = \infty,0\leq i\leq k-3, \\ (-1)^{k-2} (k-2)! \times a_0(h) & \alpha = \infty, i=k-2 \\
	\alpha^i (k-2)! (-c)^{k-2} \times a_0(h|_{2-k}\gamma) & \alpha\neq \infty, \gamma = \SmallMatrix{a&b\\c&d}\in \text{SL}_2(\mathbb{R}) \text{ such that }\alpha = \gamma\infty.\end{cases}$$
	\end{lemma}
	Note that the result is independent of the path from $\alpha$ to $\beta$. Indeed, if $\frac{d^{k-1}}{d\tau^{k-1}} h(\tau)$ has a pole in $\mathbb{H}$, then it is of order $\geq k-1$, and since $i\leq k-2$, the integrand does not have residues at any pole in $\mathbb{H}$.
\begin{proof}
First we claim that, for $\alpha\neq \infty$, with $\gamma = \SmallMatrix{a&b\\c&d}\in \text{SL}_2(\mathbb{R})$ such that $\alpha = \gamma\infty$, we have
$$\lim_{\tau\to \alpha} \frac{d^i}{d\tau^i} h(\tau) = \begin{cases}0 &\text{if } 0\leq i\leq k-3 \\ (k-2)! (-c)^{k-2} a_0(h|_{2-k}\gamma) &\text{if } i=k-2
\end{cases}$$
Write $g(\tau) = (h|_{2-k}\gamma)(\tau)$ and $g(\tau) = A + O(e^{-\varepsilon \Im(\tau)})$ as $\tau\to \infty$. Then $$h = g|_{2-k} \gamma^{-1} = (-c\tau + a)^{k-2} g(\gamma^{-1}\tau) = (-c\tau+a)^{k-2}\left(A+O(e^{-\varepsilon \Im(\gamma^{-1}\tau)})\right), \qquad \tau\to \alpha.$$
The term $e^{-\varepsilon \Im(\gamma^{-1}\tau)}$ and all of its derivatives tend to $0$ exponentially as $\tau\to \alpha$. Also note that $\gamma\infty = \alpha \iff -c\alpha+a = 0$, so the only nonvanishing derivative corresponds to $i=k-2$. Moreover,
$$\left.\frac{d^{k-2}}{d\tau^{k-2}} (-c\tau+a)^{k-2}\right|_{\tau=\alpha} = (-c)^{k-2}(k-2)!,$$
and $A = a_0(h|_{2-k}\gamma)$. This proves the claim. \par
Repeated integration by parts gives, for $0\leq i\leq k-2$, 
$$\int \tau^i h^{(k-1)}(\tau) d\tau = \sum_{j=0}^i (-1)^j i(i-1)\cdots (i-j+1) \tau^{i-j} h^{(k-2-j)}(\tau).$$
It remains to show $$C_h(\alpha,i) := \sum_{j=0}^i (-1)^j i(i-1)\cdots (i-j+1) \lim_{\tau\to \alpha}\left(\tau^{i-j} h^{(k-2-j)}(\tau)\right)$$
equals $B_h(\alpha,i)$ as defined in the lemma. We split into the cases $\alpha = \infty$ and $\alpha \in \mathbb{Q}$.
\begin{itemize}[leftmargin=*]
\item $\alpha = \infty$: then $h^{(\geq 1)}(\tau)$ decays exponentially as $\tau\to \infty$, so only the term with $j=k-2$ can contribute. Thus $C_h(\alpha,i) = (-1)^{k-2}i(i-1)\cdots(i-k+3) \lim_{\tau \to \infty} (\tau^{i-(k-2)} h(\tau))$. This is nonzero only when $i=k-2$; in this case, it is $(-1)^{k-2} (k-2)! a_0(h)$.
\item $\alpha \in \mathbb{Q}$: in this case
$$C_h(\alpha,i) = \sum_{j=0}^i (-1)^j i(i-1)\cdots (i-j+1) \alpha^{i-j}\lim_{\tau\to \alpha}\left( h^{(k-2-j)}(\tau)\right).$$
By the previous observation, the summand is non-zero only when $k-2-j = k-2\iff j=0$, so $C_h(\alpha,i) = \alpha^i \lim_{\tau\to\alpha} h^{(k-2)}(\tau)$, completing the proof.
\end{itemize}
\end{proof}
	
Using the Fricke involution $W_4$, which maps $\infty$ to $0$, we deduce
	$$B_n := \int_0^{i\infty} \left(\frac{d^3}{d\tau^3} h_n(\tau)\right) d\tau = -8a_0(h_n|W_4), \qquad \int_0^{i\infty} \left(\frac{d^3}{d\tau^3} h_n(\tau)\right) \tau d\tau = 0.$$
From this it becomes evident that $B_n \in \mathbb{Q}$.
	\begin{example}
		The value of $B_n$ is more difficult to calculate than that of $A_n$. At this stage, we have only a naive method. Note that any form in $\M_k^!(\Gamma_0(4))$ is a rational function\footnote{Recall that $\div(u) = (\frac12)$, $\div(v) = (\infty)$, and $\div(u-16v)=(0)$; these are the three cusps of $\Gamma_0(4)$.} in $u,v$. We first find $h_n$ explicitly in terms of $u,v$ such that \eqref{aux_6} holds, and then use $$u|W_4 = -u,\qquad v|W_4 = \frac{-1}{16}u+v.$$
		Consider the case $n=2$. Since $\div(t^2 g) = 3(\infty)+(0)-2(\frac12)$, the negative part of $\div(h_2)$ is $2(\frac12)$. Hence $u^2 h_2 \in \M_{2}(\Gamma_0(4))$ is a holomorphic modular form, and $h_2 = (c_1 u + c_2 v)/u^2$ for some constants $c_1,c_2$. Their values can be found by comparing sufficiently many $q$-coefficients on both sides of
		$$t^2 g = A_2 g + c_1 D^3(\frac{1}{u}) + c_2 D^3(\frac{v}{u^2}),$$
		giving $A_2 = 88$, $c_1 = 15$, and $c_2 = 32$. Thus $$B_2 = -8 a_0(h_2|W_4) = -8 a_0\left(\frac{-17u+32v}{u^2}\right) = 136.$$
		For general $n$, one sets $h_n = u^{-n}\left(\sum_{i=0}^{n-1} c_i u^i v^{n-i}\right)$ and then solves for $c_i$, giving the values of $B_n$ in Table \ref{AB_table_1}.
	\end{example}
	
	It is not obvious from the above computation that $B_n$ has denominator type $\lcmD{n}^3$. In fact, it is a useful feature of the theory that its denominator type can be deduced without any computation. For the present introductory example, we give a more direct proof, which will be vastly generalized in a later section.
	
	\begin{lemma}\label{victor_miller_gamma0(4)}
	Let $f\in \M_{2k}(\Gamma_0(4))$. If its Fourier coefficients of $q^0,q^1,\ldots,q^k$ are integers, then all its Fourier coefficients are integers.
	\end{lemma}
	\begin{proof}
	From the structure of the ring of modular forms, $f$ is a polynomial in $u,v$: $f = \sum_{j=0}^k c_j u^{k-j} v^j$. Note that $u \in 1+\mathbb{Z}\llbracket q\rrbracket$ and $v \in q(1+\mathbb{Z}\llbracket q\rrbracket)$. Comparing the coefficient of $q^0$ implies that $c_0 \in \mathbb{Z}$; comparing the coefficient of $q^1$ then implies that $c_1\in \mathbb{Z}$, and so on. Thus $c_0,\ldots,c_k\in \mathbb{Z}$.
	\end{proof}

	\begin{lemma}\label{16n_trans}
	Let $f\in \M_{2k}(\Gamma_0(4))$ have integral Fourier coefficients. Then so do $$16^k f|W_4,\qquad 16^k f|V,\qquad 16^k f|(V W_4).$$
	\end{lemma}
	\begin{proof}
	We can write $f = \sum_{j=0}^k c_j u^{k-j} v^j$ with $c_j\in \mathbb{Z}$. The assertion then follows immediately from \eqref{aux_7} and
	\begin{equation*}u|_2 (V W_4) = -u+16v,\qquad v|_2 (V W_4) = -\frac{u}{16}.\qedhere\end{equation*}
	\end{proof}

	Now we are ready to prove the assertion that $\lcmD{n}^3 B_n\in\mathbb{Z}$. Because $t^ng=A_ng+D^3(h_n)$ has poles only at $\frac12$, of order at most $n$, the same is true of $h_n$. Thus $u^nh_n\in\M_{2n-2}(\Gamma)$. Write
	\begin{equation}\label{aux_8}
	h_n|V=a_0(h_n|V)+
	\sum_{\substack{m=-n\\m\neq0}}^\infty
	\frac{a_m(D^3(h_n|V))}{m^3}q^m.
	\end{equation}
	From
	$$D^3(h_n|V)=\bigl((u/v)^n-A_n\bigr)(g|V),
	\qquad g|V=\frac{u}{16}(16v-u),$$
	and $A_n\in\mathbb{Z}$, we see that $16D^3(h_n|V)$ has integral Fourier coefficients. Multiplying by $\lcmD{n}^3$ to clear the denominators $m^3$ in \eqref{aux_8}, we find that the negative Fourier coefficients of $16\lcmD{n}^3(h_n|V)$ are integers.
	Now consider
	$$F:=16\lcmD{n}^3v^n(h_n|V)
	=16^{1-n}\lcmD{n}^3(u^nh_n)|V.$$
	This is a holomorphic modular form of weight $2n-2$. Since $v^n=q^n+O(q^{n+1})$ has integral Fourier coefficients, the coefficients of $q^0,q^1,\ldots,q^{n-1}$ in $F$ are integers. Lemma \ref{victor_miller_gamma0(4)} implies that all Fourier coefficients of $F$ are integers. Applying Lemma \ref{16n_trans} in weight $2n-2$, and using $V^2\in\Gamma_0(4)$, shows that both
	$$16^{n-1}F|V=\lcmD{n}^3u^nh_n,\qquad
	16^{n-1}F|(V W_4)=\lcmD{n}^3(-u)^n(h_n|W_4)$$
	have integral Fourier coefficients. Since $u$ is a unit in $\mathbb{Z}\llbracket q\rrbracket$, the same holds for $\lcmD{n}^3h_n$ and $\lcmD{n}^3(h_n|W_4)$. In particular,
	$$\lcmD{n}^3B_n=-8a_0\bigl(\lcmD{n}^3(h_n|W_4)\bigr)\in\mathbb{Z}.$$
	The proof of Proposition \ref{basic_example} is now completed.
	
	\begin{corollary}\label{congruence_level4_cor}
	Let $t,g$ be as above, let $p$ be an odd prime, and let $l\geq 1$. If $p>n$, then $$(1-p^3) a_{p^l}(t^n g) \equiv A_n \pmod{p^{3l}}.$$
	\end{corollary}
	\begin{proof}
	Taking the $p^l$-th Fourier coefficient on both sides of \eqref{aux_6} gives
	$$a_{p^l}(t^n g) = A_n a_{p^l}(g) + p^{3l} a_{p^l}(h_n).$$
	The proof above also shows that $\lcmD{n}^3 h_n$ has integral Fourier coefficients, so $\lcmD{n}^3 a_{p^l}(h_n) \in \mathbb{Z}$. When $p>n$, the denominator of $a_{p^l}(h_n)$ is not divisible by $p$; hence $$a_{p^l}(t^n g) \equiv A_n a_{p^l}(g) \pmod{p^{3l}}.$$
	Since $g$ is an Eisenstein series, it is easy to calculate that $a_{p^l}(g) = \sigma_3(p^l) = 1+p^3+\cdots+p^{3l} = \frac{p^{3l+3}-1}{p^3-1} \equiv (1-p^3)^{-1} \pmod{p^{3l}}$. This proves the claim.
	\end{proof}

\begin{corollary}\label{congruence_level4_cor_2}
If $p$ is an odd prime, then $$A_{p-1} \equiv 256^{p-1}\pmod{p^3},\qquad A_{p-2} \equiv 256^{p-2} (8-8p) \pmod{p^3}.$$
\end{corollary}
\begin{proof}
There are corresponding statements for $A_{p-m}$ for every positive integer $m$; we give only the cases $m=1,2$ as illustrations. Taking $l=1$ and $n=p-1$ or $n=p-2$ in the preceding corollary gives $A_{p-1} \equiv a_p(t^{p-1}g) \pmod{p^3}$ and $A_{p-2} \equiv a_p(t^{p-2}g) \pmod{p^3}$. Since $$t = 256q(1-8q+44q^2+\cdots),\qquad g=q(1-8q+28q^2+\cdots),$$
we have $$t^{p-1}g = 256^{p-1}q^p + O(q^{p+1}),\qquad t^{p-2}g = 256^{p-2}q^{p-1}(1-8q)^{p-1} + O(q^{p+1}),$$
so $a_p(t^{p-1}g) = 256^{p-1}, a_p(t^{p-2}g) = 256^{p-2}(8-8p)$.
\end{proof}
	
	\begin{comment}
	\begin{remark}
	As mentioned in the introduction, certain aspects of Proposition \ref{basic_example} is actually known in the literature, although usually not formulated in the language of modular form. Denote $K(x) := \int_0^1 \frac{1}{\sqrt{(1-y^2)(1-xy^2)}}dy = \frac{\pi}{2} {_2F_1}(\frac12,\frac12;1;x)$ to be the complete elliptic integral of the first kind. There are many hypergeometric-function-based methods that allow one to show that \cite{zhou2014legendre, Wanthesis, campbell2019interplay, wan2012moments}
	$$\int_0^1 x^n K(x)^2 dx \in \mathbb{Q}\zeta(3) + \mathbb{Q}, \qquad \int_0^1 x^n K(x) K(1-x) dx \in \mathbb{Q}\pi^3 + \mathbb{Q}.$$
	Using a well-known pullback that transform $K(x)$ into weight $1$ modular form (see e.g. \cite{chanlegendre, wan2016integrals, rogers2015moments, wan2012moments}), one immediately recognizes they are the same object, with $x=16v/u=t/16$:
		$$\begin{aligned}\int_0^1 x^n K(x)^2 dx &= 16^{-n}(2\pi i)^3 \int_{0}^{i\infty} t(\tau)^n g(\tau) d\tau, \\
		\int_0^1 x^n K(x) K(1-x) dx &= -2i\,16^{-n}(2\pi i)^3 \int_{0}^{i\infty} t(\tau)^n g(\tau) \tau d\tau.\end{aligned}$$
	The denominator types and congruences properties of these properties, on the other hand, are not usually discussed in this context.
	\end{remark}
	\end{comment}
	
	\begin{remark}\label{weight_4_Gamma04_cong}
	It can be shown that the sequence $A_n$ satisfies the following Dwork-type supercongruence:
	$$A_{p^am-1}\equiv 256^{(p-1)p^{a-1}m}A_{p^{a-1}m-1}\pmod{p^{3a}},\qquad p\geq3,\quad a,m\geq1.$$
	\end{remark}
	
	\subsection{A weight $6$ example}\label{sec:gamma0-4-weight6}
	Using the fact that the plus $W_4$-eigenspace of $\S_6(\Gamma_0(4))$ is trivial, we can deduce an analogous version of the previous result for $\zeta(5)$. Namely, let $G = \langle \Gamma_0(4),W_4\rangle$. It has two cusps $\{\frac12,\infty\}$, where the distinct $\Gamma_0(4)$-cusps $0$ and $\infty$ merge. The quotient $\frac{\M_6^!(G)}{D^5(\M_{-4}^!(G))}$ has dimension $2$. Let
	\begin{equation}\begin{aligned}t(\tau) &:= \frac{64^2 v}{u}\left(1-\frac{16v}{u}\right) = 2^{12} q(1+24q+300q^2+\cdots) \in \M_0^!(G) \\
	g(\tau) &:= v(u-16v)(u-32v) = q - 32 q^2 + 244 q^3 + \cdots	\in \M_6(G) 
%	\\ \theta &:= \frac{64^2 g}{Dt} = u^2.
\end{aligned}\end{equation}
	%Here $Dt= q\frac{dt}{dq}  = t(u-32v)$. 
	Note that $t$ is a Hauptmodul on $G$, with $\div_G(t) = (\infty)-(\frac12)$. Also $g$ is, up to scalar multiple, the unique element of $\E_6(G)$ vanishing at $\infty$. Therefore, when $n\geq 0$, we have
	\begin{equation}\label{aux_26}t^n g = A_n g + D^5(h_n), \qquad A_n\in \mathbb{Q},\quad h_n \in \M_{-4}^!(G).\end{equation}
	The form $h_n$ is holomorphic at $0$ and $\infty$. Integrating as before and using Lemma \ref{IBP_lemma}, together with $h_n| W_4=h_n$, we obtain
	$$(2\pi i)^5 \int_0^{i\infty} t(\tau)^n g(\tau)\,d\tau
	= A_n (186\zeta(5))-384a_0(h_n).$$
	These periods can also be converted into integrals involving $K(x)$. With $x=16v/u$, we have
	$$(2\pi i)^5 \int_0^{i\infty} t(\tau)^n g(\tau)\,d\tau
	=16\int_0^1 [256x(1-x)]^n(2x-1)K(x)^4\,dx.$$
	This is a rare example of moments involving only a single $\zeta(5)$, without mixing lower zeta values; see \cite{zlobin2007rhin} for another example. It appears that only the $n=0$ case is known in the literature \cite{zhou2018two}:
	$$\int_0^1 (2x-1)K(x)^4\,dx=\frac{93}{8}\zeta(5).$$
	Set $B_n:=-384a_0(h_n)$. Using the same argument as above, one can show that $A_n$ has denominator type $1$ and $B_n$ has denominator type $\lcmD{n}^5$. We leave the details to the interested reader.

	\section{Examples on $\Gamma_1(6)$ and irrationality}\label{sec:gamma1-6}
	We quickly recall some basic facts about modular forms on $\Gamma:=\Gamma_1(6)$. Recall the Atkin--Lehner involutions
	$$W_2 = \frac{1}{\sqrt{2}} \begin{pmatrix}-2 & -1 \\ 6 & 2 \end{pmatrix}, \qquad W_3 = \frac{1}{\sqrt{3}}\begin{pmatrix} 3 & 1 \\ 6 & 3\end{pmatrix},\qquad W_6 = \frac{1}{\sqrt{6}}\begin{pmatrix} 0 & -1 \\ 6 & 0\end{pmatrix}.$$
	They normalize $\Gamma_1(6)$ and act transitively on its four cusps. Let \begin{align}\label{level6_uv}\begin{split}u &:= \frac{\eta(\tau)^6 \eta(6\tau)}{\eta(2\tau)^3 \eta(3\tau)^2} =1-6q+12q^2+\cdots \in \M_1(\Gamma), \\ v &:= \frac{\eta(\tau) \eta(6\tau)^6}{\eta(2\tau)^2 \eta(3\tau)^3} = q-q^2+\cdots \in \M_1(\Gamma), \\
	u+8v &=  \frac{\eta(2\tau) \eta(3\tau)^6}{\eta(\tau)^2 \eta(6\tau)^3},\qquad u+9v =  \frac{\eta(3\tau) \eta(2\tau)^6}{\eta(6\tau)^2 \eta(\tau)^3}.\end{split}\end{align}
	The dimension of $\M_k(\Gamma)$ is $k+1$; therefore, the ring of holomorphic modular forms is $\mathbb{C}[u,v]$.
	For each cusp, there is a unique (up to a scalar multiple) weight $1$ form vanishing exactly at that cusp:
	$$\div u = (0),\qquad \div v = (\infty),\qquad \div(u+9v) = (\frac12),\qquad \div(u+8v) = (\frac13).$$
	Also
	$$u| W_6 = -6i\sqrt{2} v,\qquad v| W_6 = \frac{-i}{6\sqrt{2}}u,\qquad t| W_6 = t.$$
	In most of this section, we will consider
	\begin{equation}\label{t_choice_level6}t(\tau) := \frac{uv}{(u+8v)(u+9v)} = \left(\frac{\eta(\tau)\eta(6\tau)}{\eta(2\tau) \eta(3\tau)}\right)^{12}\sim (0)+(\infty)-(\frac12)-(\frac13).\end{equation}
	For a fixed $g\in \M_k^!(\Gamma)$, the integrals $$(2\pi i)^k \int_0^{i\infty} t(\tau)^n g(\tau) \tau^j d\tau,\qquad n\gg 0,\qquad j\in \{0,\ldots,k-2\},$$ are absolutely convergent. Since, on the imaginary axis, $$|t(\tau)| \leq \left|t(\frac{i}{\sqrt{6}})\right| = (\sqrt{2}-1)^4,$$
	these integrals decay at the rate $O((\sqrt{2}-1)^{4n})$.

		\subsection{Weight $4$ minus space and Ap\'ery sequences}\label{sec:gamma1-6-minus}
	Let $\M_k^!(\Gamma,\pm)$ denote the $\pm$-eigenspaces of $W_6$ inside $\M_k^!(\Gamma)$. The unique cusp form in $\M_4(\Gamma)$ lies in the plus space. Then Bol's identity implies the splitting
	$$\frac{\M_4^!(\Gamma)}{D^{3}(\M_{-2}^!(\Gamma))} = \frac{\M_4^!(\Gamma,+)}{D^{3}(\M_{-2}^!(\Gamma,+))} \oplus \frac{\M_4^!(\Gamma,-)}{D^{3}(\M_{-2}^!(\Gamma,-))}.$$
	By Theorem \ref{dim_theorem}, the left-hand side has dimension $\dim \E_4(\Gamma) + 2\dim\S_4(\Gamma) = 6$. If we let $G = \langle \Gamma,W_6 \rangle$ be the group generated by $\Gamma$ and $W_6$, then the first factor on the right-hand side is $\frac{\M_4^!(G)}{D^3(\M_{-2}^!(G))}$, which, again by Theorem \ref{dim_theorem}, has dimension $\dim \E_4(G) + 2\dim\S_4(G) = 2+2 = 4$. Hence the second factor has dimension $2$, and a basis can be chosen to consist of two independent elements of $\E_4(\Gamma,-)$.

	First we look at the minus space. Let $$g = \frac{1}{2}(72uv^3-u^3v) =  -\frac{E_4(\tau)}{480}+\frac{7}{120}E_4(2\tau)-\frac{21}{160}E_4(3\tau)+\frac{3}{40}E_4(6\tau)$$ be the unique (up to a scalar multiple) element of $\E_4(\Gamma,-)$ that is nonvanishing at the cusps $\frac12$ and $\frac13$. Since $t$ vanishes at $0$ and $\infty$, we can write, for $n\geq 0$, 
	\begin{equation}\label{aux_11}t^n g = \Ap_n g + D^3(h_n),\qquad \Ap_n\in \mathbb{Q}, h_n\in \M_{-2}^!(\Gamma,-).\end{equation}
	It is easy to show that $(2\pi i)^3 \int_0^{i\infty} g\,d\tau = \zeta(3)$. Invoking Lemma \ref{IBP_lemma} yields $$(2\pi i)^3 \int_0^{i\infty} t(\tau)^n g(\tau) d\tau = \Ap_n \zeta(3) - 2! (\sqrt{6})^{4-2} a_0(h_n|W_6) = \Ap_n \zeta(3) + \underbrace{12a_0(h_n)}_{:= -\Bp_n}.$$
	The sequence $\{\Ap_n\}$ has trivial denominator type. To see this, let $\gamma \in \text{SL}_2(\mathbb{R})$ map $\infty$ to $\frac13$. The same argument as in the previous section gives
	$$\Ap_n = \frac{a_0((t^ng)|\gamma)}{a_0(g|\gamma)}.$$
	For example, we can take the Atkin--Lehner involution $W_2$; then $t|W_2 = t^{-1}$. Hence
	$$\Ap_n = \frac{a_0(t^{-n} (g|W_2))}{a_0(g|W_2)}.$$
	Since $t\in q(1+\mathbb{Z}\llbracket q  \rrbracket)$, this proves the claim about the denominator type of $\Ap_n$. The denominator type of $\Bp_n$ is less immediate. We will see below (Corollary \ref{cor:ABCD_denom_type}) that $\lcmD{n}^3 h_n$ has integral Fourier coefficients; therefore, $a_0(h_n)$ has denominator type $\lcmD{n}^3$.
	
	Hence we have produced a rational linear combination $\Ap_n \zeta(3) - \Bp_n = O((\sqrt{2}-1)^{4n})$, with $\Ap_n$ having trivial denominator type and $\Bp_n$ having denominator type $\lcmD{n}^3$. Since $(\sqrt{2}-1)^4 e^3 < 1$, we conclude that $\zeta(3)$ is irrational. 
	
	As an illustration of the creative telescoping method, let us show that the sequences $\Ap_n, \Bp_n$ defined above are annihilated by the recurrence operator
	$$\boldL := (n+1)^3 -(3 + 2 n) (39 + 51 n + 17 n^2) \textbf{S}_n +(n+2)^3 \textbf{S}_n^2,$$
	where $\textbf{S}_n$ shifts $n$ to $n+1$. Together with the initial conditions $(\Ap_0,\Ap_1,\Bp_0,\Bp_1)=(1,5,0,6)$, this shows that $\Ap_n, \Bp_n$ coincide with the well-known approximations to $\zeta(3)$. 
	
	Let $R = -\frac{uv}{2(u+8v)^2(u+9v)^2} \in \M_{-2}^!(\Gamma)$, which we will call the \textit{certificate} of creative telescoping. We pull out of the hat the following identity, which holds for any $n\in \mathbb{Z}$:
	$$\boldL(t^n g) = D^3\left(t^n R\right).$$
	More concretely, this means
	\begin{equation}\label{aux_12}(n+1)^3 t^n g -(3 + 2 n) (39 + 51 n + 17 n^2) t^{n+1} g+(n+2)^3 t^{n+2}g =  D^3\left(t^n  R\right).\end{equation}
	Assuming that \eqref{aux_12} holds, let us prove that $\boldL$ annihilates $\Ap_n, \Bp_n$. Apply $\boldL$ to both sides of \eqref{aux_11}, giving
	$$\boldL(t^n g) = \boldL(\Ap_n) g + D^3(\boldL h_n) = D^3(t^n R).$$
	Thus $\boldL(\Ap_n) g \in D^3(\M_{-2}^!(\Gamma))$, which forces $\boldL(\Ap_n) = 0$ because $g$ has a nonzero constant term at some cusp, whereas every element of $D^3(\M_{-2}^!(\Gamma))$ has zero constant term at every cusp. Then $D^3(\boldL h_n) = D^3(t^n R)$ implies $\boldL h_n = t^n R$: the kernel of $D^3$ consists of polynomials in $\tau$, none of which is a nonzero modular form of weight $-2$. Taking constant terms gives $$\boldL a_0(h_n) = a_0(\boldL h_n) = a_0(t^n R).$$
	Since $n\geq 0$ and both $t$ and $R$ vanish at $\infty$, this is zero. Thus $\boldL$ also annihilates $\Bp_n$. 
	
	In the spirit of classical creative telescoping, once an identity like \eqref{aux_12} has been written down, it is straightforward to verify. Let us briefly indicate how this can be done.
	\begin{itemize}[leftmargin=*]
	\item Using quasimodular forms, one can successively compute $D, D^2, D^3$ of the right-hand side, introducing only one extra object, $E_2 := 1-24\sum_{n\geq 1} \frac{nq^n}{1-q^n}$, with
	$$Du = \frac{E_2 u}{12}-\frac{u^3}{12}-5 u^2 v-36 u v^2, \qquad Dv = \frac{E_2v}{12}+\frac{11 u^2 v}{12}+12 u v^2+36 v^3.$$
	The verification then becomes a straightforward computation in the purely transcendental field $\mathbb{Q}(u,v,E_2)$. We will not pursue this approach further, as the next approach seems simpler and more natural.
	\item Instead, after developing some theory in Section~\ref{creative_telescope_section}, we will see a more natural and simpler way to verify identities like these: compare the first few Fourier coefficients on both sides of \eqref{aux_12} (see Example \ref{Ap\'ery_rec_verification}).
	\end{itemize}

	\begin{corollary}
	Let $t,g$ be as in \eqref{aux_11}, let $p\geq 5$ be a prime, and let $l\geq 1$. If $p>n$, then $$(1-p^3) a_{p^l}(t^n g) \equiv \frac{-1}{2}\Ap_n \pmod{p^{3l}}.$$
	\end{corollary}
\begin{proof}
The proof proceeds in exactly the same way as that of Corollary \ref{congruence_level4_cor}. One needs only the facts that $\lcmD{n}^3 h_n$ has integral Fourier coefficients (which we will establish later in a more general context) and that $g = -\frac{E_4(\tau)}{480}+\frac{7}{120}E_4(2\tau)-\frac{21}{160}E_4(3\tau)+\frac{3}{40}E_4(6\tau)$, so that, when $p\geq 5$,
$a_{p^l}(g) = \frac{-1}{2}\frac{p^{3l+3}-1}{p^3-1}$.
\end{proof}

\begin{corollary}\label{cor:Ap\'ery_mod_p}
If $p\geq 5$ is a prime, then $$\Ap_{p-1} \equiv 1\pmod{p^3},\qquad \Ap_{p-2}\equiv 5-12p \pmod{p^3}.$$
\end{corollary}
\begin{proof}
Again, this is the same, \textit{mutatis mutandis}, as the proof of Corollary \ref{congruence_level4_cor_2}. We have $a_p(t^{p-1}g) \equiv \frac{-1}{2}\Ap_{p-1} \pmod{p^3}$ and $a_p(t^{p-2}g) \equiv \frac{-1}{2}\Ap_{p-2} \pmod{p^3}$. Using
$$t = q(1-12q+66q^2+\cdots),\qquad g=\frac{-q}{2}(1-19q+91q^2+\cdots),$$
one easily calculates that $a_p(t^{p-1}g) = \frac{-1}{2}$ and $a_p(t^{p-2}g) = \frac{12p-5}{2}$.
\end{proof}

Note that the congruence for $\Ap_{p-2}$ also follows directly from the recurrence once that for $\Ap_{p-1}$ is known, so only the latter is the interesting conclusion.

\subsection{Weight $4$ plus space}\label{sec:gamma1-6-plus}
	Next let us turn our attention to the plus space. Let $\Gamma^+ := \langle \Gamma_1(6),W_6\rangle$. To perform explicit computations, we need to fix a basis of the two-dimensional space $$\frac{\S_4^!(\Gamma^+)}{D^3(\M_{-2}^!(\Gamma^+))},$$
	which we fix as
	\begin{align}\begin{split}
		&f := u v (u + 8 v) (u + 9 v) =\eta(\tau)^2 \eta(2\tau)^2 \eta(3\tau)^2 \eta(6\tau)^2\in \S_4(\Gamma^+) \\ & f^\sharp := \frac{u v \left(u^4+12 u^3 v+60 u^2 v^2+864 u v^3+5184  v^4\right)}{(u+8 v) (u+9 v)} \in \S_4^!(\Gamma^+).\end{split}\end{align}
Also let $E := u v (13 u^2 + 216 u v + 936 v^2) \in \E_4(\Gamma^+)$ be the unique Eisenstein series that vanishes at the cusps $0$ and $\infty$. Consider the sequences $A_n, B_n, C_n \in \mathbb{Q}$ defined by
\begin{equation}\label{aux_14}t^n f = A_n E + B_n f + C_n f^\sharp + D^3(h_n), \qquad h_n \in \M_{-2}^!(\Gamma^+).\end{equation}
Denote
$$\begin{aligned}
	\xi_1 &= \frac{(2\pi i)^3}{2}  \int_0^{i\infty} f(\tau) d\tau,\qquad \xi_2 = \frac{(2\pi i)^3}{2}  \int_0^{i\infty} f(\tau) \tau d\tau, \\
	\eta_1 &= \frac{(2\pi i)^3}{2} i\int_0^{i\infty} f^\sharp(\tau) d\tau - \frac{78i}{5},\qquad \eta_2 = \frac{(2\pi i)^3}{2} i \int_0^{i\infty} f^\sharp(\tau) \tau d\tau 
\end{aligned}$$
One can show that they satisfy the Legendre relation \cite{golyshev2021interpolated} $$\eta_2 \xi_1 - \eta_1 \xi_2 = -\frac{4\pi^3}{5}.$$
By observing that $(u+8v)^n (u+9v)^n h_n \in \M_{2n-2}(\Gamma)$, one can find $A_n, B_n, C_n, h_n$ for each fixed $n$, obtaining the example values in the table below, where $D_n = \frac{156}{5} C_n - 12 a_0(h_n)$. Using the by-now familiar argument of comparing constant terms at cusps, it is evident that $A_n$ has trivial denominator type. The denominator types of $B_n,C_n,D_n$ are not obvious. We will show that they are $\lcmD{n}^3$. In fact, we will give two proofs of this fact of completely different kinds, one in Proposition \ref{prop:level_six_plus_constant_terms} and the other in Corollary \ref{cor:ABCD_denom_type}.
$$\begin{array}{c|cccc}
	n&A_n&B_n&C_n & D_n \\ \hline
	0&0&1&0 & 0\\
	1&-\frac15&\frac{13}{5}&0 & 0\\
	2&-\frac{22}{5}&\frac{281}{5}&1 & \frac{156}{5}\\
	3&-\frac{591}{5}&\frac{7528}{5}&\frac{135}{4} & \frac{2103}{2}\\
	4&-\frac{17108}{5}&\frac{652957}{15}&\frac{19177}{18}&\frac{1492981}{45}
\end{array}$$
Applying $(2\pi i)^3\int_0^{i\infty} \cdots d\tau$ and $(2\pi i)^3\int_0^{i\infty} \cdots \tau d\tau$ to \eqref{aux_14} gives
$$\begin{aligned}\frac{(2\pi i)^3}{2} \int_0^{i\infty} t(\tau)^nf(\tau) d\tau &= 47\zeta(3)A_n+B_n\xi_1-C_n\eta_1+\frac{D_n}{2},\\
\frac{(2\pi i)^3}{2}\int_0^{i\infty} t(\tau)^nf(\tau) \tau d\tau &= \frac{2\pi^3}{3}A_n+B_n\xi_2-C_n\eta_2.
\end{aligned}$$
Both left-hand sides decay at the geometric rate $(\sqrt{2}-1)^{4n}$. Taking linear combinations of these two expressions proves the following.
\begin{corollary}\label{cor_plus_level6_irrational}
For any $(u,v)\in \mathbb{R}^2$, $(u,v)\neq (0,0)$, at least one of the following four numbers is irrational:
$$\left\{47\zeta(3)u + \frac{2\pi^3}{3}v,  u\xi_1 + v\xi_2, u\eta_1+v\eta_2, u \right\}.$$
\end{corollary}
By taking $u=\frac{\xi_2}{\pi^3}$ and $v = -\frac{\xi_1}{\pi^3}$ and using the Legendre relation, we see that at least one of the two numbers $$\left\{47\zeta(3)\frac{\xi_2}{\pi^3} - \frac{2\xi_1}{3}, \frac{\xi_2}{\pi^3}\right\}$$ is irrational. This result was obtained by Beukers \cite[Theorem~2]{beukers1987irrationality}. \par

The principle of creative telescoping also applies here. Let
\begin{multline*}\boldL := (n + 1)^3 - 2 (2 n + 3) (17 n^2 + 51 n + 45) \textbf{S}_n + 
6 (n + 2) (193 n^2 + 772 n + 781) \textbf{S}^2_{n} \\ - 
2 (2 n + 5) (17 n^2 + 85 n + 113) \textbf{S}^3_{n} + (n + 3)^3 \textbf{S}^4_{n},\end{multline*}
and let $R= \frac{u^5 v + 17 u^4 v^2 - 1224 u^2 v^4 - 5184 u v^5}{(u+8v)^4 (u+9v)^4} \in \M_{-2}^!(\Gamma^+)$. Then one can easily verify (in the sense that we will describe in Section \ref{creative_telescope_section}) that
\begin{equation}\label{eq:level_six_plus_telescoping}\boldL(t^nf) = D^3\left(t^n R\right).\end{equation}
We claim that $\boldL$ annihilates all four sequences $A_n,B_n,C_n,D_n$ defined above. Applying $\boldL$ to both sides of \eqref{aux_14}, we have
$$\boldL(t^n f) = (\boldL A_n) E + (\boldL B_n) f + (\boldL C_n) f^\sharp + D^3(\boldL h_n) = D^3\left(t^n R\right).$$
Since $E,f,f^\sharp$ are independent elements of $\M_4^!(\Gamma^+)/D^3(\M_{-2}^!(\Gamma^+))$, $\boldL$ annihilates $A_n,B_n,C_n$, and $D^3(\boldL h_n) = D^3\left(t^n R\right)$ implies that $\boldL h_n = t^n R$. Taking constant terms gives $\boldL a_0(h_n) = a_0(t^n R)$. Since both $t$ and $R$ vanish at $\infty$, we have $\boldL D_n = 0$. Therefore, $A_n,B_n,C_n,D_n$ are a complete set of solutions to the recurrence $\boldL$. 

Next we give a proof of the denominator types of these four sequences. We will also see a generalization of this method later.

	\begin{lemma}\label{lemma:constant_term_integration_by_parts}
	For $f,g\in\mathbb{C}((q))$ and a positive integer $r$, one has
	$$a_0\bigl(D^r(f)g\bigr)=(-1)^r a_0\bigl(fD^r(g)\bigr).$$
\end{lemma}
\begin{proof}
	This follows immediately from the identity
	\begin{equation*}(D^r f)g+(-1)^{r-1}f(D^r g)
	=D\left(\sum_{i=0}^{r-1}(-1)^iD^{r-1-i}(f)D^i(g)\right)\qedhere.\end{equation*}
\end{proof}

	\begin{lemma}\label{lemma:weight_two_residue_relation}
	Let $\Gamma \subset \SL_2(\mathbb{R})$ be any subgroup commensurable with $\SL_2(\mathbb{Z})$. For each cusp $\beta$ of $\Gamma$, choose $\gamma_\beta\in\SL_2(\mathbb{R})$ with $\gamma_\beta\infty=\beta$, and let $w_\beta>0$ be the corresponding cusp width, so that $f|_2\gamma_\beta$ is $w_\beta$-periodic. Then, for $f\in\M_2^!(\Gamma)$, one has
	\begin{equation}\label{eq:weight_two_residue_relation}
		\sum_{\beta}w_\beta\,a_0(f|_2\gamma_\beta)=0.
	\end{equation}
\end{lemma}
\begin{proof}
	Since $f$ has weight $2$, $2\pi i f(\tau)\,d\tau$ is a meromorphic differential on the compact Riemann surface $X(\Gamma)$. Near the cusp $\beta$, put $q_\beta=e^{2\pi i\tau/w_\beta}$ and write
	$$(f|_2\gamma_\beta)(\tau)=\sum_{m\in \mathbb{Z}}a_{\beta,m}q_\beta^m.$$
	Since $2\pi i\,d\tau=w_\beta\,dq_\beta/q_\beta$, the residue of the differential at $\beta$ is $w_\beta a_{\beta,0}$. The result now follows from the fact that the sum of the residues of a meromorphic differential on a compact Riemann surface is zero.
\end{proof}

As an immediate consequence, we have
\begin{itemize}[leftmargin=*]
\item If $f\in \M_2^!(\Gamma_1(6))$, then $a_0(f) + a_0(f | W_2) + a_0(f | W_3) + a_0(f | W_6) = 0.$
\item If $f\in \M_2^!(\Gamma_1(6),+)$, then $a_0(f) + a_0(f | W_2)= 0.$
\end{itemize}

For $f\in \mathbb{C}((q))$, put
\begin{equation}\label{def:Eichler_int}\E_{k-1}(f):=\sum_{\substack{n\in\mathbb{Z}\\n\neq 0}}\frac{a_n(f)}{n^{k-1}}q^n.\end{equation}
Note that $D^{k-1}\E_{k-1}(f)=f-a_0(f)$.

\begin{proposition}\label{prop:level_six_plus_constant_terms}
For $n\geq 1$, the four sequences $A_n, B_n, C_n, D_n$ defined above are given by
\begin{align*}
A_n&=-\frac15a_0(t^{-n}f),\quad B_n=-13A_n-a_0\bigl(t^{-n}f\E_3(f^\sharp)\bigr),\\
C_n&=a_0\bigl(t^{-n}f\E_3(f)\bigr),\quad D_n=\frac{12}{5}a_0\bigl(t^{-n}f\E_3(E)\bigr).
\end{align*}
In particular, $A_n$ has trivial denominator type, while $B_n,C_n,D_n$ have denominator type $\lcmD{n}^3$.
\end{proposition}
\begin{proof}
It suffices to show that the sequences on the right-hand sides of the above expressions satisfy $\boldL$ and agree with $A_n,B_n,C_n,D_n$ at four consecutive initial values. The initial values are checked directly from the table above. Note that $t|W_2=t^{-1}$ and $f| W_2=f$. Put $\widetilde R:=R| W_2$. Applying $W_2$ to \eqref{eq:level_six_plus_telescoping} gives
$$\boldL(t^{-n}f)=D^3(t^{-n}\widetilde R).$$ Define
$$U_0(n):=a_0(t^{-n}f),\qquad
U_X(n):=a_0\bigl(t^{-n}f\E_3(X)\bigr),\quad X\in\{E,f,f^\sharp\}.$$
Taking constant terms shows that $\boldL U_0=0$. Since $a_0(X)=0$, 
$$\begin{aligned}\boldL U_X(n)&= \boldL a_0(t^{-n}f \E_3(X)) = a_0( \boldL(t^{-n}f) \E_3(X)) \\
	&= a_0(D^3(t^{-n}\widetilde{R}) \E_3(X))\\
	&= a_0(t^{-n}\widetilde R X), \qquad \text{ by Lemma }\ref{lemma:constant_term_integration_by_parts} \\
	&= -a_0(t^n R (X|W_2)), \qquad \text{ by Lemma }\ref{lemma:weight_two_residue_relation}. \end{aligned}$$
	Since $t$ vanishes at $\infty$, the above is $0$ for sufficiently large $n$; here $n\geq 1$ is enough. The claim about denominator types follows immediately from $t\in q(1+q\mathbb Z\llbracket q\rrbracket)$ and $f,E,f^\sharp\in q\mathbb Z\llbracket q\rrbracket$.
\end{proof}

\subsection{Denominator types of $h_n$: a general approach}\label{sec:gamma1-6-denominators}
The main result of this subsection is to prove \ref{cor_denomtype_intro}. The result will not be used in the sequel so can be skipped on the first reading. It extends the technique used in Section~\ref{sec:gamma0-4-weight4} to establish the denominator type of $B_n$ there. In Section~\ref{sec:denominator-types}, we will place this technique in a more general context. Before reading this subsection, the reader might want to read Section~\ref{sec:divisors}, because we will freely use the language of divisors and the notation introduced there. \par
Throughout this subsection, we let $A$ be a subring of $\mathbb{C}$ containing $\overline{\mathbb{Z}}$, the ring of algebraic integers. For a positive integer $N$ and a ring $R$, we use $R_{(N)}=R[1/N]$ for localization by powers of $N$. 
	\begin{lemma}\label{linear_comb_int_lemma}
	Let $\Gamma \subset \SL_2(\mathbb{R})$ be commensurable with $\SL_2(\mathbb{Z})$. Let $f_1,\ldots,f_n$ be linearly independent elements of $\M_k^!(\Gamma)$ with algebraic Fourier coefficients. Let $g \in \M_k^!(\Gamma)$ be a $\mathbb{C}$-linear combination $g = \sum_j c_j f_j$. Assume that all Fourier coefficients of $g$ at $\infty$ are contained in $A$. Then there exists a positive integer $N$, depending only on $f_1,\ldots,f_n$, such that $Nc_j \in A$ for every $j$.
\end{lemma}
\begin{proof}
	Linear independence implies that finitely many Fourier-coefficient functionals give a matrix of rank $n$ on the span of $f_1,\ldots,f_n$. Choose an invertible $n\times n$ submatrix $T$ whose entries are algebraic numbers. Comparing the corresponding coefficients gives $T(c_1,\ldots,c_n)^{\mathsf T}=b$ with $b\in A^n$. Choose a positive integer $N$ such that every entry of $NT^{-1}$ is an algebraic integer. Since $A$ is an $\overline{\mathbb{Z}}$-module, $N(c_1,\ldots,c_n)^{\mathsf T}=NT^{-1}b\in A^n$.
\end{proof}

For the rest of this subsection, let $\Gamma=\Gamma_1(6)$ and $k\geq 3$. Note that if $f\in \M_k^!(\Gamma) \cap \overline{\mathbb{Q}}((q))$, then the same is true of $f|W_2$, $f|W_3$, and $f|W_6$. Also $$W_2 \infty = \frac{1}{3},\quad W_3 \infty = \frac{1}{2},\quad W_6 \infty = 0.$$ We say that $f$ is $A$-integral at the cusp $\infty$ if $f\in A((q))$; $f$ is $A$-integral at the cusp $\frac{1}{2}$ if $f|W_3 \in A((q))$; $f$ is $A$-integral at the cusp $\frac{1}{3}$ if $f|W_2 \in A((q))$; and $f$ is $A$-integral at the cusp $0$ if $f|W_6 \in A((q))$. We say that $f$ is $A$-integral at all cusps if $f,f|W_2,f|W_3,f|W_6$ all belong to $A((q))$. 

We introduce notation that will also be used for other groups. Let $S$ be a finite set of cusps and $\mathcal{D}=\sum_{s\in S}d_s(s)$ an effective divisor. For a weakly holomorphic modular form $X$, write its chosen Fourier expansion at a cusp $s$ as $X_s=\sum_m a_{s,m}q_s^m$, with exponents measured in the local cusp parameter $q_s$. Define the \textit{principal part beyond $\mathcal{D}$} by
$$\mathrm{pp}_{>\mathcal{D}}(X):=\left(\sum_{m<-d_s}a_{s,m}q_s^m\right)_{s\in S}.$$
Saying that $\mathrm{pp}_{>\mathcal{D}}(X)$ has coefficients in $A$ imposes no condition on the coefficients with $m\geq-d_s$. In the present subsection, for $\mathcal{D}=A_0(\frac12)+B_0(\frac13)$, we use $X_{1/2}=X|W_3$ and $X_{1/3}=X|W_2$, both expanded in $q$. For a divisor $\mathcal{D}$, we let $\mathcal{D}_\infty$ be the nonnegative divisor formed from its negative part.

\begin{lemma}[Reduction lemma]\label{lemma:gamma16_reduction}
	Let $X \in \M_0^!(\Gamma)$, with $(\div X)_\infty$ supported on $\{\frac12,\frac13\}$. Let $\mathcal{D}$ be the effective divisor $A_0(\frac12)+B_0(\frac13)$. Assume that $\mathrm{pp}_{>\mathcal{D}}(X)$ has coefficients in $\overline{\mathbb{Z}}$. Then \\ (a) $X$ can be written as $X = G+\kappa$, where $G$ is $\overline{\mathbb{Z}}_{(2)}$-integral at all cusps and $\kappa\in \mathcal{L}_0( \mathcal{D})$. \\
	(b) $X$ can also be written as $X = G+\kappa$, with possibly different $G$ and $\kappa$, such that $G$ is $\overline{\mathbb{Z}}_{(3)}$-integral at all cusps and $\kappa\in \mathcal{L}_0( \mathcal{D})$.
\end{lemma}
\begin{proof}
The key point of the proof is to construct enough eta-quotients with desired integrality properties at the cusps. We abbreviate
$$\eta_6(r_1,r_2,r_3,r_6) := \prod_{j\mid 6} \eta(j \tau)^{r_j}.$$
The following table lists rational numbers $\mu$ such that $f|W \in \mu \mathbb{Z}((q))^\times$ for various $f\in \M_0^!(\Gamma_1(6))$ and $W\in \{I,W_2,W_3,W_6\}$. For example,
$$\eta_6(3,-3,-9,9)|W_2  = \frac{-1}{8} \left(q^{-1} + 3 + 6 q + O(q^2)\right).$$
\begin{table}[H]
	\begin{tabular}{|c|c|c|c|c|}
		\hline
		 & $W_6$ & $W_2$       & $W_3$ & $W_1 = I$ \\ \hline
		$\eta_6(3,-3,-9,9) \sim (\infty)-(\frac13)$ & $\frac18$      & $\frac18$ & $1$       & $1$ \\ \hline
		$\eta_6(9,-9,-3,3)\sim (0)-(\frac12)$       & 8        & 8         & 1         & 1         \\ \hline
		$\eta_6(8,-4,-8,4)\sim (0)-(\frac13)$       & 9        & 1         & 9         & 1         \\ \hline
		$\eta_6(4,-8,-4,8)\sim (\infty)-(\frac12)$  & $\frac19$        & 1 & $\frac19$ & 1         \\ \hline
	\end{tabular}
	\caption{Rational numbers $\mu$ such that $f|W \in \mu \mathbb{Z}((q))^\times$.}
	\label{Gamma0(6)eta_content}
\end{table}
Let $(\div X)_\infty = A(\frac12) + B(\frac13)$. There is nothing to prove when $X\in \mathcal{L}_0(\mathcal{D})$: we can simply take $G = 0$ and $\kappa = X$. We therefore assume that $A>A_0$ or $B>B_0$.  \\
(a) Let $f_1 = \eta_6(3,-3,-9,9)$ and $f_2 = \eta_6(9,-9,-3,3)$. They have simple poles at $\frac13$ and $\frac12$, respectively. Therefore, by matching principal parts, we see that there exists $P \in \mathbb{C}[X_1,X_2]$ such that $X-P(f_1,f_2) \in \mathcal{L}_0(\mathcal{D})$. By the hypothesis that $\mathrm{pp}_{>\mathcal{D}}(X)$ has coefficients in $\overline{\mathbb{Z}}$, we see that $P \in \overline{\mathbb{Z}}_{(2)}[X_1,X_2]$. Thus, we can take $G = P(f_1,f_2)$ and $\kappa = X-P(f_1,f_2)$.\\
(b) The same argument applies to $\eta_6(8,-4,-8,4)$ and $\eta_6(4,-8,-4,8)$: each has a simple pole at one of the two cusps and no other poles, and they are $\mathbb{Z}_{(3)}$-integral. 
\end{proof}
	
 Our goal in this subsection is to prove the following:
\begin{proposition}\label{prop:gamma16_denom_type}
Let $\mathfrak{B}$ be a basis of $\M_k^!(\Gamma) \cap \overline{\mathbb{Q}}((q))$ modulo $D^{k-1}(\M_{2-k}^!(\Gamma))$. Let $F \in \M_k^!(\Gamma)$ with pole divisor $(\div F)_\infty = A(\frac12)+B(\frac13)$. Write
\begin{equation}\label{aux_21}F = \sum_{f\in \mathfrak{B}} c_f f + D^{k-1}(h),\qquad h\in \M_{2-k}^!(\Gamma).\end{equation} Let $M = \max(A,B)$. Assume that $F$ is $\overline{\mathbb{Z}}$-integral at the cusps $\frac12$ and $\frac13$. Then there exists a positive integer $\nu$, depending only on $k$ and $\mathfrak{B}$ and independent of $F$, such that
\begin{itemize}[leftmargin=*]
\item $\nu \lcmD{M}^{k-1} h$ is $\overline{\mathbb{Z}}$-integral at all cusps.
\item $\nu \lcmD{M}^{k-1} c_f \in \overline{\mathbb{Z}}$.
\end{itemize}
\end{proposition}
\begin{proof} Note that it suffices to prove the statement for one $\mathfrak{B}$: choosing any other basis will only alter the constant $\nu$. Fix an effective divisor $\mathcal{D} = A_0(\frac12)+B_0(\frac13)$ such that $\mathfrak{B}\subset \mathcal{L}_{k}(\mathcal{D})$. We proceed in steps. We will use $\nu_1,\nu_2,\ldots \in \mathbb{Z}_{>0}$ to denote integers that depend only on $k$ and $\mathfrak{B}$ (i.e., are independent of $F$).\\
	\underline{Step 1: the origin of the LCM $\lcmD{M}$.}  Let $\gamma \in \{W_2,W_3\}$ and write
	\begin{equation}\label{aux_30}h|_{2-k}\gamma := a_0 + \sum_{n\neq 0} \frac{a_n}{n^{k-1}} q^n.\end{equation}
	Then $$ \sum_{n\neq 0} a_n q^n = F|_k \gamma -\sum_{f\in \mathfrak{B}} c_f f|_k \gamma.$$
	By assumption, $F$ is integral at the two cusps $\frac12$ and $\frac13$. Hence, $a_n \in\overline{\mathbb{Z}}$ for $n<-B_0$ when $\gamma=W_2$ and for $n<-A_0$ when $\gamma=W_3$. Any nonzero coefficient in these ranges has $n\geq -M$. Using $\lcmD{M}^{k-1}$ to cancel the denominators in \eqref{aux_30}, we see that $\mathrm{pp}_{>\mathcal{D}}(\lcmD{M}^{k-1}h)$ has coefficients in $\overline{\mathbb{Z}}$. \\
	\underline{Step 2.} Take $g=c(u+8v)^{k-2}$, where the positive integer $c$ is chosen so that $g$ is integral at all cusps. Its divisor is $(k-2)(\frac13)$, so $g$ has no zeros in $\mathbb{H}$ and both $g$ and $g^{-1}$ have bounded denominators at every cusp. For the fixed effective divisor $\mathcal{D}_1=\mathcal{D}+\div g$, the principal part $\mathrm{pp}_{>\mathcal{D}_1}(\lcmD{M}^{k-1}gh)$ has coefficients in $\overline{\mathbb{Z}}$, and the poles of $\lcmD{M}^{k-1}gh$ are supported on $\{\frac12,\frac13\}$. Lemma \ref{lemma:gamma16_reduction} therefore gives $G$, $\overline{\mathbb{Z}}_{(2)}$-integral at all cusps, and $\kappa\in\mathcal{L}_0(\mathcal{D}_1)$ such that $$\lcmD{M}^{k-1}g h = G + \kappa.$$
	\underline{Step 3.} Multiply both sides of the above equation by $g^{-1}$, apply $D^{k-1}$, and then use \eqref{aux_21} to obtain
	$$\begin{aligned}\lcmD{M}^{k-1} F - D^{k-1}(Gg^{-1}) &= \sum_{f\in \mathfrak{B}} \lcmD{M}^{k-1} c_f f + D^{k-1}(\kappa g^{-1}) \\ 
		&= \sum_{f\in \mathfrak{B}} \lcmD{M}^{k-1} c_f f + \sum_{j=1}^r \theta_j D^{k-1}(\kappa_j g^{-1}),
	\end{aligned}$$ 
	where, in the last line, we chose $\kappa_1,\ldots,\kappa_r$ as a basis of $\mathcal{L}_0(\mathcal{D}_1)$ with algebraic Fourier coefficients and wrote $\kappa = \sum_j \theta_j \kappa_j$. After applying $W_3$, the left-hand side has Fourier coefficients in $\nu_2^{-1}\overline{\mathbb{Z}}_{(2)}$, since $F$ is integral at $\frac12$, $G$ is $\overline{\mathbb{Z}}_{(2)}$-integral there, and $g^{-1}|_{2-k}W_3$ has bounded denominators. By definition of $\mathfrak{B}$, the forms $\{f\mid f\in \mathfrak{B}\}$ and $\{D^{k-1}(\kappa_1g^{-1}),\ldots,D^{k-1}(\kappa_rg^{-1})\}$ are linearly independent. Applying Lemma \ref{linear_comb_int_lemma} to their transforms by $W_3$ shows that there exists $\nu_3\in \mathbb{Z}_{>0}$ such that $$\lcmD{M}^{k-1} c_f \in \overline{\mathbb{Z}}_{(2)}/\nu_3, \qquad \theta_j \in  \overline{\mathbb{Z}}_{(2)}/\nu_3.$$
	Note that this $\nu_3$ is independent of $F$. Also, writing $\lcmD{M}^{k-1} h = Gg^{-1} + \kappa g^{-1}$, we see that $\nu_4\lcmD{M}^{k-1} h$ is $\overline{\mathbb{Z}}_{(2)}$-integral at all cusps. Therefore, we have proved the two statements in the bullets with $\overline{\mathbb{Z}}$ replaced by $\overline{\mathbb{Z}}_{(2)}$. \\
	\underline{Step 4.} Using part (b) of Lemma \ref{lemma:gamma16_reduction}, one can repeat the previous step with $G$ being $\overline{\mathbb{Z}}_{(3)}$-integral. Since $2$ and $3$ are coprime, this implies the conclusion.
\end{proof}

\begin{corollary}\label{Gamma1(6)_denomtype_cor}
Let $t$ be as in \eqref{t_choice_level6} and let $g\in \M_k(\Gamma_1(6)) \cap \overline{\mathbb{Q}}((q))$. Write
$$t^n g = \sum_{f\in \mathfrak{B}} c_{f,n} f + D^{k-1}(h_n),\qquad h_n\in \M_{2-k}^!(\Gamma).$$ Then there exists an integer $\nu$, independent of $n$, such that
\begin{itemize}[leftmargin=*]
	\item $\nu \lcmD{n}^{k-1} h_n$ is $\overline{\mathbb{Z}}$-integral at all cusps.
	\item $\nu \lcmD{n}^{k-1} c_{f,n} \in \overline{\mathbb{Z}}$.
\end{itemize}
\end{corollary}
\begin{proof}
This follows immediately from the preceding proposition, since $(\div(t^ng))_\infty \leq n(\frac12)+n(\frac13)$, $t\in q(1+q\mathbb{Z}[[q]])$, and $t|W_2 = t|W_3 = t^{-1}$. A fixed nonzero integer clears the denominators of $g$ at all cusps.
\end{proof}

\begin{corollary}\label{cor:ABCD_denom_type}
The Ap\'ery sequence $\Bp_n$ in \eqref{aux_11} and the sequences $B_n,C_n,D_n$ in \eqref{aux_14} all have denominator types $\lcmD{n}^3$.
\end{corollary}
\begin{proof}
Take $k=4$ in the preceding corollary. Choose $\mathfrak{B}$ to consist of elements in the plus and minus $W_6$-eigenspaces. The sequence in \eqref{aux_11} comes from the minus space, and those in \eqref{aux_14} come from the plus space.
\end{proof}

	\section{Divisor of automorphic forms on $\Gamma\backslash \mathbb{H}^\ast$}\label{sec:divisors}
	Let $\Gamma$ be a subgroup of $\SL_2(\mathbb{R})$ commensurable with $\SL_2(\mathbb{Z})$. The locally compact quotient $\Gamma\backslash \mathbb{H}$ can be compactified to $X(\Gamma) := \Gamma\backslash \mathbb{H}^\ast$ by adding a finite number of cusps. Let $k\in \mathbb{Z}$, which we will use chiefly to denote weights of various modular forms. We employ the following notation:
	\begin{itemize}[leftmargin=*]
	\item $\mathcal{A}_k(\Gamma)$: the space of weight $k$ meromorphic modular forms on $\Gamma$;
	\item $\theta_k$: a fixed nonzero element of $\mathcal{A}_k(\Gamma)$;
	\item $g$: the genus of the compact Riemann surface $\Gamma\backslash \mathbb{H}^\ast$;
\item $\varepsilon_2, \varepsilon_3$: the numbers of inequivalent elliptic points of orders $2$ and $3$;
\item $\varepsilon_\infty, \varepsilon^{\text{reg}}_\infty, \varepsilon^{\text{irr}}_\infty$: the numbers of inequivalent cusps, regular cusps, and irregular cusps, respectively;
	\item $S$: a nonempty finite subset of $X(\Gamma)$, which we assume to contain all cusps.
	
	\end{itemize}

	For $f\in \mathcal{A}_k(\Gamma)$, define its associated divisor, a formal rational combination of points in $\Gamma\backslash \mathbb{H}^\ast$, by
	$$\div f := \sum_{P \in \Gamma\backslash \mathbb{H}^\ast} v_P(f) P.$$ Because of its importance, we recapitulate the full definition of $v_P(f)$ here, following closely the treatment of \cite[Chapter~2]{shimura1971introduction} or \cite[Chapter~3]{diamond2005first}.
	\begin{itemize}[leftmargin=*]
	\item For $P\in \mathbb{H}$ that is not an elliptic point of $\Gamma$: $v_P(f)$ is the order of the meromorphic function $f$ in a neighbourhood of $P$.
	\item For $P\in \mathbb{H}$ that is an elliptic point of $\Gamma$: $v_P(f)$ is the order (as above) divided by the order of the elliptic point.
	\item For $P\in X(\Gamma)\setminus(\Gamma\backslash\mathbb{H})$ corresponding to a cusp $s$: choose an element $\rho \in \text{SL}_2(\mathbb{R})$ such that $\rho(s) = \infty$. Putting $\Gamma_s = \{\gamma \in \Gamma \mid \gamma(s) = s\}$, we have $$\rho \Gamma_s \rho^{-1} \cdot \{\pm 1\} = \left\{\pm  \SmallMatrix{1 & h \\ 0 & 1}^m  \Bigg| m\in \mathbb{Z}\right\}$$
	for some positive integer $h$. \begin{itemize}[leftmargin=*]
	\item $k$ is even. Since $(f|_k \rho^{-1})(\tau)$ is invariant under $\tau \mapsto \tau+h$, there exists a meromorphic function $\Phi$ near the origin such that $(f|_k \rho^{-1})(\tau) = \Phi(e^{2\pi i \tau/h})$. The order $v_P(f)$ is defined as the order of $\Phi$ at the origin, namely $v_0(\Phi)$.
	\item $k$ is odd. In this case $\rho \Gamma_s \rho^{-1}$ is generated either by $\SmallMatrix{1 & h \\ 0 & 1}$ or by $- \SmallMatrix{1 & h \\ 0 & 1}$. In the former case (a regular cusp), $f|_k \rho^{-1}$ is invariant under $\tau \mapsto \tau+h$, and $v_P(f)$ is defined in the same way as above. In the latter case (an irregular cusp), $f|_k \rho^{-1}$ is negated under $\tau \mapsto \tau+h$. Hence, there exists an odd function $\Psi$ such that $(f|_k \rho^{-1})(\tau) = \Psi(e^{\pi i \tau/h})$; we define $v_P(f)$ to be $v_0(\Psi)/2$.
	\end{itemize}
	\end{itemize}
	Note that $v_P(f)$ can be nonintegral only at elliptic points or irregular cusps. For a divisor $\mathcal{D}$, set $$\mathcal{L}_k(\mathcal{D}) := \{f\in \mathcal{A}_k(\Gamma) \mid \div(f) \geq -\mathcal{D}\} \cup \{0\}.$$
	This is a finite-dimensional complex vector space; denote its dimension by $\ell_k(\mathcal{D}) := \dim \mathcal{L}_k(\mathcal{D})$.
	After dividing by our fixed nonzero element $\theta_k \in \mathcal{A}_k(\Gamma)$, we have
	$$\ell_k(\mathcal{D}) = \ell_0(\mathcal{D} + \div(\theta_k )) = \ell_0(\lfloor \mathcal{D} + \div(\theta_k )\rfloor),$$
	where the floor is applied componentwise to yield an integral divisor. 
	
			\begin{lemma}\label{lemma_k_2-k_dimension}
		For $k\geq 2$, we have \begin{equation}\label{aux_1}\deg (\lfloor \div \theta_k \rfloor) -  \deg (\lfloor \div \theta_{2-k} \rfloor) =  \dim \E_k(\Gamma) + 2\dim \S_k(\Gamma) - \delta_{k,2}.\end{equation}
	\end{lemma}
	\begin{proof}
		The proof consists only of substituting various formulas from Shimura \cite[Chapter~2]{shimura1971introduction}. Let $e_1,\ldots, e_r$ be the orders of the inequivalent elliptic points of $\Gamma\backslash \mathbb{H}^\ast$. First, we have\footnote{Here the even case is found at the end of page 45 of \cite{shimura1971introduction}, and the odd case in the middle of page 47.}
		$$\deg (\lfloor \div \theta_k \rfloor) = \begin{cases}
			\frac{k}{2}(2g-2 +\varepsilon_\infty) + \sum_{j=1}^r \lfloor \frac{k(e_j-1)}{2e_j}\rfloor, & \qquad k\in 2\mathbb{Z},\\
			\frac{k}{2}(2g-2) + \frac{k}{2}\varepsilon^{\text{reg}}_\infty + \frac{k-1}{2}\varepsilon^{\text{irr}}_\infty + \sum_{j=1}^r \lfloor \frac{k(e_j-1)}{2e_j}\rfloor, & \qquad k\in 2\mathbb{Z}+1.
		\end{cases}$$ 
		Now we recall various dimension formulas for modular forms \cite[Theorems~2.24,2.25]{shimura1971introduction}. When $k=2$, we have $\dim \E_k(\Gamma) = \varepsilon_\infty-1$ and $\dim \S_k(\Gamma) = g$. It is then easy to check that \eqref{aux_1} holds. When $k\geq 4$ is even, $$\dim \E_k(\Gamma) = \varepsilon_\infty,\quad \dim \S_k(\Gamma) = (k-1)(g-1) + (\frac{k}{2}-1)\varepsilon_\infty +\sum_{j=1}^r \lfloor \frac{k(e_j-1)}{2e_j}\rfloor.$$
		When $k\geq 3$ is odd, $$\dim \E_k(\Gamma) = \varepsilon^{\text{reg}}_\infty,\quad \dim \S_k(\Gamma) = (k-1)(g-1) + (\frac{k}{2}-1)\varepsilon^{\text{reg}}_\infty + \frac{k-1}{2}\varepsilon^{\text{irr}}_\infty + \sum_{j=1}^r \lfloor \frac{k(e_j-1)}{2e_j}\rfloor.$$
		Note that when $k$ is odd, the $e_j$ must also be odd \cite[Corollary~1.21]{shimura1971introduction}. The following elementary observation about the floor function then completes the proof:
		\begin{itemize}[leftmargin=*]
			\item For $k\in 2\mathbb{Z}$, $e\in \mathbb{N}$, we have $ \lfloor \frac{k(e-1)}{2e}\rfloor +  \lfloor \frac{(2-k)(e-1)}{2e}\rfloor = 0$.
			\item For $k\in 2\mathbb{Z}+1$, $e\in 2\mathbb{N}+1$, we have $ \lfloor \frac{k(e-1)}{2e}\rfloor +  \lfloor \frac{(2-k)(e-1)}{2e}\rfloor = 0$.
		\end{itemize} 
	\end{proof}
	
		\begin{lemma}\label{lemma_elliptic_order_congruence}
		Let $x\in\Gamma\backslash\mathbb{H}$. Then, for $f\in \mathcal{A}_k(\Gamma)$, we have
		$$v_x(f)\equiv\frac{k(e_x-1)}{2e_x}\pmod{\mathbb{Z}}.$$
	\end{lemma}
	\begin{proof}
		If $e_x = 1$, i.e., $x$ is not an elliptic point, there is nothing to prove. Thus, assume $e := e_x >1$. Introduce the coordinate
		$$w=\frac{\tau-x}{\tau-\overline{x}},\qquad
		\tau=\frac{x-\overline{x}w}{1-w}.$$
		If $\gamma=\begin{psmallmatrix}a&b\\c&d\end{psmallmatrix}$ generates the stabilizer of $x$ modulo scalar matrices, put
		$$j_x:=cx+d,\qquad \zeta_x:=j_x^{-2}.$$
		One easily checks that $w(\gamma\tau)=\zeta_x w(\tau)$. Thus, $\zeta_x$ is a primitive $e$-th root of unity. Consider $$\widetilde{f}(w) := (1-w)^{-k}f\left(\frac{x-\overline{x}w}{1-w}\right).$$
		Since $f$ has weight $k$, $$\widetilde f(\zeta_x w)=j_x^k\widetilde f(w).$$
		If $\widetilde f(w)=\sum_{m\in \mathbb{Z}} a_mw^m$, comparison of coefficients shows that
		$$a_m\neq 0 \implies \zeta_x^m=j_x^k \iff j_x^{2m+k} = 1.$$
		When $k$ is even, the fact that $j_x^2$ is a primitive $e$-th root of unity implies $m+k/2 \equiv 0 \pmod{e}$. When $k$ is odd, $e$ is odd and $j_x$ is a primitive $e$-th root of unity. Hence $2m+k\equiv 0 \pmod{e} \implies m\equiv k\frac{e-1}{2} \pmod{e}$. Taking $m=e v_x(f)$ in both cases gives the conclusion.
	\end{proof}

	For a positive integer $e$, define
	$$\mu_k(e):=\left\lfloor\frac{(2-k)(e-1)}{2e}\right\rfloor-\left\lfloor\frac{(2-k)(e-1)}{2e}-\frac{k-1}{e}\right\rfloor.$$
	For the given finite $S\subset X(\Gamma)$, define the following divisor
	$$S^\circ := \sum_{x\in S\cap (\Gamma \backslash \mathbb{H})} \frac{x}{e_x}.$$
	\begin{lemma}\label{lemma_k_2-k_dimension_S}
	For $k\geq 2$, we have
	$$\deg\lfloor\div\theta_k\rfloor-\deg\left\lfloor\div\theta_{2-k}-(k-1)S^\circ\right\rfloor+\delta_{k,2} = \dim\E_k(\Gamma)+2\dim\S_k(\Gamma)+\sum_{x\in S\cap(\Gamma\backslash\mathbb{H})}\mu_k(e_x).$$
	\end{lemma}
	\begin{proof}
	Note that this reduces to Lemma \ref{lemma_k_2-k_dimension} when $S$ has no interior points in $\mathbb{H}$. To prove this, start with Lemma~\ref{lemma_k_2-k_dimension}, which says that the quantity in question is
	$$
	\dim\E_k(\Gamma)+2\dim\S_k(\Gamma)+\deg\lfloor\div\theta_{2-k}\rfloor-\deg\left\lfloor\div\theta_{2-k}-(k-1)S^\circ\right\rfloor.$$
	By Lemma~\ref{lemma_elliptic_order_congruence},
	$$v_x(\theta_{2-k})\equiv\frac{(2-k)(e_x-1)}{2e_x}\pmod{\mathbb{Z}}.$$
	Consequently,
	\begin{align*}
		\deg\lfloor\div\theta_{2-k}\rfloor-\deg\left\lfloor\div\theta_{2-k}-(k-1)S^\circ\right\rfloor
		&=\sum_{x\in S\cap(\Gamma\backslash\mathbb{H})}\left(\lfloor v_x(\theta_{2-k})\rfloor-\left\lfloor v_x(\theta_{2-k})-\frac{k-1}{e_x}\right\rfloor\right)\\
		&=\sum_{x\in S\cap(\Gamma\backslash\mathbb{H})}\mu_k(e_x). \qedhere
	\end{align*}
	\end{proof}
	
	Recall the following consequence of the Riemann--Roch theorem, which will be of paramount importance to us.
	\begin{lemma}\cite[Proposition~2.14]{shimura1971introduction}\label{lemma_riemann_roch}
		Let $\mathcal{D}$ be an integral divisor of $\Gamma \backslash \mathbb{H}^\ast$. If $\deg(\mathcal{D}) > 2g-2$, then $$\ell_0(\mathcal{D}) = \deg(\mathcal{D})-g+1.$$
	\end{lemma}
	
	%We say a divisor $\mathcal{D}$ is \textit{cuspidal} if it is supported on the cusps of $\Gamma\backslash \mathbb{H}^\ast$.
	
	\begin{lemma}\label{lemma_D_L}
		Let $\mathcal{D}$ be a divisor of $X(\Gamma)$ supported on $S$. Then, \begin{enumerate}[leftmargin=*]
			\item $D^{k-1}(\mathcal{L}_{2-k}(\mathcal{D})) \subset \mathcal{L}_{k}(\mathcal{D} + (k-1)S^\circ)$.
			\item Assume that all coefficients of $\mathcal{D}$ at points of $S$ are positive. If $f\in \M_{2-k}^S(\Gamma)$ is such that $D^{k-1}(f) \in \mathcal{L}_{k}(\mathcal{D} + (k-1)S^\circ)$, then $f\in \mathcal{L}_{2-k}(\mathcal{D})$. 
		\end{enumerate}
	\end{lemma}
	\begin{proof}
		Let $f\in\mathcal{L}_{2-k}(\mathcal{D})$. Bol's identity shows that $D^{k-1}(f)$ is a meromorphic modular form of weight $k$. We first consider the cusp $\infty$. Write $f = \sum_{n\in \mathbb{Z}} a_n e^{2\pi i n\tau/w}$, where $w$ is the width at $\infty$. Then
		\begin{equation}\label{aux_2}D^{k-1}(f) =  \frac{1}{w^{k-1}} \sum_{\substack{n\in \mathbb{Z} \\ n\neq 0}} a_n n^{k-1} e^{2\pi i n\tau/w}.\end{equation}
		Thus $v_\infty(D^{k-1}(f)) \geq v_\infty (f)$, and equality holds if and only if $v_\infty(f)\neq 0$. The situation at other cusps can be translated to that of $\infty$ by using Bol's identity $D^{k-1}(f|_{2-k} \gamma) = (D^{k-1}f)|_k \gamma$; at an irregular cusp, the same calculation is made in the corresponding half-integral local parameter. At an interior point $x$, differentiation of a local Laurent expansion gives
		$$v_x(D^{k-1}f)\geq v_x(f)-\frac{k-1}{e_x}.$$
		Moreover, if $f$ is holomorphic at $x$, then so is $D^{k-1}f$. It follows that $D^{k-1}f$ has no poles outside $S$ and that
		$$\div(D^{k-1}f)\geq-\mathcal{D}-(k-1)S^\circ,$$
		which proves (1).

		For (2), let $x\in S\cap(\Gamma\backslash\mathbb{H})$. If $f$ has a pole at $x$, differentiation of its leading Laurent term gives the equality
		$$v_x(D^{k-1}f)=v_x(f)-\frac{k-1}{e_x}.$$
		The assumed lower bound for $v_x(D^{k-1}f)$ therefore implies $v_x(f)\geq-\mathcal{D}(x)$. If $f$ is holomorphic at $x$, the same conclusion follows from the positivity of $\mathcal{D}(x)$. At a cusp $s\in S$, equation \eqref{aux_2} and Bol's identity similarly show that
		$$v_s(D^{k-1}f)=v_s(f)$$
		whenever $f$ has a pole at $s$; the holomorphic case again follows from the positivity of $\mathcal{D}(s)$. Since $f$ has no poles outside $S$, we conclude that $\div(f)\geq-\mathcal{D}$, proving (2).
	\end{proof}
	
Next we restate and prove Theorem \ref{dim_theorem} in the introduction.
	\begin{theorem}
	For $k\geq 2$, the quotient $$\frac{\M_k^S(\Gamma)}{D^{k-1}(\M_{2-k}^S(\Gamma))}$$ is finite-dimensional, with dimension
	\begin{equation}\label{aux_20}\dim\E_k(\Gamma)+2\dim\S_k(\Gamma)+\sum_{x\in S\cap(\Gamma\backslash\mathbb{H})}\mu_k(e_x).\end{equation}
	\end{theorem}
	\begin{proof}
	For $N\in\mathbb{N}$, let
	$$\mathcal{D}_N:=\sum_{s\in S}Ns.$$
	For sufficiently large $N$, the divisor $\mathcal{D}_N-(k-1)S^\circ$ has support $S$. Lemma~\ref{lemma_D_L} gives
	$$D^{k-1}\left(\mathcal{L}_{2-k}(\mathcal{D}_N-(k-1)S^\circ)\right)\subset\mathcal{L}_k(\mathcal{D}_N).$$ We may therefore consider the quotient
	$$Q_N:=\frac{\mathcal{L}_k(\mathcal{D}_N)}{D^{k-1}(\mathcal{L}_{2-k}(\mathcal{D}_N-(k-1)S^\circ))}.$$
	The kernel of $D^{k-1}$ on $\M_{2-k}^S(\Gamma)$ is zero when $k\geq 3$ and consists of the constants when $k=2$. It follows that
	$$\dim Q_N=\ell_k(\mathcal{D}_N)-\ell_{2-k}(\mathcal{D}_N-(k-1)S^\circ)+\delta_{k,2}.$$
	Using the fixed nonzero forms $\theta_k\in\mathcal{A}_k(\Gamma)$ and $\theta_{2-k}\in\mathcal{A}_{2-k}(\Gamma)$, this becomes
	$$\dim Q_N={}\ell_0\left(\mathcal{D}_N+\lfloor\div\theta_k\rfloor\right)
	-\ell_0\left(\mathcal{D}_N+\left\lfloor\div\theta_{2-k}-(k-1)S^\circ\right\rfloor\right)+\delta_{k,2}.$$
	For sufficiently large $N$, Riemann--Roch (Lemma \ref{lemma_riemann_roch}) gives
	$$\dim Q_N=\deg\lfloor\div\theta_k\rfloor-\deg\left\lfloor\div\theta_{2-k}-(k-1)S^\circ\right\rfloor+\delta_{k,2},$$
	which is independent of $N$. By Lemma~\ref{lemma_k_2-k_dimension_S}, this is exactly the quantity on the right-hand side of \eqref{aux_20}.
	Hence, $d:= \dim Q_N$ equals \eqref{aux_20} when $N$ is sufficiently large. For any given $f\in \M_k^S(\Gamma)$, we have $f\in \mathcal{L}_k(\mathcal{D}_N)$ for some $N$, so we have proved
	$$\dim\frac{\M_k^S(\Gamma)}{D^{k-1}(\M_{2-k}^S(\Gamma))} \leq d.$$
	
	Next we argue that the dimension is exactly $d$. Let $N$ be sufficiently large and fix elements $f_1,\ldots,f_d \in \mathcal{L}_k(\mathcal{D}_N)$ whose classes form a basis of $Q_N$. It remains to show that $f_1,\ldots,f_d$ are linearly independent modulo $D^{k-1}(\mathcal{M}_{2-k}^S(\Gamma))$. Let $h \in \mathcal{M}_{2-k}^S(\Gamma)$ be such that
	$\sum_{j=1}^d c_j f_j = D^{k-1}(h)$. The left-hand side is in $\mathcal{L}_k(\mathcal{D}_N)$. Since $\mathcal{D}_N-(k-1)S^\circ$ is positive for sufficiently large $N$, part (2) of Lemma \ref{lemma_D_L} implies that $h\in \mathcal{L}_{2-k}(\mathcal{D}_N - (k-1)S^\circ)$. Thus, $c_j = 0$ by our choice of $f_j$, completing the proof.
	\end{proof}

	\section{Creative telescoping}\label{creative_telescope_section}
	We continue to use the notation introduced at the beginning of the previous section. Throughout this section, we let $t$ be a nonconstant element of $\M_0^S(\Gamma)$ whose divisor is supported on $S$. \par
	The space of meromorphic modular forms $\M_k^S(\Gamma)$ is a $\mathbb{C}$-vector space. Letting $n$ be an indeterminate,\footnote{Sometimes we will specialize $n$ to an integer. It will always be clear from context which interpretation of $n$ is intended.} we shall consider the $\mathbb{C}(n)$-vector space $\M_k^S(\Gamma) \otimes_\mathbb{C} \mathbb{C}(n)$. We will still use $D^{k-1}$ to denote the map obtained from $\M_{2-k}^S(\Gamma) \to  \M_{k}^S(\Gamma)$ by tensoring with $\mathbb{C}(n)$.
	\subsection{Existence of creative telescoping and proof of Theorem \ref{creative_telescoping_S_intro}}
	We need some preparation results, which will also be important conceptually.
	\begin{proposition}\label{Dt^n_continuation}
	Let $k\geq 2$ and $h\in \mathcal{L}_{2-k}(\mathcal{D})$ for a divisor $\mathcal{D}$ supported on $S$. Let $f_1,\ldots,f_r$ be a basis of $\mathcal{L}_k(\mathcal{D}+(k-1)S^\circ)$. Then there exist polynomials $p_1,\ldots,p_r\in \mathbb{C}[n]$ such that, for every $n\in \mathbb{Z}$, we have
	\begin{equation}\label{aux_16}\frac{D^{k-1}(t^n h)}{t^n} = p_1(n) f_1 + \cdots + p_r(n) f_r.\end{equation}
	\end{proposition}
	\begin{proof}
	Let $\mathcal{D}_0 = -\div t$, which is supported on $S$ by assumption. Then $t^n h \in \mathcal{L}_{2-k}(n\mathcal{D}_0 + \mathcal{D})$. Part (1) of Lemma \ref{lemma_D_L} implies that $D^{k-1}(t^n h) \in \mathcal{L}_k(n\mathcal{D}_0 + \mathcal{D}+(k-1)S^\circ)$, and hence $$\frac{D^{k-1}(t^n h)}{t^n} \in \mathcal{L}_k(\mathcal{D}+(k-1)S^\circ),\qquad \forall n\in \mathbb{Z}.$$
	Therefore, there exist functions $p_i\colon \mathbb{Z}\to \mathbb{C}$ such that \eqref{aux_16} holds. It remains to show that these $p_i$ are polynomials in $n$. To do this, we consider the $q$-expansion at the cusp $\infty$. Without loss of generality, we assume that this cusp has width $1$, so $t, h \in \mathbb{C}(\!(q)\!)$. When $t$ and $h$ are interpreted as formal Laurent series, we claim that
	\begin{equation}\label{aux_15}\frac{D^{k-1}(t^n h)}{t^n} \in \mathbb{C}[n](\!(q)\!),\end{equation}
	i.e., the $q$-coefficients of the left-hand side are polynomials in $n$. Indeed, write $a = \frac{Dt}{t} \in \mathbb{C}(\!(q)\!)$. Then
	$$\frac{D(t^nh)}{t^n} = Dh + n \frac{Dt}{t}h = (D+na)h.$$
	Note that $D$ and multiplication by $a$ need not commute. Nonetheless, we have $$\frac{D^{k-1}(t^nh)}{t^n} = (D+na)^{k-1}h.$$
	Thus, each individual coefficient is a polynomial in $n$ of degree at most $k-1$. The left-hand side of \eqref{aux_16}, when expanded as a $q$-series, has coefficients in $\mathbb{C}[n]$, while each $f_i$ has $q$-coefficients in $\mathbb{C}$. Hence, the $p_i$ must be polynomials in $n$. 
	\end{proof}

	The proposition gives a way to analytically continue the expression $\frac{D^{k-1}(t^n h)}{t^n}$ with respect to $n$, which is \textit{a priori} defined only for $n\in \mathbb{Z}$: the right-hand side of \eqref{aux_16} makes sense for any complex $n$, or even when $n$ is viewed as an indeterminate. The proposition's conclusion can be stated more succinctly as $$\frac{D^{k-1}(t^n h)}{t^n} \in \mathcal{L}_k(\mathcal{D}+(k-1)S^\circ)\otimes_\mathbb{C} \mathbb{C}[n].$$
	Write $K = \mathbb{C}(n)$. Extending this $K$-linearly with respect to $h$, we have the following containment of $K$-vector spaces:
	\begin{equation}\label{quotient_D_space_CT}\left\{ \frac{D^{k-1}(t^n h)}{t^n} \Bigg| h\in \mathcal{L}_{2-k}(\mathcal{D})\otimes_\mathbb{C} K \right\} \subset \mathcal{L}_k(\mathcal{D}+(k-1)S^\circ)\otimes_\mathbb{C} K.\end{equation}
	Let $\mathcal{H}_t(\mathcal{D})$ denote the quotient of the right-hand side by the left-hand side. \par
	The $K$-dimension of the right-hand side is $\dim_\mathbb{C} \mathcal{L}_k(\mathcal{D}+(k-1)S^\circ) = \ell_k(\mathcal{D}+(k-1)S^\circ)$. We claim that the $K$-dimension of the left-hand side is $\ell_{2-k}(\mathcal{D})$. This amounts to the fact that $h \mapsto \frac{D^{k-1}(t^n h)}{t^n}$ is injective. When $k\geq 3$, this is immediate: if $\frac{D^{k-1}(t^n h)}{t^n} = 0$ identically in $n$, then $h = 0$. When $k=2$, the same identity implies that $t^n h$ is constant for every integer $n$, which forces $h=0$ because $t$ is nonconstant. \par
	Therefore $\dim_K \mathcal{H}_t(\mathcal{D}) = \ell_k(\mathcal{D}+(k-1)S^\circ) - \ell_{2-k}(\mathcal{D})$. Now assume that $\mathcal{D}$ is integral. Then $$\begin{aligned} &\quad \ell_k(\mathcal{D}+(k-1)S^\circ) - \ell_{2-k}(\mathcal{D}) \\ &= \ell_0\left(\mathcal{D}+\left\lfloor \div \theta_k+(k-1)S^\circ\right\rfloor\right) - \ell_0(\mathcal{D} + \lfloor \div \theta_{2-k} \rfloor) \\
		&= \deg\left\lfloor \div \theta_k+(k-1)S^\circ\right\rfloor - \deg(\lfloor \div \theta_{2-k} \rfloor) \qquad \text{when }\deg \mathcal{D} \text{ is sufficiently large, by Riemann--Roch} \\
		&= \dim \frac{\M_k^S(\Gamma)}{D^{k-1}(\M_{2-k}^S(\Gamma))} - \delta_{k,2} \qquad \text{by Lemma }~\ref{lemma_k_2-k_dimension_S}.\end{aligned}$$
	Therefore $\dim_K \mathcal{H}_t(\mathcal{D}) = \dim \frac{\M_k^S(\Gamma)}{D^{k-1}(\M_{2-k}^S(\Gamma))} - \delta_{k,2}$ when $\mathcal{D}$ is integral, supported on $S$, with $\deg \mathcal{D}$ sufficiently large.
	
		\begin{proof}[Proof of Theorem \ref{creative_telescoping_S_intro}]
		Write $d:= \dim \frac{\M_k^S(\Gamma)}{D^{k-1}(\M_{2-k}^S(\Gamma))} - \delta_{k,2}$. Consider the $d+1$ forms $g,tg,\ldots,t^d g$. Choose an integral divisor $\mathcal{D}$ supported on $S$, with $\deg \mathcal{D}$ sufficiently large, such that they all lie in $\mathcal{L}_k(\mathcal{D}+(k-1)S^\circ)$. We may also assume that $\dim_K \mathcal{H}_t(\mathcal{D}) = d$. The $d+1$ elements above are therefore linearly dependent when viewed as elements of $\mathcal{H}_t(\mathcal{D})$. That is, there exist $r_0,\ldots,r_d \in K$ such that
		$$\sum_{j=0}^d r_j(n) t^j g = 0 \text{ in }\mathcal{H}_t(\mathcal{D})\iff \sum_{j=0}^d r_j(n) t^j g = \frac{D^{k-1}(t^n R_n)}{t^n}$$
		for some $R_n \in \mathcal{L}_{2-k}(\mathcal{D}) \otimes_\mathbb{C} K$. The proof is completed by setting $\boldL = \sum_{j=0}^d r_j(n) \textbf{S}^j$.
	\end{proof}
	
	\subsection{Computation of creative telescoping}\label{sec:creative-computation}
	Having established the existence of a creative telescoper $\boldL$ and a certificate $R_n \in \M_{2-k}^S(\Gamma) \otimes_\mathbb{C} \mathbb{C}(n)$, we now seek a practical algorithm to compute them. 
	
	Assume that $\infty$ is a regular cusp, and write $q_\infty=e^{2\pi i\tau/w}$, where $w$ is its width. Recall \cite[Proposition~2.16]{shimura1971introduction}, which says
	\begin{equation}\label{degXk}\deg(\div\theta_k) = \frac{k}{2}\left(2g-2+\varepsilon_\infty+\sum_{j=1}^r\left(1-\frac{1}{e_j}\right)\right).\end{equation}
	The following lemma derives an easy-to-compute valence bound for $\mathcal{L}_k(\mathcal{D})$. 
	\begin{lemma}\label{lemma_valence_bound}
	Assume that $\infty$ is a regular cusp. For a divisor $\mathcal{D}$, put
	$$B_{k,\mathcal{D}}:=\left\lfloor \deg(\div\theta_k) +\deg\mathcal{D}-\mathcal{D}(\infty)\right\rfloor.$$
	If $f=\sum_n a_nq_\infty^n\in\mathcal{L}_k(\mathcal{D})$ and $a_n=0$ for every $n\leq B_{k,\mathcal{D}}$, then $f=0$. 
	\end{lemma}
	\begin{proof}
	Suppose that $0\neq f\in\mathcal{L}_k(\mathcal{D})$. By the definition of $\mathcal{L}_k(\mathcal{D})$, we have $v_x(f)\geq-\mathcal{D}(x)$ for every $x\in X(\Gamma)$. Therefore
	\begin{align*}
	 \deg(\div\theta_k) =\deg(\div f)
	&=v_\infty(f)+\sum_{x\neq\infty}v_x(f)\\
	&\geq v_\infty(f)-\sum_{x\neq\infty}\mathcal{D}(x)=v_\infty(f)-\deg\mathcal{D}+\mathcal{D}(\infty).
	\end{align*}
	It follows that
	$$v_\infty(f)\leq \deg(\div\theta_k) +\deg\mathcal{D}-\mathcal{D}(\infty).$$
	Since $\infty$ is regular, $v_\infty(f)$ is an integer. Thus every nonzero $f\in\mathcal{L}_k(\mathcal{D})$ satisfies $v_\infty(f)\leq B_{k,\mathcal{D}}$. If all coefficients $a_n$ with $n\leq B_{k,\mathcal{D}}$ vanished, then $v_\infty(f)>B_{k,\mathcal{D}}$, a contradiction, forcing $f=0$.
	\end{proof}
	
	\begin{remark}\label{CT_classical}
		The modular version of creative telescoping has the feature, shared by its classical counterpart, that verifying it is much easier than finding it. In this case, verification consists simply of computing enough $q$-coefficients beyond a valence bound; see Example \ref{Ap\'ery_rec_verification} below. \par
		%A distinction and nicer feature for the modular version is that the upper-order bound given in the statement is often tight, unless there is extra symmetry available (e.g. Proposition \ref{prop:CT_refined_order} and Theorem \ref{creative_telescoping_char}). This is in contrast to classical version: a priori bound for the recurrence order is not easy to obtain, and is often large in comparison to the actual order.
	\end{remark}
	
	\begin{example}\label{Ap\'ery_rec_verification}
		Consider the modular forms $t,u,v$ on $\Gamma_1(6)$ defined in \eqref{level6_uv}, with $t= \frac{uv}{(u+8v)(u+9v)} \sim (0)+(\infty)-(\frac12)-(\frac13)$. Write $R= -\frac{uv}{2(u+8v)^2(u+9v)^2}\in \mathcal{L}_{-2}(\mathcal{D})$. Here $k=4$ and $\mathcal{D} = 2(\frac12)+2(\frac13)$. Then Proposition \ref{Dt^n_continuation} implies $$\frac{D^3(t^n R)}{t^n} \in \mathcal{L}_4(\mathcal{D}).$$
		
		A valence bound for $\mathcal{L}_4(\mathcal{D})$ is $8$. Therefore, to check the identity \eqref{aux_12}, in which both sides lie in $\mathcal{L}_4(\mathcal{D})$,
		$$(n+1)^3 g -(3 + 2 n) (39 + 51 n + 17 n^2) tg+(n+2)^3 t^{2}g =  \frac{D^3\left(t^n  R\right)}{t^n},$$
		it suffices to compare the $q$-expansions of both sides up to $q^8$. The $q$-expansion of the right-hand side is computed just as in the proof of Proposition \ref{Dt^n_continuation}, and that of the left-hand side is even easier. 
	\end{example}

	In the following algorithm, we assume we can compute the $q$-expansion of any given element of $\mathcal{L}_k(\mathcal{D})$ to arbitrary precision.
	
	\begin{tcolorbox}[breakable, title=Algorithm for linear relations in the $\frac{D^{k-1}(t^n \cdot)}{t^n}$ space]
	Given
	\begin{itemize}[leftmargin=*]
	\item divisors $\mathcal{D}_1, \mathcal{D}_2$ supported on $S$;
	\item $f_1,\ldots,f_l \in \mathcal{L}_k(\mathcal{D}_1)$.
	\end{itemize}
	The following algorithm determines whether there exist $r_1,\ldots,r_l\in K = \mathbb{C}(n)$ and $R_n \in \mathcal{L}_{2-k}(\mathcal{D}_2) \otimes K$ such that \begin{equation}\label{aux_13}r_1 f_1 + \cdots + r_l f_l = \frac{D^{k-1}(t^n R_n)}{t^n}.\end{equation}
	\tcblower
	\begin{enumerate}[leftmargin=*]
	\item Compute a valence bound for $\mathcal{L}_k(\max(\mathcal{D}_1,\mathcal{D}_2+(k-1)S^\circ))$ using Lemma~\ref{lemma_valence_bound}, and denote it by $B$.
	\item Compute the $q_\infty$-expansions of $f_1,\ldots,f_l$ up to $O(q_\infty^{B+1})$.
	\item Compute a basis of $\mathcal{L}_{2-k}(\mathcal{D}_2)$ through at least $O(q_\infty^{B+1})$, obtaining $v_1,\ldots,v_N \in \mathbb{C}((q_\infty))$, and put
	$$\mu:=\min_{1\leq i\leq N}v_\infty(v_i).$$
	\item Put $\nu:=v_\infty(t)$, compute $t$ to the absolute precision $O(q_\infty^{\nu+B-\mu+1})$. This determines
	$a:=\frac{Dt}{t}$
	up to $O(q_\infty^{B-\mu+1})$. Use the procedure in the proof of Proposition~\ref{Dt^n_continuation} to calculate
	$$\vartheta_i = \frac{D^{k-1}(t^n v_i)}{t^n}=(D+na)^{k-1}v_i \pmod{q_\infty^{B+1}}, \quad 1\leq i\leq N.$$
	%Indeed, $a\in\mathbb{C}\llbracket q_\infty\rrbracket$, and an $a$-coefficient of degree greater than $B-\mu$ cannot contribute to a coefficient of $\theta_i$ of degree at most $B$. 
	The coefficients of $\vartheta_i$ are in $K$.
	\item Perform Gaussian elimination over $K$ on the homogeneous system obtained by comparing all coefficients of $q_\infty^m$ with $m\leq B$ in
	$$\sum_{j=1}^l r_j f_j = \sum_{i=1}^N A_i \vartheta_i$$
	and determine whether it has a solution with $(r_1,\ldots,r_l)\neq 0$. If not, equation \eqref{aux_13} has no nontrivial solution involving the $f_j$; otherwise, the resulting $r_1,\ldots,r_l$, together with $R_n = \sum_{i=1}^N A_i v_i$, give a desired solution.
	\end{enumerate}
	\end{tcolorbox}
	
	We elaborate a bit on Step (3), which is less straightforward than the other steps. For the purpose of the algorithm, it is enough to compute a spanning set instead of a basis of $\mathcal{L}_{2-k}(\mathcal{D}_2)$. Usually one can identify $\Xi \in \M_l(\Gamma)$ such that $\Xi \mathcal{L}_{2-k}(\mathcal{D}_2) \subset \M_{2-k+l}(\Gamma)$; the latter space has an easily computable basis. If a spanning set is used, the kernel computed in Step (5) may also contain relations among the $v_i$ themselves; this is why Step (5) requires $(r_1,\ldots,r_l)\neq 0$.
	
	\begin{tcolorbox}[breakable, title=Algorithm for a creative telescoper and certificate in Theorem \ref{creative_telescoping_S_intro}]
	Given $t,g$ as in the statement of Theorem~\ref{creative_telescoping_S_intro}, the following algorithm computes a telescoper $\boldL$ and a certificate $R_n \in \M_{2-k}^S(\Gamma) \otimes K$ occurring in the theorem.
	\tcblower
	\begin{enumerate}[leftmargin=*]
	\item Put $\mathcal{D}_t = (\div t)_\infty$ and $\mathcal{D}_g = (\div g)_\infty$; both divisors are supported on $S$.
	\item Initialize $j=0$. For the current value of $j$, put $\mathcal{D}_2 =\mathcal{D}_g + j\mathcal{D}_t$ and $\mathcal{D}_1 =\mathcal{D}_2 + (k-1)S^\circ$.
	\item Execute the previous algorithm with
	$$(f_1,\ldots,f_{j+1}) = (g,tg,\ldots,t^j g)$$
	and with $\mathcal{D}_1, \mathcal{D}_2$ as above.
	\item If it finds no relation involving the $f_i$, increase $j$ by $1$ and return to step (2).
	\item If it returns a relation, relabel it as
	$$\sum_{r=0}^j r_r(n)t^r g=\frac{D^{k-1}(t^nR_n)}{t^n}.$$
	Output
	$$\boldL=\sum_{r=0}^j r_r(n)\textbf{S}_n^r$$
	and the certificate $R_n$.
	\end{enumerate}
	\end{tcolorbox}
	Note that the proof of Theorem~\ref{creative_telescoping_S_intro} guarantees that the algorithm terminates. 
	
	For later reference, we give a refinement of Theorem~\ref{creative_telescoping_S_intro}. The key point is that we obtain a saving in the order of $\boldL$ from suitable cusps depending on $t$ and $g$.
	\begin{proposition}\label{prop:CT_refined_order}
		Let $k\geq 2$ and let $0\neq g\in \M_k^S(\Gamma)$. Put
		$$\mu(t,g):=\left\{\alpha\in\Gamma\backslash\mathbb{P}^1(\mathbb{Q}) \mid \alpha \text{ regular},\ v_\alpha(t)=0,\quad v_\alpha(g)\in\mathbb{Z}_{>0}\right\}.$$
		Then, in the statement of Theorem~\ref{creative_telescoping_S_intro}, we can choose $\boldL$ to have order at most
		$$\dim\frac{\M_k^S(\Gamma)}{D^{k-1}(\M_{2-k}^S(\Gamma))}
		-\delta_{k,2}-\#\mu(t,g).$$
	\end{proposition}
	\begin{proof}
		Let
		$$\mathcal{D}^\sharp:=\sum_{\alpha\in\mu(t,g)}(\alpha).$$
		Choose an effective integral divisor $\mathcal{D}$ supported on the nonempty set $S\setminus\mu(t,g)$. At a cusp $\alpha\in\mu(t,g)$, an element $h\in\mathcal{L}_{2-k}(\mathcal{D})$ is holomorphic and $t$ is a unit. Thus $t^nh$ is holomorphic at $\alpha$, and $D^{k-1}(t^nh)$ has zero constant Fourier coefficient there. Division by the unit $t^n$ preserves this vanishing. Together with Proposition~\ref{Dt^n_continuation}, this gives
		$$\left\{\frac{D^{k-1}(t^nh)}{t^n}:h\in\mathcal{L}_{2-k}(\mathcal{D})\otimes_\mathbb{C}K\right\}
		\subset
		\mathcal{L}_k\bigl(\mathcal{D}+(k-1)S^\circ-\mathcal{D}^\sharp\bigr)\otimes_\mathbb{C}K.$$
		Define $\mathcal{H}_t^\sharp(\mathcal{D})$ to be the quotient of the space on the right by the space on the left. The map defining the space in the denominator is injective, as in the discussion following \eqref{quotient_D_space_CT}. Hence, when $\deg\mathcal{D}$ is sufficiently large, Riemann--Roch gives
		\begin{align*}
		\dim_K\mathcal{H}_t^\sharp(\mathcal{D})
			&=\ell_k\bigl(\mathcal{D}+(k-1)S^\circ-\mathcal{D}^\sharp\bigr)-\ell_{2-k}(\mathcal{D})\\
			&=\dim\frac{\M_k^S(\Gamma)}{D^{k-1}(\M_{2-k}^S(\Gamma))}
			-\delta_{k,2}-\#\mu(t,g).
		\end{align*}
		Put the last quantity equal to $d$. By increasing $\mathcal{D}$ if necessary, the forms $g,tg,\ldots,t^dg$ all belong to the numerator defining $\mathcal{H}_t^\sharp(\mathcal{D})$. Their $d+1$ classes are therefore linearly dependent over $K$. The resulting relation gives a creative telescoper of order at most $d$, exactly as in the proof of Theorem~\ref{creative_telescoping_S_intro}.
	\end{proof}
	
	\section{Two examples of meromorphic modular forms}\label{sec:meromorphic}
In this section, we look closely at two examples of modulonomic sequences on meromorphic modular forms, both on $\Gamma =\SL_2(\mathbb{Z})$. We also experimentally observe that their arithmetic properties differ somewhat from those of modulonomic sequences coming from weakly holomorphic forms.
\subsection{An example at an elliptic point}
	\label{level_one_creative_telescoping}
	Let $\Gamma=\mathrm{SL}_2(\mathbb{Z})$, $\rho=e^{2\pi i/3}$, and $S=\{\rho,\infty\}$. Recall that
	$$E_4 \sim \frac{(\rho)}{3},\qquad E_6 \sim \frac{(\sqrt{-1})}{2},\qquad J := \frac{j}{1728} \sim (\rho) - (\infty).$$ 
	Theorem~\ref{dim_theorem} gives
	$$\dim\frac{\M_4^S(\Gamma)}{D^3\bigl(\M_{-2}^S(\Gamma)\bigr)}=2,$$
	and the classes of $E_4$ and $E_4/J$ form a basis of this quotient. Hence, for every $n\geq1$, there are unique $A_n\in\mathbb{C}$ and $h_n\in\M_{-2}^S(\Gamma)$ such that
	\begin{equation}\label{aux_17}\frac{E_4}{J^n}=A_n\frac{E_4}{J}+D^3(h_n).\end{equation}
	The first values are
	\begin{center}
	\renewcommand{\arraystretch}{1.5}
	\begin{tabular}{c|c|c}
	$n$ & $A_n$ & $h_n$ \\ \hline
	$1$ & $1$ & $0$ \\
	$2$ & $\dfrac{41}{64}$ & $\dfrac{9E_6}{8E_4^2}$ \\
	$3$ & $\dfrac{18147}{35840}$ & $\dfrac{E_6}{E_4^2}\left(\dfrac{999}{896}+\dfrac{9}{70J}\right)$
	\end{tabular}
	\end{center}
	The creative telescoping algorithm above (with $t=J^{-1}$) produces $\boldL(t^n E_4) = D^3(t^n R_n)$
	with $$\boldL = 72n^3 -(2n+1)(72n^2+72n+41)\textbf{S}_n+8(n+1)(3n+2)(3n+4)\textbf{S}_n^2,\qquad R = R_n = 72\frac{E_6}{E_4^2}.$$
	As in Example~\ref{Ap\'ery_rec_verification}, once this identity has been written down, it can be verified by comparing finitely many Fourier coefficients using the valence bound of Lemma~\ref{lemma_valence_bound}. It follows that $\boldL(A_n)=0$. Let $B_n:=a_0(h_n)$. Since $t$ vanishes at $\infty$, we also have $\boldL(B_n) = 0$. Experimentally, the two sequences have the following denominator types, which have been verified up to $n=10000$.
	\begin{conjecture}\label{LCM_bound_Ex1}
	The denominator types of $A_n$ and $B_n$ are $$A_n:\quad 64^n\lcmD{3n-2},\qquad
	B_n:\quad 64^n\lcmD{3n-2}  \lcmD{n-1} \lcmD{\frac{3n-2}{4}}.$$
	\end{conjecture}
	
	On the other hand, it is not difficult\footnote{For example, one may use standard creative telescoping to verify directly that the right-hand side satisfies the same recurrence $\boldL$.} to show that
	\begin{equation}\label{AB_explicit_expr}\begin{aligned}A_{m+1}&=\frac{m!^3}{64^m(3m+1)!}\sum_{j=0}^{m}432^{m-j}\binom{2m-2j}{m-j}\frac{(6j)!}{(3j)!(j!)^3}, \\
	B_{m+1} &= \frac{1728 \: m!^3}{64^m(3m+1)!}\sum_{j=0}^{m}432^{m-j}\binom{2m-2j}{m-j}\frac{(6j)!}{(3j)!(j!)^3} \left(\sum_{\ell=1}^j 432^{\ell-1}
	\frac{(2\ell-2)!(3\ell)!(\ell-1)!}{(6\ell)!} \right).
	\end{aligned}\end{equation}
	Note that Conjecture \ref{LCM_bound_Ex1} gives $$\limsup_{n\to\infty} \frac{\log(\den(64^n A_n))}{n} \leq 3,\qquad \limsup_{n\to\infty} \frac{\log(\den(64^n B_n))}{n} \leq 3+1+\frac{3}{4} = \frac{19}{4}.$$
	The fact that $A_n$ and $B_n$ have the above explicit expressions allows us to improve the denominator types. 
	\begin{proposition}\label{beta_exp_mero}
	Assuming Conjecture \ref{LCM_bound_Ex1} holds, we have $$ \limsup_{n\to\infty} \frac{\log(\den(64^n A_n))}{n} \leq \beta, \qquad  \limsup_{n\to\infty}  \frac{\log(\den(64^n B_n))}{n} \leq 2+\beta,$$
	where $\beta=\frac{\pi}{2\sqrt3}+\frac32\log3 \approx 2.55482$.
	\end{proposition}
	\begin{proof}[Proof sketch]
	Let $X_n = \log(\den(64^n A_n))$. Since $\frac{(6j)!}{(3j)!(j!)^3} \in \mathbb{Z}$, we have $$X_n \leq \log \left( \gcd\left(\lcmD{3n-2},  \frac{(3n-2)!}{(n-1)!^3}\right) \right) + O(1).$$ We consider which primes $p$ occur in $\lcmD{3n-2}$. We can safely ignore the primes $p\leq \sqrt{3n-2}$, since they contribute at most $O(\sqrt{n})$ to $X_n$. The remaining primes $p$ occur in $\lcmD{3n-2}$ with multiplicity one. Now
	$$v_p\left( \frac{(3n-2)!}{(n-1)!^3}\right)=\sum_{a\geq 1}
	\left(\left\lfloor\frac{3n-2}{p^a}\right\rfloor-3\left\lfloor\frac{n-1}{p^a}\right\rfloor\right).$$
	We only need to focus on $a=1$. When $\lfloor\frac{3n-2}{p}\rfloor-3\left\lfloor\frac{n-1}{p}\right\rfloor = 0$, $p$ does not contribute to $X_n$. Since $\lfloor 3x \rfloor - 3\lfloor x\rfloor = 0 \iff j\leq x < \frac{1}{3}+j$ for some $j\in \mathbb{Z}$, we see that such primes $p$ lie roughly in the following intervals:
	$$(\frac{3n}{4},n),(\frac{3n}{7},\frac{n}{2}),\cdots (\frac{n}{1/3+j},\frac{n}{j}),\cdots.$$
	For fixed $J$, the prime number theorem implies
	$$ X_n \leq 3n- n\sum_{j=1}^J \left(\frac{1}{j} - \frac{1}{1/3+j}\right) + o(n).$$
	Letting $J\to \infty$ gives the conclusion. The argument for $B_n$ is similar.
	\end{proof}
	Numerical experimentation indicates the above two bounds for $\limsup$ are tight. 
	\begin{remark}\label{E4/J_eval}
	(a) As in previous sections, we can integrate \eqref{aux_17} from $0$ to $i\infty$. Lemma \ref{IBP_lemma} then gives
	$$(2\pi i)^3 \int_0^{i\infty} \frac{E_4}{J^n} d\tau = A_n \left((2\pi i)^3 \int_0^{i\infty} \frac{E_4}{J} d\tau \right)-2B_n.$$
	It is a rather non-trivial fact that \begin{equation}\label{aux_18}(2\pi i)^3 \int_0^{i\infty} \frac{E_4}{J} d\tau = 270 L_{-3}(2).\end{equation}
	The proof consists of recognizing $\frac{E_4}{J}$ as a simple multiple of 
	$$\sum_{\substack{(a,b,c)\in \mathbb{Z}^3 \\ b^2-4ac=-3}} (a\tau^2+b\tau+c)^{-2}.$$
	Such meromorphic modular forms were first considered by Zagier \cite{zagier1975modular} and subsequently studied by several authors; see, for example, \cite{lobrich2022locally, bringmann2019regularized, bengoechea2015meromorphic}. The evaluation \eqref{aux_18} is a special case of \cite[Theorem~1.5]{lobrich2022locally}. \\
	(b) We note that the sequence $A_n$ defined above is (up to a constant factor) the $(\tau-\rho)^{-2}$-coefficient of $E_4/J^n$. This observation does not seem to directly give the desired denominator type in Conjecture \ref{LCM_bound_Ex1}, because such expansions at interior points are much less understood; see also \cite{bogo2024laurent}.
	\end{remark}
	
	\begin{comment}
	\begin{remark}
	Consider the integral
	$$I(s) := (2\pi i)^3 \int_0^{i\infty} \frac{E_4(\tau)}{j(\tau)^s} d\tau,\qquad \Re(s)>0.$$
	Apart of its special value at $s\in \mathbb{N}$ we have just seen, one also has
	$$\begin{aligned} I(\frac13) = \frac{\sqrt{3} \Gamma \left(\frac{1}{3}\right)^9}{8 \pi ^3},\qquad I(\frac{1}{2}) =8 L(f_{24.a.1.1},3).\end{aligned}$$
	Here the first evaluation is rather trivial by noting that $E_4/j^{1/3} = \eta^8$, a CM cusp form. The second, due to Bonish \cite[p.~84]{bonisch2023modularity}, is more involved, note here $\sqrt{j}$ is not even a well-defined function on $\mathbb{H}$. After some effort, one can also prove the following expression:
	$$I(s) =
		\frac{8\pi^{5/2}\,\Gamma(s)}{\Gamma(s+\frac12)} \: \pFq{4}{3}{s,\frac16,\frac12,\frac56}{s+\frac12,1,1}{1}+
		\frac{4\pi^{3/2} \Gamma(s)^3}
		{\Gamma(s+\frac16)\Gamma(s+\frac12)\Gamma(s+\frac56)} \: \pFq{4}{3}{s,s,s,\frac12}{s+\frac16,s+\frac12,s+\frac56}{1}.$$
	\end{remark}
	\end{comment}

	\subsection{An example at a non-elliptic CM point}
	\begin{example}
	Consider again $\Gamma = \SL_2(\mathbb{Z})$ and $k=4$. Pick a non-elliptic point $\tau_0 \in \mathbb{H}$ and let $j_0 := j(\tau_0)$, $t = (j-j_0)^{-1}$, and $S=\{\infty,\tau_0\}$. Theorem \ref{dim_theorem} says that $\frac{\M_4^S}{D^3(\M_{-2}^S)}$ has dimension $4$. A basis is given by
	$$\left\{ E_4, \frac{E_4}{j-j_0},\frac{E_4}{(j-j_0)^2},\frac{E_4}{(j-j_0)^3}\right\}.$$
	Theorem \ref{creative_telescoping_S_intro} predicts a telescoper of order $\leq 4$ and a certificate $R_n$ such that
	$$\boldL(t^n E_4) = D^3(t^n R_n).$$
	Explicitly, we have
	$$\begin{aligned}
	\boldL={}&n^3
	+(2n+1)\bigl(2(j_0-864)n(n+1)+j_0-984\bigr)\textbf{S}_n\\
	&+(n+1)\bigl((6n^2+12n+7)j_0(j_0-1728)-240j_0
	+331776(3n+2)(3n+4)\bigr)\textbf{S}_n^2\\
	&+2j_0(j_0-1728)(j_0-864)(n+1)(n+2)(2n+3)\textbf{S}_n^3\\
	&+j_0^2(j_0-1728)^2(n+1)(n+2)(n+3)\textbf{S}_n^4,\\
	R_n={}&R:=\frac{E_4E_6}{E_4^3-j_0\Delta}=\frac{Dt}{tE_4}.
	\end{aligned}$$
	%After division by $t^n$, both sides of the telescoping identity have weight $4$ and a pole of order at most $4$ at $\tau_0$, with no other poles. Thus comparing the coefficients of $q^0,\ldots,q^4$ proves the identity by Lemma~\ref{lemma_valence_bound}.
	Actually, for this special case, it is not difficult to prove conceptually that the minimal order of $\boldL$ is $4$ and that an order-$4$ certificate $R_n$ can be normalized to be independent of $n$. Suppose that $\boldL=\sum_{j=0}^l r_j(n)\textbf{S}_n^j$. Then $$\sum_{j=0}^l r_j(n) E_4 t^j = \frac{D^3(t^n R_n)}{t^n}.$$ Since $t$ and $E_4$ are holomorphic at $\infty$, the right-hand side is analytic at $\infty$, so $v_\infty(R_n)\geq 0$. Since $t$ has a simple pole at $\tau_0$, the right-hand side has a pole of order $l$ at $\tau_0$, and hence $v_{\tau_0}(R_n)\geq 3-l$. Since $R_n$ has weight $-2$ and is holomorphic away from $S$, Lemma~\ref{lemma_elliptic_order_congruence} gives
	$$v_\rho(R_n)\geq\frac13,\qquad v_i(R_n)\geq\frac12,
	\qquad \rho=e^{2\pi i/3},\quad i=\sqrt{-1}.$$
	The valence formula \eqref{degXk} therefore implies
	$$-\frac16=\deg\div(R_n)\geq(3-l)+\frac13+\frac12,
	\qquad\text{hence }l\geq4.$$
	When $l=4$, equality is forced throughout, so
	$$\div(R_n)=\frac{(\rho)}3+\frac{(i)}2-(\tau_0)=\div(R).$$
	Consequently, $R_n/R\in\mathbb{C}(n)$, and rescaling the telescoping identity makes its certificate equal to $R$, independent of $n$. We will see a similar argument again in Lemma \ref{four_cusp_cert} below.
	\end{example}
	
	\begin{example}
	In the previous example, let $\tau_0 = \frac{1+\sqrt{-7}}{2}$, $j_0 = -3375$, and $t = \dfrac{1}{j+3375}$. Instead of using the naive basis $\{tE_4,t^2E_4,t^3E_4\}$, we use the basis recorded in \cite[Example~4.2]{zhang2025elliptic}:
	\begin{equation}\label{aux_19}\begin{aligned}
	G^{(1)}&:=tE_4,\\
	G^{(2)}&:=\frac{19}{27}tE_4-3375t^2E_4,\\
	G^{(3)}&:=\frac{6995}{11907}tE_4-\frac{55875}{7}t^2E_4+22781250t^3E_4.
	\end{aligned}\end{equation}
	We will soon see that they produce the natural splitting for the denominator types of the corresponding modulonomic sequences. When $n\geq 1$, there exist rational numbers $A_{(1),n},A_{(2),n},A_{(3),n}$ such that
	$$\frac{E_4}{(j+3375)^n} = \sum_{j=1}^3 A_{(j),n} G^{(j)} + D^3(h_n),\qquad h_n \in \M_{-2}^S(\Gamma).$$
	Let $B_n := a_0(h_n)$. Then these four sequences are annihilated by $\boldL$. 
	Using the recurrence and the following initial values, it is possible to calculate $A_{(j),n}$ and $B_n$ efficiently for large $n$. 
	\begin{center}
	\renewcommand{\arraystretch}{1.7}
	\begin{tabular}{c|c|c|c|c}
	$n$ & $A_{(1),n}$ & $A_{(2),n}$ & $A_{(3),n}$ & $h_n$ \\ \hline
	$1$ & $1$ & $0$ & $0$ & $0$ \\
	$2$ & $\dfrac{19}{3^6 5^3}$ & $-\dfrac{1}{3^3 5^3}$ & $0$ & $0$ \\
	$3$ & $\dfrac{2137}{3^{10}5^6 7^2}$ & $-\dfrac{149}{2\cdot3^8 5^6 7}$ & $\dfrac{1}{2\cdot3^6 5^6}$ & $0$ \\
	$4$ & $\dfrac{2903701}{3^{18}5^9 7^3}$ & $-\dfrac{27883}{2\cdot3^{14}5^9 7^2}$ & $\dfrac{157}{3^{12}5^9 7}$ & $\dfrac{R}{2\cdot3^{19}5^6 7^2}$
	\end{tabular}
	\end{center}
	Here $R=E_4E_6/(E_4^3+3375\Delta)$ is the preceding certificate specialized at $j_0=-3375$. The first three rows follow by changing basis. For the fourth, write $\boldL=\sum_{j=0}^4r_j(n)\textbf{S}_n^j$; setting $n=0$ in the telescoping identity gives $h_4=R/r_4(0)$. In particular, $a_0(R)=1$ gives
	$$(B_1,B_2,B_3,B_4)=\left(0,0,0,\frac{1}{2\cdot3^{19}5^6 7^2}\right).$$
	We summarize our observations for these sequences below. They have been verified up to $n\leq 10000$.
	\begin{conjecture}\label{conj_mero_-7}Let $X_{i,n}=(3^6\,5^3\,7)^nA_{(i),n}$ and $Y_n = (3^6\,5^3\,7)^nB_n$. Then \\
	(a) The denominator types of these four sequences are \begin{equation}\label{conj_mero_denom}\begin{aligned}&X_{2,n}:\quad \lcmD{n-1}\\ &X_{1,n},X_{3,n}:\qquad \lcmD{n-1}\lcmD{\frac{n-1}{2}}\\
		&Y_n:\qquad \lcmD{n-1}\lcmD{\frac{n-1}{2}} \lcmD{\frac{n-1}{3}} .\end{aligned}\end{equation}
	(b) The denominators of $X_{1,n}$ and $X_{3,n}$ have essentially complementary prime support. More precisely, if $p\mid \mathrm{den}(X_{1,n})$, then $p$ is inert in $\mathbb{Q}(\sqrt{-7})$; if $p\mid \mathrm{den}(X_{3,n})$, then either $p$ splits in $\mathbb{Q}(\sqrt{-7})$ or $p^2<n$. \\
	(c) $X_{1,n}$ is always a $2$-adic unit, i.e., $v_2(X_{1,n})= 0$. \\
	(d) For $n\geq 3$, $v_7(A_{(3),n})=v_7(A_{(1),n})+2$.
	\end{conjecture}
	Assuming the conjecture, we see that there exist integer sequences $d_{i,n}$ satisfying
	$$d_{2,n} = (3^6\,5^3\,7)^n e^{n+o(n)},\qquad d_{1,n}, d_{3,n} = (3^6\,5^3\,7)^n e^{3n/4+o(n)}$$
	such that $d_{i,n} A_{(i),n} \in \mathbb{Z}$. See Section \ref{CM_level7_sect} for an example of such an estimate.
	\end{example}

%$$(2\pi i)^3 \int_0^{i\infty} G^{(2)} d\tau = (2\pi i)^3 \int_0^{i\infty} \frac{E_4(19j-27000)}{27(j+3375)^2} d\tau = \frac{10}{729} L_{-7}(2).$$
	\begin{comment}
	Actually, part (a), (b) of the above conjecture seems to generalize to any negative discriminant $D$, regardless of class number. First define, $j_D$ and $\psi_D \in \mathbb{Q}(j_D)$ to be the the value of $j$ and $\frac{E_2^\ast E_4}{E_6}$ evaluated at any CM-point\footnote{picking another such CM-point changes $j_D,\psi_D$ to a Galois conjugate, this does not affect the statement of the conjecture we will give below.} of discriminant $D$. Then let (see \cite[Example~4.1]{zhang2025elliptic}):
	$$G^{(1)} := XXX$$
	When $D=-7$, we recover exactly \eqref{aux_19}. Consider the following expansion $$\frac{E_4}{(j-j_D)^n} = \sum_{j=1}^3 A_{(j),n} G^{(j)} + D^3(h_n),\qquad h_n \in \M_{-2}^S(\Gamma).$$ Then we conjecture the following denominator types of the modulonomic sequences $A_{(i),n} \in \mathbb{Q}(j_D)$ and $B_n = a_0(h_n)$. 
		\begin{conjecture}\label{conj_mero_D}Let $D$ be any negative discriminant. Write $s_D = j_D(j_D-1728)$. Let $X_{i,n}=s_D^n A_{(i),n}$ and $Y_n = s_D^n  B_n$. Then \\
		(a) The denominator type of these four sequences are the same as in \eqref{conj_mero_denom}.\\
		(b) Let $p$ co-prime to $s_D$. If $p\mid \mathrm{den}(X_{1,n})$, then $p$ is inert in $\mathbb{Q}(\sqrt{D})$; if $p\mid \mathrm{den}(X_{3,n})$, then either $p$ splits in $\mathbb{Q}(\sqrt{D})$ or $p^2<n$. 
	\end{conjecture}
	\end{comment}
	Conjecture \ref{conj_mero_D_intro} mentioned in the introduction is a generalization to other discriminants. Note that when $D=-7$, $s_{-7} = 3^9 5^3 7$, while Conjecture \ref{conj_mero_-7} gives the stronger $3^6 5^3 7$. It is not immediately obvious what the optimal scaling factor for general $D$ is. The conjecture predicts that $j_D (j_D-1728)$ is always enough.

	\section{Zagier's weight $3$ modulonomic sequences}\label{sec:zagier}
	
	Suppose that $t$ is a weakly holomorphic modular function of degree $1$ that vanishes at the cusp $\infty$ and has a pole at the cusp $\alpha$. The following lemma, whose core ideas already appeared in Proposition \ref{prop:level_six_plus_constant_terms}, allows one to construct explicit solutions to the recurrence $\boldL$ arising from creative telescoping:
	$$\boldL(t^{-n}g) = D^{k-1}(t^{-n}R_n).$$
	\begin{lemma}\label{lemma:eichler_solution}
		(a) Suppose that $X\in \M_k^!(\Gamma)$ vanishes at $\infty$. Then the sequence
		$$u_X(n):=a_0\bigl(t^{-n}g \E_{k-1}(X)\bigr)$$
		satisfies
		$$\boldL u_X(n)\in \operatorname{span}_{\mathbb{C}(n)}
		\bigl\{t(\beta)^{-n}:\beta\notin\{\infty,\alpha\}\bigr\}$$
		for $n\gg0$, where $\beta$ runs over the cusps of $\Gamma$ and $t(\beta)$ denotes the finite nonzero value of $t$ at $\beta$. \\
		(b) Moreover, if $g$ is holomorphic at all cusps, then $R_n$ can be chosen to be holomorphic at the cusps other than $\infty$ and $\alpha$. If $X$ vanishes at every cusp other than $\alpha$, then $\boldL u_X(n) = 0$ for $n\gg 0$.
	\end{lemma}
	\begin{proof} (a) Using Lemma \ref{lemma:constant_term_integration_by_parts}, we have
		$$\boldL u_X(n) = a_0( D^{k-1}(t^{-n}R_n)\E_{k-1}(X)) = \pm a_0( t^{-n}R_n X).$$
		The product $R_nX$ has weight $2$, so Lemma \ref{lemma:weight_two_residue_relation} gives
		$$a_0( t^{-n}R_n X) = - \sum_{\beta\neq \infty} a_0((t^{-n} R_n X)|\omega_\beta).$$
		When $\beta = \alpha$, the summand vanishes because $t^{-1}$ has a zero there. At a cusp $\beta \notin \{\alpha,\infty\}$, the function $t$ is a unit, and hence
		$$a_0((t^{-n} R_n X)|\omega_\beta) \in \mathbb{C}(n)t(\beta)^{-n}.$$
		(b) The assertion that $R_n$ can be chosen to be holomorphic at the relevant cusps follows from the proof of Proposition \ref{prop:CT_refined_order}. In this case, $a_0((t^{-n} R_n X)|\omega_\beta) = 0$ for every $\beta\notin \{\alpha,\infty\}$.
	\end{proof}
	
	\subsection{Zagier's seven sequences}
	For the rest of this section, we focus on a group $\Gamma$ satisfying the following conditions:
	\begin{itemize}[leftmargin=*]
	\item $-I \notin \Gamma$;
	\item $X(\Gamma)$ has genus $0$, four cusps, and no elliptic points;
	\item all cusps are regular.
	\end{itemize}
	It is known that, up to conjugation in $\SL_2(\mathbb{R})$, there are exactly four such groups $\Gamma$ (see \cite[Section~9]{sebbar2001torsion} or \cite{beauville1982familles}):
	\begin{equation}\label{possible_gamma}\Gamma_1(5),\quad \Gamma_1(6),\quad \Gamma_0(8)\cap\Gamma_1(4),\quad \Gamma_0(9)\cap\Gamma_1(3).\end{equation}
	
	From equation \eqref{degXk}, we see that every weight-$k$ form in $\mathcal{A}_k(\Gamma)$ has degree $k$. As $X(\Gamma)$ has genus zero, the converse is also true: every degree-$k$ divisor is the divisor of an element of $\mathcal{A}_k(\Gamma)$. The dimension formula also gives $\dim \E_3(\Gamma) = 4$ and $\dim \S_3(\Gamma) = 0$. Let $\alpha_1,\alpha_2,\alpha_3,\infty$ be the four cusps of $\Gamma$. \par
	
	Fix $t \in \M_0^!(\Gamma)$ and $g,X \in \M_3(\Gamma)$ such that 
	\begin{equation}\label{tgX_div}\div t = (\infty) - (\alpha_1),\quad \div g = (\alpha_1) + (\alpha_2) + (\alpha_3),\quad \div X = (\infty)+(\alpha_2)+(\alpha_3).\end{equation}
	Then Proposition \ref{prop:CT_refined_order} predicts the existence of a recurrence $\boldL$ of order at most $2$ such that
	$$\boldL(t^{-n}g) = D^2(t^{-n}R_n).$$
	The special properties of $\Gamma$ allow us to determine the exact order of $\boldL$ and an explicit expression for the certificate. 
	\begin{lemma}\label{four_cusp_cert}
	In the above setting and notation, the order of $\boldL$ is exactly $2$, and we can take $$R_n = \frac{Dt}{t^3 g} \sim -2(\infty)+(\alpha_1).$$
	Moreover, the coefficients of $\boldL$ are polynomials in $n$ of degree at most $2$.
	\end{lemma}
	\begin{proof}
	Let $\boldL = \sum_{j=0}^l a_j(n) \textbf{S}_n^j$. Then
	$$\sum_{j=0}^l a_j(n) t^{-n-j} g = D^2(t^{-n} R_n).$$
	First, the refined construction in the proof of Proposition \ref{prop:CT_refined_order} allows us to take $R_n$ holomorphic at $\alpha_2$ and $\alpha_3$. Let $\mathcal{D}$ denote the divisor of the left-hand side. When $n$ is sufficiently positive, we have $\mathcal{D}(\infty) = -n-l$, and hence $(\div R_n)(\infty) = -l$. When $n$ is sufficiently negative, we have $\mathcal{D}(\alpha_1) = -n+1$, and hence $(\div R_n)(\alpha_1) = 1$. Thus
	$$\div R_n\geq -l(\infty)+(\alpha_1).$$
	Since a weight-$(-1)$ form has a divisor of degree $-1$, this implies $l\geq2$. The upper bound $l\leq2$ from Proposition \ref{prop:CT_refined_order} therefore gives $l=2$ and
	$$\div R_n=-2(\infty)+(\alpha_1).$$
	It follows that $R_n$ is a $\mathbb{C}(n)$-multiple of $\frac{Dt}{t^3g}$. After rescaling, we may therefore assume that $R_n = \frac{Dt}{t^3 g}$ is independent of $n$. Finally, the proof of Proposition \ref{Dt^n_continuation} shows that, for any $t,h\in \mathbb{C}((q))$,
	$$\frac{D^{k-1}(t^n h)}{t^n} \in \mathbb{C}[n]((q)),$$
	with every coefficient a polynomial in $n$ of degree at most $k-1$. Consequently, the coefficients $a_i(n)$ are polynomials of degree at most $2$.
	\end{proof}
	
	The two solutions of $\boldL(\cdots)=0$ for $n\gg0$ are
	$$a_0(t^{-n}g)\qquad\text{and}\qquad a_0(t^{-n}g\E_2(X)),$$
	the second being supplied by Lemma \ref{lemma:eichler_solution}.
	
	By taking all possible choices of $\Gamma$ in \eqref{possible_gamma} and then all possible choices of $\alpha_1,\alpha_2,\alpha_3$, we obtain a total of seven rational telescopers $\boldL$, i.e., telescopers with coefficients in $\mathbb{Q}[n]$. These are exactly the seven cases\footnote{There is an additional case arising from $\Gamma = \Gamma(5)$, but the corresponding $\boldL$ is defined over $\mathbb{Q}(\sqrt{5})$ rather than $\mathbb{Q}$.} written down by Zagier in \cite{zagier2009integral}; see Table \ref{tab:examples_groups} below. Two linearly independent solutions are $a_0(t^{-n}g)$ and $a_0(t^{-n}g \E_2(X))$. Their denominator types are trivial and $\lcmD{n}^2$, respectively. \par
	
	As in the proofs of Corollaries \ref{cor:Ap\'ery_mod_p} and \ref{congruence_level4_cor_2}, one could investigate the properties of these integer solutions modulo $p$. We do not pursue this point here.

	\begin{table}[h]
		\centering
		\begin{tabular}{@{}c c c c c@{}}
			\toprule
			& $\Gamma$ & $(\alpha_1,\alpha_2,\alpha_3)$ & $t$ & $g$ \\
			\midrule
			$\textbf{A}$ & $\Gamma_1(6)$ & $(\tfrac13,0,\tfrac12)$ & $q-3q^2+3q^3$ & $1-q-5q^2-q^3$ \\
			$\textbf B$ & $\Gamma_0(9)\cap\Gamma_1(3)$ & $(0,\tfrac13,-\tfrac13)$ & $-q-3q^2-9q^3$ & $1-9q^3$ \\
			$\textbf C$ & $\Gamma_1(6)$ & $(\tfrac12,0,\tfrac13)$ & $q-4q^2+10q^3$ & $1-q-5q^2-q^3$ \\
			$\textbf D$ & $\Gamma_1(5)$ & $(\tfrac25,0,\tfrac12)$ & $q-5q^2+15q^3$ & $1-2q-6q^2+7q^3$ \\
			$\textbf E$ & $\Gamma_0(8)\cap\Gamma_1(4)$ & $(\tfrac12,0,\tfrac14)$ & $q-4q^2+12q^3$ & $1-4q^2$ \\
			$\textbf F$ & $\Gamma_1(6)$ & $(0,\tfrac12,\tfrac13)$ & $-q-5q^2-19q^3$ & $1-q-5q^2-q^3$ \\
			$\textbf G$ & $\Gamma_0(8)\cap\Gamma_1(4)$ & $(\tfrac14,0,\tfrac12)$ & $q-4q^3$ & $1-4q^2$ \\
			\bottomrule
		\end{tabular}
			\caption{Groups, cusps, and relevant data that reproduce Zagier's sequences $\textbf A$--$\textbf G$. Every displayed expansion is modulo $O(q^4)$.}
			\label{tab:examples_groups}
	\end{table}

	From our description of these sequences, it is easy to compute their generating functions and monodromy. 
	\begin{lemma}\label{lemma:constant_term_generating_function}
		Let $t,g\in\mathbb{C}((q))$ be Laurent series convergent in a neighbourhood of $q=0$, and suppose that $v_\infty(t)=1$ and $g$ is holomorphic at $q=0$. Then, for $\tau$ sufficiently close to $i\infty$,
		\begin{equation}\label{eq:constant_term_generating_function}
			\sum_{n\geq0}a_0(t^{-n}g)t(\tau)^n=\frac{t(\tau)g(\tau)}{Dt(\tau)}.
		\end{equation}
	\end{lemma}
	\begin{proof}
		Since $t$ has a simple zero at $q=0$, it is a local coordinate there. Let $q=q(t)$ be its local inverse and put
		$$F(t):=\frac{t\,g(q(t))}{Dt(q(t))}.$$
		Changing variables from $q$ to $t$ in the residue gives
		\begin{align*}
			a_0(t^{-n}g)
			&=\underset{q=0}{\operatorname{Res}}\,t(q)^{-n}g(q)\frac{dq}{q}\\
			&=\underset{t=0}{\operatorname{Res}}\,t^{-n}\frac{g(q(t))}{Dt(q(t))}\,dt
			=[t^n]F(t).
		\end{align*}
		Summing over $n\geq0$ proves \eqref{eq:constant_term_generating_function}. 
	\end{proof}
	
	\begin{example}
	Here we briefly recast, in our framework, the construction of the initial approximation used in \cite[Section~11]{CDT}. In the above setting, let $\Gamma = \Gamma_1(6)$ and $(\alpha_1,\alpha_2,\alpha_3) = (\frac12,0,\frac13)$, corresponding to case $\textbf{C}$ in Table \ref{tab:examples_groups}. Choose any $t,g,X$ satisfying the divisor condition \eqref{tgX_div}. The two homogeneous solutions of the telescoper $\boldL$ are $a_n := a_0(t^{-n}g)$ and $b_n := a_0(t^{-n} g\E_2(X))$. Choose any $Y$ with $\div Y = (\infty)+(\alpha_1)+(\alpha_2)$. The proof of Lemma \ref{lemma:eichler_solution} then implies that, for $c_n := a_0(t^{-n} g \E_2(Y))$, the quantity $\boldL c_n$ is a nonzero constant. \par
	Write $$\mathscr{A}(\tau) := \sum_{n\geq 0} a_n t(\tau)^n,\qquad \mathscr{B}(\tau) := \sum_{n\geq 0} b_n t(\tau)^n,\qquad \mathscr{C}(\tau) := \sum_{n\geq 0} c_n t(\tau)^n.$$
	Then Lemma \ref{lemma:constant_term_generating_function} says
	$$\mathscr{A}(\tau) = \frac{t(\tau)g(\tau)}{Dt(\tau)} ,\qquad \mathscr{B}(\tau)= \frac{t(\tau)g(\tau)}{Dt(\tau)} \E_2(X)(\tau),\qquad \mathscr{C}(\tau) = \frac{t(\tau)g(\tau)}{Dt(\tau)} \E_2(Y)(\tau).$$
	These expressions immediately show that the generating functions extend to the universal cover of $\mathbb{P}^1(\mathbb{C})\setminus \{0,\infty,t(\alpha_2),t(\alpha_3)\}$. Their monodromy is governed by their transformations under $\Gamma_1(6)$. In particular, $\mathscr{B}(\tau)$ and $\mathscr{C}(\tau)$ encode $\zeta(2)$ and $L(2,\chi_{-3})$, respectively, through the monodromy of the Eichler integrals:
	\begin{equation}\label{eichler_trans}\E_k(f)|_{2-k} \gamma - \E_k(f) = \frac{(2\pi i)^{k-1}}{(k-1)!} \int_{\gamma^{-1}\infty}^\infty (\tau-z)^{k-2} f(z) dz,\qquad \gamma \in \Gamma.\end{equation}
	\end{example}
	
	\begin{comment}
	\begin{corollary}
	Let $k\geq 2$ and let $0\neq g\in \M_k^S(\Gamma)$. Let $t$ be a degree $1$ modular function such that $\div t=(\infty)-(\alpha)$, where $\alpha,\infty\in S$. Assume that $g$ vanishes at every cusp other than $\alpha$ and $\infty$. Then in the statement of Theorem~\ref{creative_telescoping_S_intro}, we can choose $\boldsymbol{\mathrm{L}}$ to have order at most
	$$2+2\dim\S_k(\Gamma)
	+\sum_{x\in S\cap(\Gamma\backslash\mathbb{H})}\mu_k(e_x)-2\delta_{k,2}.$$
	\end{corollary}
	\begin{proof}
	Let $c_k$ be the number of cusps at which the order of a weight-$k$ form is integral. Thus $c_k=\varepsilon_\infty$ for even $k$ and $c_k=\varepsilon_\infty^{\mathrm{reg}}$ for odd $k$. The dimension formulas recalled in the proof of Lemma~\ref{lemma_k_2-k_dimension} give
	$$c_k=\dim\E_k(\Gamma)+\delta_{k,2}.$$
	At every such cusp other than $\alpha$ and $\infty$, the function $t$ is a unit and $g$ vanishes. Hence
	$$\#\mu(t,g)\geq c_k-2.$$
	The result now follows by substituting the dimension formula of Theorem~\ref{dim_theorem} into the preceding proposition.
	\end{proof}
	\end{comment}
	
	\section{Creative telescoping with character}\label{sec:characters}
	\subsection{Isotypic decomposition}
	Let $\Gamma_1$ and $\Gamma_0$ be two $\SL_2(\mathbb{Z})$-commensurable subgroups of $\SL_2(\mathbb{R})$, with $\Gamma_1$ normal in $\Gamma_0$. Let $G := \Gamma_0 / \Gamma_1$ be the finite quotient group. Then $G$ acts on $\mathcal{A}_k(\Gamma_1)$ by the slash action $f\mapsto f|_k\gamma^{-1}$. Consider the isotypic decomposition of this representation of $G$:
	$$\mathcal{A}_k(\Gamma_1) = \bigoplus_{\chi \in \widehat{G}} \mathcal{A}_k(\Gamma_0,\chi).$$
	Here $\widehat{G}$ denotes the collection of irreducible representations of $G$, up to isomorphism. In particular, when $\chi$ is one-dimensional, we have
	$$\mathcal{A}_k(\Gamma_0,\chi) = \{f\in \mathcal{A}_k(\Gamma_1) : f|_k \gamma^{-1} = \chi(\gamma) f \text{ for all } \gamma\in \Gamma_0\}.$$
	We use the analogous notation $\M_k^!(\Gamma_0,\chi)$ for weakly holomorphic forms in this isotypic component.
	If $-I \in \Gamma_0$, a necessary condition for $\mathcal{A}_k(\Gamma_0,\chi)$ to be nontrivial is $\chi(-I) = (-1)^k$. Here are some familiar examples:
	\begin{itemize}[leftmargin=*]
	\item If $\Gamma_1 = \Gamma_1(N)$ and $\Gamma_0 = \Gamma_0(N)$, then $G = (\mathbb{Z}/N\mathbb{Z})^\times$. The above isotypic decomposition is the usual nebentypus decomposition.
	\item Let $N\geq 3$, $\Gamma_1 = \Gamma_0(N)$, $\Gamma_0 = \langle \Gamma_1,W_N\rangle$ with $W_N$ the Fricke involution. Then $G = \mathbb{Z}/2\mathbb{Z}$. The above decomposition is simply the eigenspace decomposition.
	\item Let $N\geq 3$, $\Gamma_1 = \Gamma_1(N) $, $\Gamma_0 = \langle \Gamma_1,W_N\rangle$ with $W_N$ the Fricke involution. Then $G = \mathbb{Z}/4\mathbb{Z}$. Let $\chi$ be a generator of $\widehat{G}$ defined by $\chi(W_N) = e^{\pi i/2}$. Then $$\mathcal{A}_k(\Gamma_1) = \bigoplus_{j=0}^3 \mathcal{A}_k(\Gamma_0,\chi^j).$$
	When $k$ is even, only the summands with even $j$ are nonzero; when $k$ is odd, only those with odd $j$ are nonzero.
	\end{itemize}
	
	Bol's identity implies the isotypic decomposition descends to the quotient: 
	$$\frac{\M_k^!(\Gamma_1)}{D^{k-1}(\M_{2-k}^!(\Gamma_1))}
	=\bigoplus_{\chi\in\widehat G}
	\frac{\M_k^!(\Gamma_0,\chi)}{D^{k-1}(\M_{2-k}^!(\Gamma_0,\chi))}.$$
	Let $d_\chi$ denote the dimension of each summand. With a more involved argument, one can show that 
	\begin{equation}\label{d_chi_dim}d_\chi = \dim \E_k(\Gamma_0,\chi) + \dim \S_k(\Gamma_0,\chi) + \dim \S_k(\Gamma_0,\overline{\chi}).\end{equation}
	We omit the argument because we will not use this conclusion.
	
	\begin{theorem} [Creative telescoping with character] \label{creative_telescoping_char}
		With the notation above, let $t\in \M_0^!(\Gamma_0)$ have no zeros or poles in $\mathbb{H}$, let $k\geq 3$, and let $g\in \M_k^!(\Gamma_0,\chi)$. Then there exist a nonzero linear recurrence operator $\boldL$ with coefficients in $\mathbb{C}(n)$ and an element $R_n \in \M_{2-k}^!(\Gamma_0,\chi)\otimes_\mathbb{C} \mathbb{C}(n)$ such that
		$$\boldL(t^n g) = D^{k-1}(t^n R_n).$$
		We can choose $\boldL$ to have order at most $$d_\chi = \dim \frac{\M_k^!(\Gamma_0,\chi)}{D^{k-1}(\M_{2-k}^!(\Gamma_0,\chi))}.$$
	\end{theorem}
	\begin{proof}
	Put $K:=\mathbb{C}(n)$. Let $\mathcal D$ be a sufficiently large effective $G$-stable divisor on $X(\Gamma_1)$ supported on the cusps, and denote the $\chi$-isotypic component of $\mathcal L_k(\mathcal D)$ by $\mathcal L_k(\mathcal D,\chi)$. Since $t$ is invariant under $\Gamma_0$, the map
	$$
	T_n:\mathcal L_{2-k}(\mathcal D,\chi)\otimes_\mathbb{C}K
	\longrightarrow \mathcal L_k(\mathcal D,\chi)\otimes_\mathbb{C}K,
	\qquad
	h\longmapsto \frac{D^{k-1}(t^nh)}{t^n},
	$$
	is well-defined and $G$-equivariant. It is injective because $k\geq3$. Define
	$$
	\mathcal H_t(\mathcal D,\chi):=
	\frac{\mathcal L_k(\mathcal D,\chi)\otimes_\mathbb{C}K}
	{T_n\bigl(\mathcal L_{2-k}(\mathcal D,\chi)\otimes_\mathbb{C}K\bigr)}.
	$$
	For $\mathcal D$ sufficiently large, the proof of Theorem~\ref{dim_theorem}, applied to a $G$-stable divisor, implies
	$$\dim_K\mathcal H_t(\mathcal D,\chi)=\dim\mathcal L_k(\mathcal D,\chi)
	-\dim\mathcal L_{2-k}(\mathcal D,\chi)=d_\chi.$$

	Enlarge $\mathcal D$ if necessary so that $g,tg,\ldots,t^{d_\chi}g\in\mathcal L_k(\mathcal D,\chi)$. Since $\mathcal H_t(\mathcal D,\chi)$ has dimension $d_\chi$, the classes of these $d_\chi+1$ elements are $K$-linearly dependent. The same argument as in the proof of Theorem \ref{creative_telescoping_S_intro} completes the proof.
	\end{proof}
	
	\begin{example}
	Arguably the most familiar example is obtained by taking $k=4$, $\Gamma_1 = \Gamma_0(6)$, and $\Gamma_0 = \langle\Gamma_1, W_6 \rangle$. The minus space has no cusp forms and has two Eisenstein series. Here $\Gamma_1$ has four cusps, namely $\{0,\infty,\frac12,\frac13\}$, while $\Gamma_0$ has two, namely $\{\infty,\frac12\}$.
	\par 
	Choose the unique, up to scaling, $t\in \M_0^!(\Gamma_0)$ with $\div_{\Gamma_0}(t) = (\infty) - (\frac12)$. Likewise, choose the unique, up to scaling, Eisenstein series $g$ that vanishes at $\frac12$. Both $t$ and $g$ can be chosen to have $q$-coefficients in $\mathbb{Q}$. Theorem \ref{creative_telescoping_char} predicts that the telescoper for $t^{-n}g$ has order at most $2$. One can check that it has exact order $2$ and coincides with Ap\'ery's recurrence. The corresponding Ap\'ery limit is
	$$(2\pi i)^3 \int_0^{i\infty} g(\tau)\,d\tau \in \mathbb{Q}\zeta(3).$$
	In this case, the resulting approximants converge fast enough to prove that $\zeta(3)\notin \mathbb{Q}$.
	\end{example}
	
	\begin{example}
		A less familiar example with exactly the same geometric configuration is due to Beukers \cite[Theorem~4]{beukers2018some}. Take $\Gamma_1 = \Gamma_1(5)$ and $\Gamma_0 = \langle\Gamma_1, W_5 \rangle$. The minus space has no cusp forms. Here $\Gamma_1$ has four cusps, namely $\{0,\infty,\frac25,\frac12\}$, while $\Gamma_0$ has two, namely $\{\infty,\frac12\}$. \par
		Choose, up to scaling, $t\in \M_0^!(\Gamma_0)$ with $\div_{\Gamma_0}(t) = (\infty) - (\frac12)$, and choose, again up to scaling, the Eisenstein series $g$ that vanishes at $\frac12$. Both $t$ and $g$ can be chosen to have $q$-coefficients in $\mathbb{Q}(\sqrt{5})$. Theorem \ref{creative_telescoping_char} predicts that the telescoper for $t^{-n}g$, with coefficients in $\mathbb{Q}(\sqrt{5})[n]$, has order at most $2$. One can check that it has exact order $2$. The corresponding Ap\'ery limit is
		$$(2\pi i)^3 \int_0^{i\infty} g(\tau)\,d\tau \in \mathbb{Q}(\sqrt{5})\xi.$$
		In this case, the resulting approximants converge fast enough to prove that $\xi \notin \mathbb{Q}(\sqrt{5})$. Beukers computed a representative of $\xi$, namely $8\zeta(3)-5\sqrt{5} L_{-5}(3)$.
	\end{example}

		\subsection{A non-congruence subgroup}
	Recall that $\eta_6(-5,7,7,-5) \in \M_2(\Gamma_0(6))$. Its square root (chosen to take positive values on $i\mathbb{R}_+ \subset \mathbb{H}$) is a well-defined function on the upper half-plane:
	$$H(\tau) := \sqrt{\frac{\eta(2\tau)^7 \eta(3\tau)^7}{\eta(\tau)^5 \eta(6\tau)^5}} = \sqrt{(u+8v)(u+9v)} = 1+\frac{5q}{2}+\frac{27q^2}{8} + \frac{49q^3}{16}+\frac{147q^4}{128}+\cdots,$$
	where $u,v$ are as in equation \eqref{level6_uv}. We define $$\mathcal{N} := \{\gamma \in \Gamma_0(6) : H|_1 \gamma = H\}.$$
	Then $H$ is a weight $1$ modular form in $\M_1(\mathcal{N})$. %We will see later that $\mathcal{N}$ is a non-congruence subgroup. %As we don't assume familiarity of readers on non-congruence subgroup, we supply more arguments in this section than elsewhere in the article.
	\begin{lemma}
		The map $$\varepsilon:\Gamma_0(6) \to \{\pm 1\},\qquad  \gamma \mapsto \frac{H|_1 \gamma}{H}$$ is a surjective homomorphism; hence $\mathcal{N} = \ker \varepsilon$ is an index-two subgroup of $\Gamma_0(6)$. Moreover, $\mathcal{N}$ is normalized by the Fricke involution $W_6 = \frac{1}{\sqrt{6}}\SmallMatrix{0 & -1 \\ 6 & 0}$.
	\end{lemma}
	\begin{proof}
		Since $H^2$ is a weight $2$ modular form of $\Gamma_0(6)$, $\left(\frac{H|_1 \gamma}{H}\right)^2 = 1$, so $\frac{H|_1 \gamma}{H} = 1$ or $-1$. To see the map is a homomorphism, we compute (with $k=1$) $$\varepsilon(\gamma_1) = \frac{j(\gamma_1,\tau)^{-k} H(\gamma_1\tau)}{H(\tau)} = \frac{j(\gamma_1,\gamma_2\tau)^{-k} H(\gamma_1\gamma_2\tau)}{H(\gamma_2\tau)},$$
		where we replaced $\tau$ by $\gamma_2\tau$, since the ratio is constant in $\tau$. Now
		$$\begin{aligned}\varepsilon(\gamma_1\gamma_2) &= \frac{H|\gamma_1 \gamma_2}{H|\gamma_2} \frac{H|\gamma_2}{H} 
			= \frac{j(\gamma_1 \gamma_2,\tau)^{-k} H(\gamma_1\gamma_2 \tau)}{j(\gamma_2,\tau)^{-k} H(\gamma_2 \tau)} \varepsilon(\gamma_2) \\
			&= \frac{j(\gamma_1 \gamma_2,\tau)^{-k} }{j(\gamma_2,\tau)^{-k}} j(\gamma_1,\gamma_2\tau)^k \varepsilon(\gamma_1)\varepsilon(\gamma_2).
		\end{aligned}$$
		The well-known cocycle identity then proves that $\varepsilon$ is a homomorphism. The map is surjective because $\varepsilon(-I)=-1$.\par
		Finally, $H^2 \in \M_2(\Gamma_0(6))$ is an eigenform under $W = W_6$, so $H| W = \lambda H$ for some $\lambda \in \mathbb{C}^\times$. We have
		$$\varepsilon(\gamma) = \frac{H|\gamma}{H}  =  \left(\frac{H|\gamma}{H} \right)\vert W =  \frac{ H|(\gamma W)}{H|W}  = \frac{\lambda H|(W^{-1} \gamma W)}{H|W} = \frac{H|(W^{-1} \gamma W)}{H} = \varepsilon(W^{-1}\gamma W),$$
		so $\mathcal{N}$ is normalized by $W$.
	\end{proof}
	
	\begin{remark}
		One can show that a set of generators of $\mathcal{N}$ is given by
		$$\mathcal{N} = \left\langle \begin{pmatrix} 1 & 1 \\ 0 & 1\end{pmatrix}, \begin{pmatrix} 1 & -1 \\ 6 & -5\end{pmatrix}, \begin{pmatrix} 5 & -3 \\ 12 & -7\end{pmatrix}\right\rangle.$$
		Since $\Gamma_0(6) = \langle \mathcal{N},-I \rangle$, these two groups define the same compact Riemann surface. In particular, this surface has the same genus and cusps as $X_0(6)$, with cusps $\{0,\infty,\frac12,\frac13\}$. The last two cusps are irregular for $\mathcal{N}$. The group $\mathcal{N}$ is not a congruence subgroup.
	\end{remark}
	
	We now apply the setup from the beginning of this section with
	$$\Gamma_1:=\mathcal N,\qquad \Gamma_0:=\langle\mathcal N,W_6\rangle = \langle \Gamma_0(6),W_6\rangle,\qquad
t:=\frac{uv}{(u+8v)(u+9v)} \in \M_0^!(\Gamma_0).$$
Then $\Gamma_0/\Gamma_1\cong\mathbb Z/4\mathbb Z$. Let $\chi$ be the character with $\chi(W_6) = i$. For odd $k$, we have
$$\mathcal{A}_k(\mathcal{N}) = \mathcal{A}_k(\Gamma_0,\chi) \oplus \mathcal{A}_k(\Gamma_0,\overline{\chi}).$$
There is a unique cusp form of weight $3$ in $\S_3(\mathcal{N})$, which is $F := uvH$. One checks that it lies in $\S_3(\Gamma_0,\overline{\chi})$. Its weakly holomorphic partner has the other character $$F^\sharp:=\frac{uv(u^2-72v^2)}{H} \in \S_3^!(\Gamma_0,\chi).$$
\begin{example}\label{example:beukers_weight_three}
Let $g = (u^2-72v^2)H \in \M_3(\Gamma_0,\chi)$. Creative telescoping gives
\begin{equation}\label{eq:beukers_weight_three_certificate}
	\boldL(t^{-n}g)=D^2(t^{-n}R),
\end{equation}
where $$\boldL = (2n+1)^2-2(68n^2+170n+107)\textbf{S}_n
+4(n+2)^2\textbf{S}_n^2,\qquad R = \frac{4(u^2-72v^2)}{t^2H^3}\in\M_{-1}^!(\Gamma_0,\chi).$$
Using exactly the same argument\footnote{Here $g$ lies in the $\chi$-eigenspace, while $F$ lies in the $\overline{\chi}$-eigenspace. Their product is therefore an honest weight-$2$ modular form, so the residue-sum argument goes through without modification.} as in Lemma~\ref{lemma:eichler_solution}, one sees that the following two sequences are annihilated by $\boldL$:
\begin{equation}\label{eq:beukers_weight_three_sequences}
	A_n:=a_0(t^{-n}g),\qquad
	B_n:=a_0\bigl(t^{-n}g\E_2(F)\bigr),\qquad n\geq0.
\end{equation}
Their initial values are
$$(A_0,A_1)=\left(1,\frac52\right),\qquad (B_0,B_1)=(0,1).$$
Their denominator types are also evident: $4^n$ for $A_n$ and $4^n \lcmD{n}^2$ for $B_n$. Lemma~\ref{lemma:constant_term_generating_function} gives
\begin{equation}
\sum_{n\geq0}A_nt(\tau)^n=H(\tau),\qquad \sum_{n\geq0}B_nt(\tau)^n=H(\tau) \E_2(F)(\tau)
\end{equation}
for $\tau$ near $i\infty$. The singularities of these two generating functions arise from the values of $t$ at the points where $Dt = 0$, namely
$(\sqrt2-1)^4$ and $(\sqrt2+1)^4$. Their monodromy, and hence the Ap\'ery limit of the sequences $A_n$ and $B_n$, can be read off from \eqref{eichler_trans}. This gives $|B_n - \omega A_n | < (\sqrt2-1)^{4n}$, where
$$\omega = L(F,2)=(2\pi i)^2\int_0^{i\infty} \sqrt{\frac{\eta^9(\tau)\eta^9(6\tau)}{\eta^3(2\tau)\eta^3(3\tau)}}\tau \,d\tau \approx 0.400749797.$$
The numerology $4e^2 < (\sqrt{2}+1)^4$ then proves that $\omega \notin \mathbb{Q}$. This is precisely the conclusion of \cite[Theorem~3]{beukers1987irrationality}.
\end{example}
With essentially the same argument, we can prove that the $L$-value associated with the meromorphic partner is also irrational. 
	\begin{example}\label{example:beukers_weight_three_partner}
	With the notation of Example~\ref{example:beukers_weight_three}, take the weakly holomorphic partner
	$$F^\sharp:=\frac{uv(u^2-72v^2)}{H},\qquad g^\sharp:=H^3.$$
	The corresponding recurrence and certificate for $\boldL^\sharp(t^{-n}g^\sharp)=D^2(t^{-n}R^\sharp)$ are
	\begin{equation}
	\boldL^\sharp:=(2n+3)^2-2(68n^2+238n+209)\textbf{S}_n
	+4(n+2)^2\textbf{S}_n^2, \qquad R^\sharp:=\frac{4}{t^2H}.
	\end{equation}
	Two solutions are
	\begin{equation}\label{eq:beukers_partner_sequences}
	A_n^\sharp:=a_0(t^{-n}H^3),\qquad
	B_n^\sharp:=a_0\bigl(t^{-n}H^3\E_2(F^\sharp)\bigr),\qquad n\geq0,
	\end{equation}
	with initial values
	$$(A_0^\sharp,A_1^\sharp)=\left(1,\frac{39}{2}\right),\qquad
	(B_0^\sharp,B_1^\sharp)=(0,1).$$
	For
	\begin{equation}\label{eq:beukers_partner_period}
	\eta:=L(F^\sharp,2)=(2\pi i)^2\int_0^{i\infty}\tau F^\sharp(\tau)\,d\tau \approx 0.05133879023,
	\end{equation}
	the combination
	\begin{equation}\label{eq:beukers_partner_generating_function}
	\sum_{n\geq0}(B_n^\sharp-\eta A_n^\sharp)t(\tau)^n
	=\frac{H(\tau)^3}{u(\tau)^2-72v(\tau)^2}
	\bigl(\E_2(F^\sharp)(\tau)-\eta\bigr)
	\end{equation}
	is overconvergent at $(\sqrt2+1)^4$. The same numerology as in Example~\ref{example:beukers_weight_three} proves $\eta\notin\mathbb Q$.
	\end{example}

	\subsection{An example with CM cusp form}\label{CM_level7_sect}
	This subsection illustrates, through a particular example, that for a CM cusp form the denominators of the associated modulonomic sequences appear to involve only certain primes, rather than the full $\lcmD{n}^{k-1}$; see Proposition \ref{-7_CM_cusp_form}. We take $\Gamma_0 = \Gamma_0(7)$, $\Gamma_1 = \Gamma_1(7)$, and $k=3$. Write $\chi = \left(\frac{-7}{\cdot}\right)$. Let
	$$t(\tau) = \frac{\eta(7\tau)^4}{\eta(\tau)^4} = q+4q^2+14q^3+\cdots \sim (\infty)-(0).$$
	The normalized cusp form spanning $\S_3(\Gamma_0,\chi)$ and a weakly holomorphic partner are
$$\begin{aligned}g&= \eta(\tau)^3 \eta(7\tau)^3 =  q-3q^2+5q^4 + \cdots \sim (0)+(\infty),\\
g^\sharp&=g(t^{-2}+11t^{-1}+33) = q^{-1}-33q^3+88q^5+\cdots .\end{aligned}$$
Here $g^\sharp$ has a simple pole at $\infty$ and vanishes at $0$.
Assuming \eqref{d_chi_dim}, an upper bound for the order of a telescoper in $\boldL(t^{-n}g) = D^2(t^{-n} R_n)$ is
$$d_\chi = \dim \E_3(\Gamma_0,\chi) + 2\dim \S_3(\Gamma_0,\chi) = 2+2 = 4.$$
Indeed, we have
$$\begin{aligned}
\boldL={}&2401(n+1)^2+49(26n^2+78n+63)\textbf{S}_n\\
&+3(89n^2+356n+378)\textbf{S}_n^2+(26n^2+130n+167)\textbf{S}_n^3+(n+3)^2\textbf{S}_n^4,
\qquad R_n=\frac{Dt}{t^3g}.
\end{aligned}$$
Four linearly independent solutions of this homogeneous recurrence, for $n\geq0$, are
$$\begin{aligned}
	A_n&=a_0(t^{-n}g), \qquad
	B_n&=a_0(t^{-n}g\,\E_2(g)),\\
	C_n&=a_0(t^{-n}g\,\E_2(g^\sharp)), \qquad
	D_n&=a_0(t^{-n}g\,\E_2(E)),
\end{aligned}$$
where $E=(1+8t)g$ is the unique Eisenstein series in $\E_3(\Gamma_0,\chi)$ with $a_0(E)=0$ and $a_1(E)=1$. We have $A_n\in\mathbb{Z}$, while $B_n,C_n,D_n$ have denominator type $\lcmD{n}^2=e^{2n+o(n)}$. The CM property of $g$ gives sharper denominator bounds for $B_n$ and $C_n$.
\begin{proposition}\label{-7_CM_cusp_form}
There exist positive integers $d_n,d_n^\sharp$ with $d_n=e^{n+o(n)}$ and $d_n^\sharp=e^{n+o(n)}$ such that
$$d_n B_n \in \mathbb{Z},\qquad d_n^\sharp C_n \in \mathbb{Z}.$$
Moreover, up to factors of size $e^{o(n)}$, the prime divisors of $d_n$ (resp. $d_n^\sharp$) are exactly those $p\leq n$ that split (resp. remain inert) in $\mathbb{Q}(\sqrt{-7})$.
\end{proposition}

The rest of this subsection is devoted to proving this proposition. The assertion about $d_n$ is easier. Since $g$ has CM, $a_p(g)=0$ if $\chi(p)=-1$. This also follows from the theta-series representation
$$g=\frac14\sum_{(m,\ell)\in\mathbb{Z}^2}
(2m^2+2m\ell-3\ell^2)q^{m^2+m\ell+2\ell^2},$$
because $m^2+m\ell+2\ell^2$ is the norm of an algebraic integer in $\mathbb{Q}(\sqrt{-7})$, and no such norm can be an inert prime $p$. Since $g$ is a normalized Hecke eigenform, for $p\neq7$ and $r\geq1$ we have
$$a_{p^{r+1}}(g)=a_p(g)a_{p^r}(g)-\chi(p)p^2a_{p^{r-1}}(g).$$
For inert $p$, this becomes $a_{p^{r+1}}(g)=p^2a_{p^{r-1}}(g)$. Starting with $a_1(g)=1$ and $a_p(g)=0$, we obtain
$$a_{p^{2j}}(g)=p^{2j},\qquad a_{p^{2j+1}}(g)=0,\qquad j\geq0.$$
By multiplicativity, $a_m(g)=0$ whenever an inert prime occurs to an odd exponent in $m$. Therefore, in
$$\E_2(g)=\sum_{m\geq1}\frac{a_m(g)}{m^2}q^m,$$
an inert prime $p$ can occur in the denominator of the $q^m$-coefficient only if $m\geq p^2$. %For $n\geq1$, a sufficient denominator is
%$$d_n:=7^{2\lfloor\log_7n\rfloor}
%\prod_{\substack{p\leq n\\\chi(p)=1}}p^{2\lfloor\log_pn\rfloor}
%\prod_{\substack{p\leq\sqrt n\\\chi(p)=-1}}p^{2\lfloor\log_pn\rfloor},$$
%where all products are over primes. 
Thus we can take $d_n$ such that
$$\log d_n = \sum_{\substack{p\leq n \\ \chi(p)=1}} 2\log p + \sum_{\substack{p\leq \sqrt{n} \\ \chi(p)= - 1}} 2\log p + o(n).$$ 
Here the $o(n)$ comes from the ramified prime $7$ and higher prime powers. The second sum is itself $O(\sqrt n\log n)=o(n)$. The prime number theorem in arithmetic progressions now gives $\log d_n=n+o(n)$.  \par

Inert primes can occur in the denominators of $B_n$, but an inert prime $p$ first occurs at $n=p^2+1$, through the coefficient $a_{p^2}(g)/p^4=1/p^2$. Next, we consider the denominators of $C_n$, which involve the complementary set of primes: split primes are absent altogether.
\begin{lemma}\label{level_7_sharp_lemma}
$\E_2(g^\sharp)$ is integral at all split primes. More precisely, for every prime $p$ that splits in $\mathbb{Q}(\sqrt{-7})$, we have
$$\E_2(g^\sharp)\in q^{-1}+\mathbb{Z}_p[[q]].$$
Equivalently, for every split prime $p$ and all integers $a\geq1$ and $m\geq1$ with $p\nmid m$, we have
$$p^{2a}\mid a_{p^am}(g^\sharp).$$
\end{lemma}
Hence only inert primes and the ramified prime $7$ can contribute to the denominator of $C_n$. We may take $d_n^\sharp$ such that
$$\log d_n^\sharp = \sum_{\substack{p\leq n \\ \chi(p)=-1}} 2\log p + o(n) = n+o(n).$$

For $n\geq1$ coprime to $7$, recall the Hecke operator $T_n$, which maps $\M_k^!(\Gamma_0,\chi)$ to itself and satisfies
$$a_m(T_nf)=\sum_{\substack{d\mid n\\d\mid m}}\chi(d)d^{k-1}a_{mn/d^2}(f),\qquad m\in\mathbb{Z}.$$
Because the level $7$ is prime, these Hecke operators preserve holomorphy and vanishing at each cusp. If $f=q^{-1}+O(1)$, then a negative coefficient $a_{-m}(T_nf)$, with $m\geq1$, can be nonzero only when $mn/d^2=1$, forcing $m=n=d$. Hence
\begin{equation}\label{aux_24}a_{-m}(T_nf)=\delta_{m,n}\chi(n)n^{k-1},\qquad m\geq1.\end{equation}

\begin{proof}[Proof of Lemma \ref{level_7_sharp_lemma}]
\underline{Step 1.} We have $a_m(g^\sharp)=0$ if $\chi(m)=1$. By the quadratic Gauss sum $\sum_{r=1}^6\chi(r)e^{2\pi ir/7}=i\sqrt7$, this is equivalent to the identity
$$\begin{aligned}
g^\sharp(\tau)
&+\frac1{i\sqrt7}\sum_{r=1}^6\chi(r)g^\sharp\!\left(\tau+\frac r7\right)-\frac17\sum_{r=0}^6g^\sharp\!\left(\tau+\frac r7\right)=0.
\end{aligned}$$
Indeed, the coefficient of $q^m$ on the left is
$(1+\chi(m)-\mathbf{1}_{7\mid m})a_m(g^\sharp)$.
The left-hand side belongs to $\M_3^!(\Gamma_0(49),\chi)$. Thus, by the valence formula on $\Gamma_0(49)$, it suffices to show that the left-hand side is $O(q^{22})$. \\
\underline{Step 2.} For each $m\geq1$, there is a unique $F_m\in\S_3^!(\Gamma_0,\chi)\cap\mathbb{Z}((q))$ that vanishes at the cusp $0$ and satisfies
$$F_m = q^{-m} + O(q^2).$$
For uniqueness, the difference of two such forms is a cusp form that is $O(q^2)$, hence zero because $\S_3(\Gamma_0,\chi)=\mathbb{C}g$ and $a_1(g)=1$. For existence, consider $t^{-m-1}g\in q^{-m}+q^{-m+1}\mathbb{Z}[[q]]$. Subtract integer multiples of the previously constructed forms $F_{m-1},\ldots,F_1$ to remove the negative powers, then an integer multiple of $g/t=1+O(q)$ to remove the constant term, and finally an integer multiple of $g=q+O(q^2)$ to remove the coefficient of $q$. Every form used vanishes at $0$. This also constructs the initial case $F_1=g^\sharp$. \\
\underline{Step 3.} For $p$ such that $\chi(p) = 1$, we claim that $$T_{p^a}g^\sharp = p^{2a} F_{p^a}.$$
Indeed, since $g^\sharp=q^{-1}+O(q^2)$, equation \eqref{aux_24} implies that their difference is $O(q)$. Moreover,
$$a_1(T_{p^a}g^\sharp - p^{2a} F_{p^a}) = a_{p^a}(g^\sharp) = 0$$
by Step 1, since $\chi(p^a)=1$. Their difference is therefore $O(q^2)$. Both forms vanish at $0$, so their difference is a cusp form and must be zero. \\
\underline{Step 4.} Finally, taking the $m$-th coefficient, with $p\nmid m$, gives
$$a_m(T_{p^a} g^\sharp) = a_{p^a m} (g^\sharp) = p^{2a} a_m(F_{p^a}),$$
so $p^{2a}$ divides $a_{p^a m} (g^\sharp)$, as desired.
\end{proof}
	
	\section{Denominator types of modulonomic sequences}\label{sec:denominator-types}
	\subsection{$q$-expansion at various cusps}
	For a ring $R$, let $R(\!( q )\!)$ denote the ring of formal Laurent series in $q$. We will mainly consider the following subring of the Puiseux series with algebraic-integer coefficients: \begin{equation}\label{puiseux_ring}\mathcal{P} := \bigcup_{w=1}^\infty \left\{f\in \overline{\mathbb{Z}}(\!( q^{1/w} )\!)  \right\}.\end{equation}
	Let $f\in \M^!_k(\Gamma)$. We say that $f$ is \textit{integral at the cusp} $\alpha$ if, for $\gamma\in \text{SL}_2(\mathbb{Z})$ with $\gamma \infty = \alpha$, we have $f|_k \gamma \in \mathcal{P}$. Note that this is independent of the choice of $\gamma$: any other choice multiplies the coefficients by roots of unity, and $\overline{\mathbb{Z}}$ is closed under multiplication by them. We say that $f$ is \textit{integral at every cusp} if $f|_k\gamma \in \mathcal{P}$ for every $\gamma\in \text{SL}_2(\mathbb{Z})$. We say that $f$ is a \textit{quasi-unit} if there exists a positive integer $\nu$ such that both $\nu f|_k \gamma$ and $\nu (f|_k \gamma)^{-1}$ belong to $\mathcal{P}$ for every $\gamma \in \text{SL}_2(\mathbb{Z})$.
	
	Recall that an eta-quotient is a weakly holomorphic modular form of the form \begin{equation}\label{eta_quo}\prod_{m\mid N} \eta^{r_m}(m\tau),\qquad N\in \mathbb{N},\quad r_m\in \mathbb{Z}.\end{equation}
	To save space, we will abbreviate the eta-quotient displayed above by $\eta_N(r_{d_1},r_{d_2},\ldots,r_{d_l})$, where $d_1,\ldots,d_l$ are the divisors of $N$ arranged in ascending order. We already saw this notation in the proof of Lemma \ref{lemma:gamma16_reduction}. The well-known result (see \cite[Proposition~5.9.2]{cohen2017modular} or \cite{rouse2015spaces}) allows us to construct eta-quotients on $\Gamma_0(N)$ easily. 
	\begin{lemma}
	Consider the eta-quotient \eqref{eta_quo} with $k := \sum_{m\mid N} \frac{r_m}{2}\in \mathbb{Z}$. It is a weakly holomorphic modular form of weight $k$ on $\Gamma_0(N)$ with Dirichlet character $\chi$ if and only if $\sum_{m\mid N} mr_m \equiv 0 \pmod{24}$ and $\sum_{m\mid N} (N/m)r_m \equiv 0 \pmod{24}$. If only the first congruence is assumed, it is instead modular on $\Gamma_0(M)$, where $M$ is the least common multiple of $N$ and the denominator of $\sum_{m\mid N}\frac{r_m}{24m}$. In either case its character is 
	$$\chi(d) = \left( \frac{(-1)^k P}{d}\right), \qquad \text{where } P = \prod_{m\mid N} m^{r_m}.$$
	\end{lemma}

	Let $f$ be an eta-quotient. We define the \textit{content} of $f$ at the cusp $\alpha$ to be an algebraic number $c$ such that the Fourier expansion of $f/c$ at $\alpha$ is a unit in $\mathcal{P}$. This $c$ can be easily computed from the $\text{SL}_2(\mathbb{Z})$-transformation law of $\eta(\tau)$, which also shows that every eta-quotient is a quasi-unit. For $f=\prod_{m\mid N}\eta(m\tau)^{r_m}$ and $\alpha=a/c$ with $\gcd(a,c)=1$, one may take $$\operatorname{cont}_{a/c}(f)=\prod_{m\mid N}\left(\frac{\gcd(c,m)}{m}\right)^{r_m/2},$$ with $c=0$ at $\infty$. Roots of unity in the transformation formula do not affect the content. 
	\begin{example}
		Consider $$f(\tau) = \eta_6(3,-3,-9,9) = \frac{\eta^3(\tau) \eta^9(6\tau)}{\eta^3(2\tau) \eta^9(3\tau)} \in \M_0^!(\Gamma_1(6)).$$ There are four cusps for $\Gamma_1(6)$, namely $\{0,\frac12,\frac13,\infty\}$. We claim that the contents of $f$ at the cusps $\infty$ and $\frac12$ are $1$, while those at the cusps $0$ and $\frac13$ are $\frac18$. Clearly, $f$ is integral at the cusp $\infty$, so its content there is $1$. The three elements
		$$\gamma_0 = \begin{pmatrix}0 & -1 \\ 1 & 0\end{pmatrix},\quad \gamma_{1/2} = \begin{pmatrix}1 & 0\\ 2 & 1\end{pmatrix},\quad \gamma_{1/3} = \begin{pmatrix}1 & 0 \\ 3 & 1\end{pmatrix}$$
		of $\text{SL}_2(\mathbb{Z})$ map $\infty$ to the other three cusps. One explicitly computes
		$$f(\gamma_0 \tau) = \frac{\eta \left(\frac{\tau }{6}\right)^9 \eta (\tau )^3}{8 \eta
			\left(\frac{\tau }{3}\right)^9 \eta \left(\frac{\tau }{2}\right)^3},\quad f(\gamma_{1/2}\tau) = -\frac{e^{\frac{i \pi }{4}} \eta \left(\frac{2 \tau }{3}+\frac{1}{3}\right)^9
			\eta (\tau )^3}{\eta \left(\frac{\tau }{3}+\frac{2}{3}\right)^9 \eta (2
			\tau )^3},\quad f(\gamma_{1/3}\tau) = \frac{i e^{-\frac{3 i \pi }{4}} \eta (\tau )^3 \eta \left(\frac{3 \tau
			}{2}+\frac{1}{2}\right)^9}{8 \eta \left(\frac{\tau
			}{2}+\frac{1}{2}\right)^3 \eta (3 \tau )^9},$$
		which proves the claim. In fact, one readily sees that the entries of Table \ref{Gamma0(6)eta_content} encountered previously are exactly the contents of various weight $0$ eta-quotients on $\Gamma_1(6)$. 
	\end{example}

	\subsection{Reduction systems and denominator types of modulonomic sequences}
	We now formulate the reduction procedure used in Section~\ref{sec:gamma1-6-denominators} in a general setting. %Here the situation is a bit simpler than that of $\Gamma_0(4)$ because the factor $16^M$ no longer appears.

For any integral domain $R$ and $f\in R$, we let $R_{(f)}$ be the localization of $R$ with respect to the multiplicative set $\{1,f,f^2,\ldots\}$. Recall the ring $\mathcal{P}$ defined in \eqref{puiseux_ring}. In $\mathrm{pp}_{>\mathcal{D}}(X)$, we now use the Fourier expansions obtained from the $\SL_2(\mathbb{Z})$ scaling matrices at the cusps. We say that a form is $N$-integral at a cusp if its Fourier expansion there belongs to $\mathcal{P}_{(N)}$.
\begin{definition}
A set $\mathcal{R} = \{f_1,\ldots,f_r\} \subset \M_0^!(\Gamma)$ is called an $N$-\textit{reduction system} with respect to a subset of cusps $S \subset \Gamma\backslash\mathbb{P}^1(\mathbb{Q})$ if
\begin{itemize}[leftmargin=*]
\item the $q$-expansions of each $f_i$ at all cusps are units in $\mathcal{P}_{(N)}$;
\item each $f_i$ has no pole outside $S$;
\item there exists $M_0 \in \mathbb{N}_{\geq 0}$ such that, whenever $M\geq M_0$ and $s\in S$, there exists $f$ in the monoid generated by $\mathcal{R}$ such that the $s$-component of $\div f$ is $-M$; for every $s'\in S\setminus\{s\}$, the $s'$-component of $\div f$ is at least $-M_0$.
\end{itemize}
\end{definition}

We give some examples to illustrate this technical definition.
\begin{example}\label{example:reduction_systems}
(1) The first two eta-quotients in Table \ref{Gamma0(6)eta_content}, namely $$\mathcal{R} := \{f_1,f_2\} := \{\eta_6(3,-3,-9,9),\eta_6(9,-9,-3,3)\} \subset \M_0^!(\Gamma_0(6)),$$ form a $2$-reduction system with respect to $S = \{\frac13,\frac12\}$. Indeed, the first bullet of the definition is immediate because the contents at all cusps are units in $\mathbb{Z}_{(2)}$. For the third bullet, we can take $M_0 = 0$. If $s=\frac13$, we can take $f=f_1^M$; if $s=\frac12$, we can take $f=f_2^M$. Similarly, the last two eta-quotients in Table \ref{Gamma0(6)eta_content} form a $3$-reduction system with respect to the same $S$. \par
(2) We claim that the first three eta-quotients in Table \ref{Gamma0(15)eta}, namely,
$$\mathcal{R} := \{f_1,f_2,f_3\}  := \{\eta_{15}(9,-3,-9,3),\eta_{15}(3,-9,-3,9),\eta_{15}(3,-3,-3,3)\}\subset \M_0^!(\Gamma_0(15)),$$
form a $5$-reduction system with respect to $S = \{\frac13,\frac15\}$. The first bullet of the definition is immediate. For the third bullet, we can take $M_0 = 4$. Assume that $s=\frac15$ and write $M = 4a+b$ with $0\leq b <4$. Then $f = f_1^a f_3^b$ satisfies $\div f = -M(\frac15)-b(\frac13) + \cdots$, so this $f$ satisfies the condition. For $s=\frac13$, one takes $f=f_2^a f_3^b$. Analogously, the last three eta-quotients in Table \ref{Gamma0(15)eta} form a $3$-reduction system with respect to the same $S$.
\begin{table}[h]
	\begin{tabular}{|c|c|c|c|c|}
		\hline
		$\text{}$ & $0$ & $\frac15$       & $\frac13$ & $\infty$ \\ \hline
		$\eta_{15}(9,-3,-9,3) \sim 4(0)-4(\frac15)$ & 125      & 1 & 125       & 1 \\ \hline
		$\eta_{15}(3,-9,-3,9) \sim 4(\infty)-4(\frac13)$ & $\frac{1}{125}$      & 1 & $\frac{1}{125}$       & 1 \\ \hline
		$\eta_{15}(3,-3,-3,3) \sim (0)+(\infty)-(\frac13)-(\frac15)$ & 1      & 1 & 1       & 1 \\ \hline
		$\eta_{15}(1,-1,-5,5) \sim 2(\infty)-2(\frac15)$ & $\frac19$      & $\frac19$ & 1       & 1 \\ \hline
		$\eta_{15}(5,-5,-1,1) \sim 2(0)-2(\frac13)$ & 9      & 9 & 1       & 1 \\ \hline
	\end{tabular}
	\caption{Contents of some weight $0$ eta-quotients of $\Gamma_0(15)$ at each cusp.}
	\label{Gamma0(15)eta}
\end{table}

(3) The group $\Gamma_0(30)$ has eight cusps. Let $S$ be the set of cusps other than $0$ and $\infty$. There exists an $N$-reduction system for each $N\in\{6,10,15\}$. This follows by inspecting Table \ref{tab:gamma0_30_reduction_eta}: the rows whose last-column entries divide $N$ form the desired reduction system. Here the third bullet is satisfied because, for every $s\in S$, the gcd of the $s$-coefficients of the divisors in each system is $1$.
\begin{table}[h]
	\centering
	\small
	\setlength{\tabcolsep}{4pt}
	\begin{tabular}{|l|c|c|}
		\hline
		Exponent vector $(r_d)_{d\mid 30}$ & $\div f$ & $C$ \\ \hline
		$(9,-5,-3,-1,-1,-3,-5,9)$ & $\begin{aligned}&6(0)+6(\infty)-2(\frac12)-2(\frac13)\\&-2(\frac15)-2(\frac16)-2(\frac1{10})-2(\frac1{15})\end{aligned}$ & $1$ \\ \hline
		$(14,-7,-2,2,1,-1,-14,7)$ & $12(0)-12(\frac1{15})$ & $15$ \\ \hline
		$(9,-3,-3,3,1,-9,-1,3)$ & $8(0)-8(\frac1{10})$ & $10$ \\ \hline
		$(25,-5,5,-5,-25,1,-1,5)$ & $24(0)-24(\frac16)$ & $6$ \\ \hline
		$(6,-3,-2,-6,1,3,2,-1)$ & $4(0)-4(\frac15)$ & $5$ \\ \hline
		$(20,-10,-20,-4,10,2,4,-2)$ & $12(0)-12(\frac13)$ & $3$ \\ \hline
		$(45,-45,-15,-9,15,9,3,-3)$ & $24(0)-24(\frac12)$ & $2$ \\ \hline
		$(-3,3,9,15,-9,-15,-45,45)$ & $24(\infty)-24(\frac1{15})$ & $2$ \\ \hline
		$(-1,2,1,5,-2,-10,-5,10)$ & $6(\infty)-6(\frac1{10})$ & $3$ \\ \hline
		$(-1,2,3,1,-6,-2,-3,6)$ & $4(\infty)-4(\frac16)$ & $5$ \\ \hline
		$(5,-1,1,-25,-5,5,-5,25)$ & $24(\infty)-24(\frac15)$ & $6$ \\ \hline
		$(3,-1,-9,1,3,-3,-3,9)$ & $8(\infty)-8(\frac13)$ & $10$ \\ \hline
		$(7,-14,-1,1,2,-2,-7,14)$ & $12(\infty)-12(\frac12)$ & $15$ \\ \hline
		$(8,-6,-2,0,0,-2,-6,8)$ & $\begin{aligned}&5(0)+5(\infty)-3(\frac12)-(\frac13)\\&-(\frac15)-(\frac16)-(\frac1{10})-3(\frac1{15})\end{aligned}$ & $1$ \\ \hline
	\end{tabular}
	\caption{Some weight $0$ eta-quotients on $\Gamma_0(30)$, given by $f=\eta_{30}(r_d)$. The last column gives a positive integer $C$ for which the Fourier expansions at every cusp are units in $\mathcal{P}_{(C)}$.}
	\label{tab:gamma0_30_reduction_eta}
\end{table}
\end{example}

The significance of the concept of a reduction system is given in the following lemma, which generalizes Lemma \ref{lemma:gamma16_reduction}. Recall the notation $\mathrm{pp}_{>\mathcal{D}}(X)$ defined there.
\begin{lemma}[Reduction lemma]\label{reduction_lemma}
	Let $\mathcal{R} \subset \M_0^!(\Gamma)$ be an $N$-reduction system with respect to the cusps in $S$. Then there exists a nonnegative cuspidal divisor $\mathcal{D}_0$ supported on $S$ such that, for every $X\in\M_0^!(\Gamma)$ whose poles are supported on $S$ and for which $\mathrm{pp}_{>\mathcal{D}_0}(X)$ has coefficients in $\overline{\mathbb{Z}}_{(N)}$, we can write $X = G + \kappa$, where $G$ is $N$-integral at all cusps and $\kappa\in \mathcal{L}_0(\mathcal{D}_0)$.
\end{lemma}
\begin{proof}
	We take $\mathcal{D}_0 = \sum_{s\in S} M_0(s)$, where $M_0$ is as in the definition of a reduction system. Write $$\div X = -\sum_{s\in S} x_s (s) + \cdots,\qquad \delta(X) := \sum_{s\in S} \max(x_s-M_0,0).$$
	We proceed by induction on $\delta(X)$. The base case is $\delta(X) = 0$. In this case, $x_s\leq M_0$ for every $s\in S$, and hence $X\in \mathcal{L}_0(\mathcal{D}_0)$, so we can take $(G,\kappa) = (0,X)$.\par
	We therefore assume that $x_s > M_0$ for some $s$. By the definition of a reduction system, we can find $f$ in the monoid generated by $\mathcal{R}$ such that $$\div f = -x_s(s) + \sum_{s' \in S\setminus\{s\}} c_{s'} (s') + \cdots, \qquad c_{s'}\geq -M_0.$$
	Moreover, the Fourier expansions of $f$ at all cusps have coefficients in $\overline{\mathbb{Z}}_{(N)}$. There exists a unique constant $a$ such that the order of the pole of $X-af$ at the cusp $s$ is less than that of $X$. By comparing the leading coefficients, we see that $a\in \overline{\mathbb{Z}}_{(N)}$. Write $\div (X-af) = -\sum_{s\in S} y_s(s) + \cdots$. We just saw that $y_s < x_s$. For $s'\in S\setminus\{s\}$, we also have $\max(x_{s'}-M_0,0) = \max(y_{s'}-M_0,0)$. Hence $\delta(X-af) < \delta(X)$. By the induction hypothesis, we can write $X-af = G' + \kappa'$, where $G'$ is $N$-integral at all cusps and $\kappa'\in \mathcal{L}_0(\mathcal{D}_0)$. Setting $G = G'+af$ and $\kappa = \kappa'$ completes the proof.
\end{proof}

The next result generalizes that of Proposition \ref{prop:gamma16_denom_type}. Let $\mathfrak{B}$ be a basis of $\M_k^!(\Gamma) \cap \overline{\mathbb{Q}}((q))$ modulo $D^{k-1}(\M_{2-k}^!(\Gamma))$. Let $F \in \M_k^!(\Gamma)$ with poles supported on $S$. Write
\begin{equation}\label{F_basis_expansion}F = \sum_{f\in \mathfrak{B}} c_f f + D^{k-1}(h),\qquad h\in \M_{2-k}^!(\Gamma).\end{equation}

	\begin{proposition}\label{N_reduction_integrality}
	Let $k\geq 3$ and $F\in \M_k^!(\Gamma)$, and write $$(\div F)_\infty = \sum_{s\in S} M_s(s), \qquad M:= \max_s\{M_s\}.$$
	Assume that \begin{enumerate}[leftmargin=*]
	\item $F$ is integral at all cusps and all cusps in $S$ are regular;
	\item there exists a quasi-unit $g\in\M_{k-2}^!(\Gamma)$;
	\item for some $N$, the group $\Gamma$ has an $N$-reduction system with respect to $S$.
	\end{enumerate} 
Then, for $c_f$ and $h$ in \eqref{F_basis_expansion}, there exists an integer $\nu$, depending only on $k$, $\Gamma$, and $\mathfrak{B}$ (i.e., independent of $F$), such that \begin{itemize}[leftmargin = *]
		\item $\left(\lcmD{M}^{k-1} \nu\right)c_f\in \overline{\mathbb{Z}}_{(N)}$;
		\item $\left(\lcmD{M}^{k-1} \nu\right)h$ is $N$-integral at all cusps.
\end{itemize}\end{proposition}
\begin{proof}
The proof is largely parallel to the three steps in the proof of Proposition \ref{prop:gamma16_denom_type}. We may enlarge each $M_s$ to a fixed sufficiently large bound, at the cost of changing $\nu$ by a fixed factor. We will use $\nu_1,\nu_2,\ldots \in \mathbb{Z}_{>0}$ to denote integers that depend only on $k$, $\Gamma$, and $\mathfrak{B}$ (i.e., are independent of $F$).\\
\underline{Step 1.} Fix an effective cuspidal divisor $\mathcal{D}$ such that $\mathfrak{B}\subset \mathcal{L}_{k}(\mathcal{D})$. Let $\gamma \in \text{SL}_2(\mathbb{Z})$, set $s=\gamma\infty$, and write \begin{equation}\label{aux_5}h|_{2-k}\gamma := a_0 + \sum_{\substack{n\geq -M_s\\ n\neq 0}} \frac{a_n}{(n/w)^{k-1}} q^{n/w},\end{equation}
where $w$ is the width of the cusp $s$. Then $$ \sum_{\substack{n\geq -M_s\\ n\neq 0}} a_n q^{n/w} = F|_k \gamma -\sum_{f\in \mathfrak{B}} c_f f|_k \gamma.$$
Since $F$ is integral at all cusps, we have $a_n \in \overline{\mathbb{Z}}$ for every negative $n$ for which $q^{n/w}$ does not appear in any $f|_k \gamma$ with $f\in\mathfrak{B}$. Using $\lcmD{M}^{k-1}$ to cancel the denominators in \eqref{aux_5} and absorbing the cusp widths into a constant $\nu_1$, we see that $\mathrm{pp}_{>\mathcal{D}}(\nu_1\lcmD{M}^{k-1}h)$ has coefficients in $\overline{\mathbb{Z}}$. \\
\underline{Step 2.} By hypothesis (2), we can choose a quasi-unit $g\in \M_{k-2}^!(\Gamma)$. After enlarging $\nu_1$ by a fixed factor, $\mathrm{pp}_{>\mathcal{D}_1}(\nu_1\lcmD{M}^{k-1}gh)$ has coefficients in $\overline{\mathbb{Z}}$ for some effective cuspidal divisor $\mathcal{D}_1$. Lemma \ref{reduction_lemma} then gives $G$, which is $N$-integral at all cusps, and $\kappa\in \mathcal{L}_0(\mathcal{D}_1)$ such that $$\nu_1\lcmD{M}^{k-1}g h = G + \kappa.$$
\underline{Step 3.} Multiply both sides of the above equation by $g^{-1}$, apply $D^{k-1}$, and then use \eqref{F_basis_expansion} to obtain
$$\begin{aligned}\nu_1\lcmD{M}^{k-1} F - D^{k-1}(Gg^{-1}) &= \sum_{f\in \mathfrak{B}} \nu_1\lcmD{M}^{k-1} c_f f + D^{k-1}(\kappa g^{-1}) \\ 
&= \sum_{f\in \mathfrak{B}} \nu_1\lcmD{M}^{k-1} c_f f + \sum_{j=1}^r \theta_j D^{k-1}(\kappa_j g^{-1}),
\end{aligned}$$ 
where, in the last line, we chose $\kappa_1,\ldots,\kappa_r$ as a basis of $\mathcal{L}_0(\mathcal{D}_1)$ and wrote $\kappa = \sum_j \theta_j \kappa_j$. Since $F$ and $G$ are $N$-integral at all cusps, the left-hand side has Fourier coefficients in $\nu_2^{-1}\overline{\mathbb{Z}}_{(N)}$. By definition of $\mathfrak{B}$ and since $k\geq 3$, the forms $\{f\mid f\in \mathfrak{B}\}$ and $\{D^{k-1}(\kappa_1g^{-1}),\ldots,D^{k-1}(\kappa_rg^{-1})\}$ are linearly independent. Lemma \ref{linear_comb_int_lemma} shows that there exists $\nu_3\in \mathbb{Z}_{>0}$ such that $$\nu_1\lcmD{M}^{k-1} c_f \in \overline{\mathbb{Z}}_{(N)}/\nu_3, \qquad \theta_j \in  \overline{\mathbb{Z}}_{(N)}/\nu_3.$$
This proves the first bullet after absorbing $\nu_1\nu_3$ into $\nu$. Also, writing $\nu_1\lcmD{M}^{k-1} h = Gg^{-1} + \kappa g^{-1}$ and recalling that $g$ is a quasi-unit, we see that $\nu_4\nu_1\lcmD{M}^{k-1} h$ is $N$-integral at all cusps.
\end{proof}

\begin{remark}
The above proof also applies when $k=2$. The only distinction is that $h$ in \eqref{F_basis_expansion} is then unique only up to the addition of a constant. Thus, the second bullet of the proposition should be replaced by the following statement: $\left(\lcmD{M}^{k-1} \nu\right)h$ is $N$-integral at all cusps up to an additive constant.
\end{remark}

\begin{example}
Let $\Gamma = \Gamma_0(15)$ and $$t(\tau) = \eta_{15}(3,-3,-3,3) = \left(\frac{\eta(\tau)\eta(15\tau)}{\eta(3\tau)\eta(5\tau)}\right)^{3} \sim (0) + (\infty)-(\frac13)-(\frac15).$$ Then, for $k\geq 3$ and any $g\in \M_k(\Gamma)\cap\overline{\mathbb{Q}}((q))$, write
$$t^ng = \sum_{f\in \mathfrak{B}} c_{f,n} f + D^{k-1}(h_n), \qquad h_n\in \M_{2-k}^!(\Gamma).$$
The sequences $\{c_{f,n}\}$ have denominator type $\lcmD{n}^{k-1}$. Also, for each $\gamma \in \text{SL}_2(\mathbb{Z})$, the sequence $\{a_0(h_n|_{2-k} \gamma)\}$ has denominator type $\lcmD{n}^{k-1}$. This follows from Example~\ref{example:reduction_systems}, where we constructed $N$-reduction systems for two coprime integers, namely $N=3$ and $N=5$.
\end{example}

\begin{example}
Consider the case $k=2$ and $\Gamma = \Gamma_0(15)$, with the same $t$ as in the preceding example. The curve has genus one, and its unique cusp form is $f = \eta(\tau)\eta(3\tau)\eta(5\tau)\eta(15\tau)$.
Let
$$\begin{aligned}
E &= \frac{-E_2(\tau)+21E_2(3\tau)-35E_2(5\tau)+15E_2(15\tau)}{24}
= (1+4t)f \in \E_2(\Gamma),\\
f^\sharp &= t(t+5)f \in \S_2^!(\Gamma).
\end{aligned}$$
Note that $E$ is the unique Eisenstein series vanishing at $0$ and $\infty$. Therefore, for $n\geq 0$, we have
\begin{equation}\label{aux_22}t^n f = A_n E + B_n f + C_n f^\sharp + D(h_n), \qquad h_n\in \M_0^!(\Gamma).\end{equation}
All elements are in the minus space of the Fricke involution, so $h_n$ is uniquely determined\footnote{Constant functions have Fricke eigenvalue $+1$, so the minus space contains no nonzero constants.}. A familiar argument shows that $A_n$ has trivial denominator type. By Proposition \ref{N_reduction_integrality}, $B_n$, $C_n$, and $a_0(h_n)$ have denominator type $\lcmD{n}$. By integrating \eqref{aux_22} from $0$ to $i\infty$, we obtain
$$(2\pi i)\int_0^{i\infty} t^n f\,d\tau = A_n \log\left(\frac{81}{125}\right) - B_n L(f,1) - C_n L(f^\sharp,1) + 2a_0(h_n).$$
Here we use the Mellin normalization $L(F,1)=2\pi\int_0^\infty F(iy)\,dy$ for $F=f,f^\sharp$; both integrals converge since these forms vanish at $0$ and $\infty$.
Since $t(\frac{i}{\sqrt{15}}) = \frac{-11+5\sqrt{5}}{2}$, at least one of the above numbers is irrational. This conclusion is rather uninteresting, since the first one is already known to be irrational.
%The telescoper and certificate of $\boldL(t^n f) = D(t^n R)$ is
%$$\boldL = (n+1)-5(2n+3)\mathbf S_n-13(n+2)\mathbf S_n^2+5(2n+5)\mathbf S_n^3+(n+3)\mathbf S_n^4,\qquad R = \frac{tDt}{f}.$$ 
\end{example}

\begin{example}
Let $\Gamma = \Gamma_0(30)$ and $k=6$. It has genus $3$ and eight cusps
$$\left\{0,\infty,\frac12,\frac13,\frac15,\frac16,\frac1{10},\frac1{15}\right\}.$$
Pick any $g\in \M_6(\Gamma)$. Let
$$t(\tau) = \frac{\eta (\tau )^9 \eta (30 \tau )^9}{\eta (2 \tau )^5 \eta (3 \tau )^3 \eta (5 \tau ) \eta (6 \tau ) \eta (10 \tau )^3 \eta (15 \tau )^5} \sim 6(0)+6(\infty)-2(\frac12)-2(\frac13)-2(\frac15)-2(\frac16)-2(\frac1{10})-2(\frac1{15}).$$
Choose a $\mathbb{Q}$-basis $\mathfrak{B}$ of the subspace of $\M_6^!(\Gamma)$ modulo $D^5(\M_{-4}^!(\Gamma))$ represented by forms vanishing at $0$ and $\infty$. In the expansion
$$t^ng = \sum_{f\in \mathfrak{B}} c_{f,n} f + D^5(h_n), \qquad h_n\in \M_{-4}^!(\Gamma),\quad n\geq 1,$$
there are a total of $\dim \E_6 + 2\dim \S_6 -2 = 58$ forms $f$. The above proposition implies that each $c_{f,n}$ has denominator type $\lcmD{2n}^5$. For our chosen $t$, the integral $(2\pi i)^5 \int_0^{i\infty} t^n f\,d\tau$ decays very rapidly, since
$$t(\frac{i}{\sqrt{30}}) = \left(3 \sqrt{\frac{5}{2}}-\frac{\sqrt{5}}{2}-\frac{3}{2}-\frac{3}{\sqrt{2}}\right)^2 \approx 1.65 \times 10^{-5}.$$
The inequality $e^{10} t(\frac{i}{\sqrt{30}}) < 1$ implies that at least one of 
\begin{equation}\label{irrational_list_30}\left\{(2\pi i)^5 \int_0^{i\infty} f(\tau)\,d\tau : f\in \mathfrak{B}\right\}\end{equation}
is irrational. In particular, $\zeta(5)$ belongs to this list by considering an Eisenstein series. Although the above list contains $58$ real numbers, the number of $\mathbb{Q}$-linearly independent numbers is much smaller because most of the forms in $\mathfrak{B}$ are oldforms. For example, if $F\in \S_6(\Gamma_0(3))$, the elements $\{F(m\tau) : m\mid 10\}$ give the same value of $(2\pi i)^5 \int_0^{i\infty} F(m\tau)\,d\tau$ up to rational multiples. A further reduction of \eqref{irrational_list_30} can be obtained by considering the $W_{30}$-eigenspaces. We do not pursue this point further here.
\end{example}

\bibliographystyle{plain} % We choose the "plain" reference style
\bibliography{../ref.bib} % Entries are in the refs.bib file
	
\end{document}